\documentclass[11pt,a4paper]{article}

\usepackage{amsmath,amssymb,amsthm}
\usepackage{mathtools}
\usepackage{graphicx}
\usepackage{enumerate}
\usepackage{soul}
\usepackage{xcolor}
\usepackage[hidelinks]{hyperref}
\usepackage[numbers]{natbib}
\usepackage{tikz,tikz-cd,quiver}
\usepackage{algorithm, algorithmic}
\usepackage{lipsum}
\usepackage{subcaption}
\usepackage{placeins}
\usepackage[margin=1in]{geometry}
\usepackage{mathrsfs}
\usepackage{soul}
\usetikzlibrary{shapes}
\usetikzlibrary{positioning}
\usetikzlibrary{decorations.markings}

\numberwithin{equation}{section}

\newcommand{\R}{\mathbb{R}}
\newcommand{\Z}{\mathbb{Z}}
\newcommand{\F}{\mathbb{F}}
\newcommand{\Rp}{\R_{\geq 0}}
\newcommand{\Zp}{\Z_{\geq 0}}

\newcommand{\mvmap}{\rightrightarrows}
\newcommand{\surjection}{\twoheadrightarrow}
\newcommand{\injection}{\hookrightarrow}
\newcommand{\setof}[1]{\left\{ {#1}\right\}}
\newcommand{\setdef}[2]{\left\{ #1 \, \middle| \, #2 \right\}}
\newcommand{\cA}{{\mathcal A}}
\newcommand{\cB}{{\mathcal B}}
\newcommand{\cC}{{\mathcal C}}
\newcommand{\cR}{{\mathcal R}}
\newcommand{\cH}{{\mathcal H}}
\newcommand{\cU}{{\mathcal U}}
\newcommand{\cV}{{\mathcal V}}
\newcommand{\cX}{{\mathcal X}}
\newcommand{\cF}{{\mathcal F}}
\newcommand{\cG}{{\mathcal G}}
\newcommand{\cI}{{\mathcal I}}
\newcommand{\Ink}{I_{n,k}}
\newcommand{\cP}{{\mathcal P}}
\newcommand{\cD}{{\mathcal D}}
\newcommand{\stoptime}{\sigma}
\newcommand{\incl}{\tilde{\iota}}
\newcommand{\susp}{\Sigma}
\newcommand{\rhosusp}{\rho_{\susp}}
\newcommand{\Wab}{W_{a,b}}
\newcommand{\inv}{^{-1}}

\newcommand{\evarphi}{\bar{\varphi}}
\newcommand{\tmu}{\widetilde{\mu}}
\newcommand{\tA}{\widetilde{A}}
\newcommand{\tB}{\widetilde{B}}
\newcommand{\tC}{\widetilde{C}}
\newcommand{\tD}{\widetilde{D}}
\newcommand{\tE}{\widetilde{E}}
\newcommand{\tF}{\widetilde{F}}
\newcommand{\tK}{\widetilde{K}}
\newcommand{\tL}{\widetilde{L}}
\newcommand{\tM}{\widetilde{M}}
\newcommand{\tN}{\widetilde{N}}

\newcommand{\tQ}{\widetilde{Q}}

\newcommand{\tS}{\widetilde{S}}

\newcommand{\tU}{\widetilde{U}}
\newcommand{\tV}{\widetilde{V}}
\newcommand{\tW}{\widetilde{W}}
\newcommand{\tcA}{\widetilde{\cA}}
\newcommand{\tcB}{\widetilde{\cB}}
\newcommand{\tcD}{\widetilde{\cD}}
\newcommand{\tell}{\widetilde{\ell}}
\newcommand{\htN}{\widehat{N}}
\newcommand{\htL}{\widehat{L}}

\newcommand{\hdot}{\dot{h}}
\newcommand{\vdot}{\dot{v}}
\newcommand{\thetadot}{\dot{\theta}}

\newcommand{\udot}{\dot{u}}

\newcommand{\CH}{\mathrm{CH}}
\newcommand{\SCH}{\susp\CH}
\newcommand{\id}{\mathrm{id}}

\newcommand{\scR}{\mathscr{R}}

\newcommand{\sInvsetp}{\mathsf{Invset}^{+}}
\newcommand{\sANbhd}{\mathsf{ANbhd}}
\newcommand{\sANbhdR}{\mathsf{ANbhd}_{\scR}}
\newcommand{\sAtt}{\mathsf{Att}}

\DeclareMathOperator{\cl}{cl}
\DeclareMathOperator{\Int}{int}
\DeclareMathOperator{\diam}{diam}
\DeclareMathOperator{\dom}{dom}

\DeclareMathOperator{\MG}{MG}
\newcommand{\Inv}{\operatorname{Inv}}
\newcommand{\im}{\operatorname{im}}

\theoremstyle{plain}
\newtheorem{theorem}{Theorem}[section]
\newtheorem{lemma}[theorem]{Lemma}
\newtheorem{proposition}[theorem]{Proposition}
\newtheorem{corollary}[theorem]{Corollary}
\theoremstyle{definition}
\newtheorem{definition}[theorem]{Definition}

\newtheorem{example}[theorem]{Example}
\theoremstyle{remark}
\newtheorem{remark}[theorem]{Remark}

\title{Computation of Attractor Lattices for a Class of Hybrid Systems}
\author{William Kalies\thanks{Department of Mathematics and Statistics, University of Toledo, Toledo, OH 43606, USA.} \and Bernardo Rivas\footnotemark[1] \and Tony Wehbe\footnotemark[1]}

\date{\today}

\begin{document}
\maketitle

\begin{abstract}
We study the attractor lattice of a class of hybrid dynamical systems on compact metric spaces. Under the trapping guard condition, we prove that hybrid attractors form a bounded distributive lattice isomorphic to the attractor lattice of an associated semiflow. The regularized preimages define a lattice epimorphism from a sublattice of the regular closed, attracting neighborhoods of the semiflow onto those of the hybrid system, and from the whole lattice when the guard set and its reset image have empty interior. Then, using grids on the state space, we construct a convergent cofiltration of outer approximations of a fixed-time map of this semiflow. This construction realizes every finite sublattice of hybrid attractors by a lattice of attracting neighborhoods that are unions of grid elements at sufficiently fine resolutions, and provides index pairs for computing Conley indices. We illustrate these constructions computationally by extracting the Conley--Morse graphs of the bouncing ball, the rimless wheel, the spiking neuron, and exhibiting the attractor lattice for an impacting oscillator.
\end{abstract}

\section{Introduction}
\label{sec:introduction}

This paper extends the lattice theory of attractors from continuous maps and flows to a class of hybrid dynamical systems, which model dynamics with impacts, switching, and control~\cite{diBernardo:Budd:Champneys:Kowalczyk,Goebel:Sanfelice:Teel,vanDerSchaft:Schumacher}. For a continuous map or flow on a compact metric space, the attracting neighborhoods and the attractors form bounded distributive lattices, and taking $\omega$-limit sets is a lattice epimorphism from the attracting neighborhoods onto the attractors~\cite{Kalies:Mischaikow:Vandervorst:2014,Kalies:Mischaikow:Vandervorst:2016}. Every Morse representation can be obtained from a finite sublattice of attractors~\cite{Kalies:Mischaikow:Vandervorst:2022}, and for a continuous map, every finite sublattice of attractors is realized by forward invariant families of a convergent cofiltration of outer approximations at all sufficiently fine resolutions~\cite{Kalies:Mischaikow:Vandervorst:2016}. Since resets can make the state at a fixed time depend discontinuously on the initial state, this theory does not apply directly to hybrid systems.

In this manuscript, a hybrid system $\cH=(X,\varphi,G,r)$ consists of a compact metric space $X$, a closed guard set $G\subseteq X$, a local semiflow $\varphi$ on $X\setminus G$, and a continuous reset map $r:G\to X$. A trajectory follows $\varphi$ until it reaches $G$, where $r$ determines the next state. To obtain a continuous semiflow, we use the hybrid suspension~\cite{Kvalheim:Gustafson:Koditschek}, the space $\susp X$ obtained via the mapping torus of the reset map when $\cH$ satisfies the trapping guard condition. This yields a suspension semiflow $\Phi$ that follows $\varphi$ on the base space $\incl(X)$ and is gradient along the extension, as in the regularization of hybrid systems with Zeno trajectories~\cite{Johansson:Egerstedt:Lygeros:Sastry}. The resulting dynamics, $\Phi$, is a continuous semiflow on the compact metrizable space $\susp X$, the $\omega$-limit sets of $\cH$ (Definition~\ref{def:hybrid-omega}) satisfy $\omega(A,\cH)=\incl^{-1}(\omega(\incl(A),\Phi))$ for every $A\subseteq X$, and the attracting--repelling pairs of $\cH$, defined through trapping regions, are the preimages under $\incl$ of those of $\Phi$~\cite{Kvalheim:Gustafson:Koditschek}. Moreover, the Conley index of $\Phi$ also defines the suspension Conley index of isolated invariant sets of $\cH$~\cite{Rivas:Kalies}.

In this setting, we compare the lattices $\sAtt$ of attractors and $\sANbhdR$ of regular closed attracting neighborhoods of $\cH$ and $\Phi$, defined in Sections~\ref{sec:preliminaries} and~\ref{sec:hybrid-lattices}, through $\incl^{-1}$ and the regularized preimage $\rho(\tE):=\cl_X\Int_X(\incl^{-1}(\tE))$ of $\tE\subseteq\susp X$. Following the notation of Section~\ref{sec:preliminaries}, our main result is the following.
\begin{theorem}
  \label{thm:main}
  Let $\cH=(X,\varphi,G,r)$ be a hybrid system on a compact metric space satisfying the trapping guard condition, and let $\Phi$ be its suspension semiflow.
  \begin{enumerate}[(i)]
    \item $\incl\inv$ is a bounded lattice isomorphism from $\sAtt(\susp X,\Phi)$ to $\sAtt(\cH)$, with inverse $A\mapsto\susp(A)$.
    \item There is a sublattice $\tcD\subseteq\sANbhdR(\susp X,\Phi)$ such that $\rho:\tcD\surjection\sANbhdR(X,\cH)$ is a lattice epimorphism. For every $\tU\in\sANbhdR(\susp X,\Phi)$,
    \[
      \omega(\rho(\tU),\cH)=\incl^{-1}\left(\omega(\tU,\Phi)\right).
    \]
    If $\Int_X(G\cup r(G))=\emptyset$, one may take $\tcD=\sANbhdR(\susp X,\Phi)$.
    \item Let $(\cX_n)$ be a contracting cofiltration of grids on $X$. There exist multivalued maps $\cF_n^X:\cX_n\mvmap\cX_n$ such that every finite sublattice $\cA\subseteq\sAtt(\cH)$ admits, for all sufficiently large $n$, a lattice monomorphism
    \[
      \ell_n:\cA\injection\sInvsetp(\cX_n,\cF_n^X)
    \]
    whose evaluations form a sublattice of $\sANbhdR(X,\cH)$ and satisfy $\omega(|\ell_n(A)|,\cH)=A$ for every $A\in\cA$.
  \end{enumerate}
\end{theorem}

We prove part (i) in Theorem~\ref{thm:main-correspondence} and part (ii) in Proposition~\ref{prop:constructed-ANbhd-epi}, Lemma~\ref{lem:regularized-preimage-omega}, and Corollary~\ref{cor:ANbhd-epi-nowhere-dense}. Part (iii) follows in Section~\ref{sec:computability} from Proposition~\ref{prop:suspension-multivalued-map}, Theorem~\ref{thm:FiniteSublattice}, Corollary~\ref{cor:base-generated-neighborhoods}, and Proposition~\ref{prop:base-scc-lattice}.

By Proposition~\ref{prop:Att-lattice}, $\sAtt(\cH)$ is a bounded distributive lattice with $A\vee B=A\cup B$ and $A\wedge B=\omega(A\cap B,\cH)$. Part (i) is the counterpart, for attractors defined through attracting neighborhoods, of the correspondence between attracting--repelling pairs, and its proof adapts a construction from the proof of that correspondence~\cite{Kvalheim:Gustafson:Koditschek}. The regularized preimage maps $\sANbhdR(\susp X,\Phi)$ onto $\sANbhdR(X,\cH)$ and preserves joins and bounds, but it can fail to preserve meets (Example~\ref{ex:ANbhd-preimage-meet}). Theorem~\ref{thm:ANbhd-meets} characterizes when it preserves meets, and in general, part (ii) holds with $\tcD$ generated by the closures of the interiors of the neighborhoods constructed for part (i).

Part (iii) applies the realization theorem (Theorem~\ref{thm:realization}) to a convergent cofiltration of outer approximations $\cF_n$ of the $\tau$-map $f_\tau:=\Phi(\tau,\cdot)$ on grids $\Xi_n$ of $\susp X$ built from $\cX_n$. Each $\mu\in\cX_n$ determines $d_n(\mu)\in\Xi_n$ with $\incl^{-1}(|d_n(\mu)|_{\susp})=|\mu|$, and $\mu'\in\cF_n^X(\mu)$ if a path of $\cF_n$ of positive length leads from $d_n(\mu)$ to $d_n(\mu')$ and its intermediate elements are outside $d_n(\cX_n)$. The nested pairs of forward invariant sets of an outer approximation of $f_\tau$ also provide index pairs for $\Phi$, and hence Conley indices (Proposition~\ref{prop:grid-conley-index}). In Section~\ref{sec:example}, we compute multivalued maps from $f_\tau$ for a bouncing ball, a rimless wheel, a spiking neuron, and an impacting oscillator with an adapted version of CMGDB~\cite{CMGDB}.

\section{Preliminaries}
\label{sec:preliminaries}

\subsection{Lattices, Attractors, and Outer Approximations}
\label{subsec:lattices}

We use the order-theoretic approach to dynamics of~\cite{Kalies:Mischaikow:Vandervorst:2014,Kalies:Mischaikow:Vandervorst:2016,Kalies:Mischaikow:Vandervorst:2022}. A \emph{lattice} is a set $L$ equipped with binary operations $\vee,\wedge:L\times L\to L$, called \emph{join} and \emph{meet}, that are idempotent, commutative, associative, and satisfy the identities $a\wedge(a\vee b)=a=a\vee(a\wedge b)$ for every $a,b\in L$. These operations induce a partial order $\leq$ via $a \leq b$ if and only if $a \wedge b = a$, or equivalently $a \vee b = b$. With respect to this order, $a\vee b$ and $a\wedge b$ are the least upper and greatest lower bounds of $a$ and $b$, respectively. A lattice is \emph{bounded} if it has a least element $0$ and a greatest element $1$, and is \emph{distributive} if
\[
  a\wedge(b\vee c)=(a\wedge b)\vee(a\wedge c),
  \quad
  a\vee(b\wedge c)=(a\vee b)\wedge(a\vee c), 
  \quad 
  \forall a,b,c \in L.
\]
A \emph{sublattice} of a bounded lattice $L$ is a subset closed under joins and meets that contains $0$ and $1$. A \emph{lattice homomorphism} preserves joins and meets, and a \emph{bounded lattice homomorphism} also preserves $0$ and $1$. We use \emph{monomorphism} and \emph{epimorphism} for an injective and a surjective bounded lattice homomorphism, respectively.

For a topological space $X$, write
\[
  \scR(X):=\setdef{F\subseteq X}{F=\cl\left(\Int(F)\right)}
\]
for the family of \emph{regular closed} subsets of $X$. For $F_1,F_2\in\scR(X)$, the operations
\[
  F_1\vee F_2:=F_1\cup F_2,
  \quad
  F_1\wedge F_2:=\cl\,\left(\Int(F_1)\cap\Int(F_2)\right),
\]
make $\scR(X)$ a bounded distributive lattice with least and greatest elements $\emptyset$ and $X$~\cite{Kalies:Mischaikow:Vandervorst:2016}. It is also a Boolean algebra with complement $F^*:=\cl_X(X\setminus F)$.

Let $(X,d)$ be a compact metric space and let $f:X\to X$ be continuous. For $U\subseteq X$, define the $\omega$-limit of $U$ under $f$ as
\[
  \omega(U,f):=\bigcap_{m\geq0}\cl\left(\bigcup_{n\geq m}f^n(U)\right).
\]
An \emph{attracting neighborhood} of $f$ is a set $U\subseteq X$ such that $\omega(U,f)\subseteq\Int(U)$. These form the family $\sANbhd(X,f)$, whose regular closed members form $\sANbhdR(X,f)$. An attracting neighborhood $U$ is a \emph{trapping region} if $f(\cl_X(U))\subseteq\Int_X(U)$~\cite{Kalies:Mischaikow:Vandervorst:2016}. The \emph{attractors} are the $\omega$-limits of attracting neighborhoods, and we denote their collection by $\sAtt(X,f)$. By~\cite{Kalies:Mischaikow:Vandervorst:2014}, the map $\omega(\cdot,f)$ satisfies
\[
  \omega(U_1\cap U_2,f) \subseteq \omega(U_1,f)\cap\omega(U_2,f),
  \quad
  \omega(\omega(U_1,f),f)=\omega(U_1,f).
\]
By~\cite{Kalies:Mischaikow:Vandervorst:2014,Kalies:Mischaikow:Vandervorst:2016}, $\sANbhd(X,f)$ and $\sANbhdR(X,f)$ are bounded distributive sublattices of $\cP(X)$ and $\scR(X)$, respectively, while $\sAtt(X,f)$ is a bounded distributive lattice with
\[
  A\vee B=A\cup B,
  \quad
  A\wedge B=\omega(A\cap B,f),
\]
and least and greatest elements $\emptyset$ and $\omega(X,f)$. The maps
\begin{equation}
  \omega(\cdot,f):\sANbhd(X,f)\surjection\sAtt(X,f),
  \quad
  \omega(\cdot,f):\sANbhdR(X,f)\surjection\sAtt(X,f)
  \label{eq:classical-omega-epi}
\end{equation}
are bounded lattice epimorphisms.

Let $\varphi:\Rp\times X\to X$ be a continuous semiflow, that is, a continuous map satisfying
\[
  \varphi(0,x)=x,
  \quad
  \varphi(t+s,x)=\varphi\left(t,\varphi(s,x)\right)
  \quad \forall x\in X,\ s,t\geq0.
\]
For $I\subseteq\Rp$ and $U\subseteq X$, write $\varphi(I,U):=\setdef{\varphi(t,x)}{t\in I,\ x\in U}$. A set $V\subseteq X$ is \emph{forward invariant} under $\varphi$ if $\varphi([0,\infty),V)\subseteq V$ and \emph{invariant} if $\varphi(t,V)=V$ for every $t\geq0$. A forward invariant $V$ is a \emph{trapping region} if $\varphi(\tau,\cl_X(V))\subseteq\Int_X(V)$ for some $\tau>0$~\cite{Kalies:Mischaikow:Vandervorst:2014}.

For $U\subseteq X$, define
\[
  \omega(U,\varphi)
  :=\bigcap_{T\geq0}\cl\,\left(\varphi([T,\infty),U)\right).
\]
Replacing $\omega(\cdot,f)$ by $\omega(\cdot,\varphi)$ in the definitions of attracting neighborhoods and attractors above defines $\sANbhd(X,\varphi)$, $\sANbhdR(X,\varphi)$, and $\sAtt(X,\varphi)$. A trapping region for $\varphi$ is equivalently a forward invariant attracting neighborhood~\cite{Kalies:Mischaikow:Vandervorst:2014}. The previous identities and lattice statements hold with $\varphi$ instead of $f$, and every $\omega(U,\varphi)$ is compact and invariant.

For $\tau>0$, the $\tau$-map of $\varphi$ is $f_\tau:=\varphi(\tau,\cdot)$. The identity map $\sAtt(X,\varphi)\to\sAtt(X,f_\tau)$ is a lattice isomorphism~\cite{Kalies:Mischaikow:Vandervorst:2016}. For arbitrary $U\subseteq X$, however, the two $\omega$-limits may differ. If $U\in\scR(X)$ and $f_\tau(U)\subseteq\Int(U)$, then $U$ is a trapping region for $f_\tau$ and an attracting neighborhood for both $f_\tau$ and $\varphi$, with $\omega(U,f_\tau)=\omega(U,\varphi)$~\cite{Kalies:Mischaikow:Vandervorst:2016}.

For compact $B\subseteq X$, write $\Inv(B,f)$ for the points lying on full orbits
\[
  (x_j)_{j\in\Z}\subseteq B,\quad f(x_j)=x_{j+1}\quad\forall j\in\Z,
\]
and $\Inv(B,\varphi)$ for the points lying on full orbits of $\varphi$ in $B$. We call $B$ \emph{isolating} for $f$ if $\Inv(B,f)\subseteq\Int_X B$, and analogously for $\varphi$. For either system, a \emph{Morse representation} of a compact invariant subset $S\subseteq X$ is a finite poset of nonempty, pairwise disjoint compact invariant subsets of $S$, called \emph{Morse sets}, such that every full orbit in $S$ either lies in one Morse set or has its $\alpha$- and $\omega$-limits in Morse sets $M'$ and $M$ with $M<M'$.

\begin{definition}
  \label{def:grid}
  A \emph{grid} on $X$ is a finite family of nonempty regular closed sets $|\xi|$, indexed by $\xi\in\cX$, such that
  \[
    X=\bigcup_{\xi\in\cX}|\xi|,
    \quad
    |\xi|\cap\Int|\xi'|=\emptyset
    \quad\text{whenever }\xi\neq\xi'.
  \]
  The \emph{evaluation map} $|\cdot|:\cP(\cX)\to\scR(X)$ is defined by $|\cU|=\bigcup_{\xi\in\cU}|\xi|$. For $\xi\in\cX$, let $\diam(\xi):=\max_{x,x'\in|\xi|}d(x,x')$, so that the diameter of $\cX$ is $\diam(\cX):=\max_{\xi\in\cX}\diam(\xi)$.
\end{definition}

By~\cite{Kalies:Mischaikow:Vandervorst:2016}, the evaluation map is a bounded distributive lattice isomorphism
\[
  \cP(\cX)\xrightarrow{\;\cong\;}\scR_{\cX}(X):=\setdef{|\cU|}{\cU\subseteq\cX},
\]
with $|\cU\cap\cV|=|\cU|\wedge|\cV|$ for all $\cU,\cV\subseteq\cX$.

For a finite family $\cF\subseteq\scR(X)$, the atoms of the Boolean subalgebra it generates form a grid~\cite{Kalies:Mischaikow:Vandervorst:2016}, called the \emph{grid generated by} $\cF$. To describe its elements, write $F^{+}:=F$ and $F^{-}:=F^*$ for $F\in\cF$, and define
\[
  |\sigma|:=\bigwedge_{F\in\cF}F^{\sigma(F)}
  \quad\forall\sigma:\cF\to\setof{+,-}.
\]
The nonempty $|\sigma|$ form the coarsest grid in which every member of $\cF$ is a union of grid elements. For two sets $F_1,F_2$, its elements are the nonempty sets among
\[
  F_1\wedge F_2,\quad F_1\wedge F_2^*,\quad F_1^*\wedge F_2,\quad F_1^*\wedge F_2^*.
\]
For example, the intervals $[0,2]$ and $[1,3]$ in $[0,3]$ generate a grid with elements $[0,1]$, $[1,2]$, and $[2,3]$, so each original interval is the union of two adjacent grid elements.

We say that a grid $\cX'$ \emph{refines} $\cX$, written as $\cX'\leq\cX$, if for every $\xi'\in\cX'$ there is a unique $\xi\in\cX$ such that $|\xi'|\subseteq|\xi|$. Their \emph{common refinement} $\cX\wedge\cX'$ is generated by the elements of both grids and consists of the nonempty meets $|\xi|\wedge|\xi'|$, where $\xi\in\cX$ and $\xi'\in\cX'$~\cite{Kalies:Mischaikow:Vandervorst:2016}.

\begin{definition}
  \label{def:outer-approximation}
  Let $f:X\to X$ be continuous and let $\cX$ be a grid. A \emph{multivalued map} $\cF:\cX\mvmap\cX$ assigns a set $\cF(\xi)\subseteq\cX$ to every $\xi\in\cX$. For $\cU\subseteq\cX$, write
  \[
    \cF(\cU):=\bigcup_{\xi\in\cU}\cF(\xi).
  \]
  The map $\cF$ is an \emph{outer approximation} of $f$ if
  \[
    f(|\xi|)\subseteq\Int|\cF(\xi)|
    \quad\forall\xi\in\cX.
  \]
  Its lattice of forward invariant sets is
  \[
    \sInvsetp(\cX,\cF) :=\setdef{\cU\subseteq\cX}{\cF(\cU)\subseteq\cU},
    \text{ with } \wedge = \cap,\, \vee = \cup.
  \]
\end{definition}

A multivalued map $\cF:\cX\mvmap\cX$ can be viewed as a directed graph on $\cX$, with edges $(\xi,\xi')$ whenever $\xi'\in\cF(\xi)$. Since a forward-invariant family contains every vertex reachable from it, $\sInvsetp(\cX,\cF)$ consists of unions of strongly connected components together with all their successors. Among these components, those whose induced graphs contain an edge are called \emph{recurrent} and form the \emph{Morse graph} $\MG(\cF)$, ordered by $M'\leq M$ when $M'$ is reachable from $M$. In this setting, we also call recurrent components and their grid realizations \emph{Morse sets}.

If $\cF$ is an outer approximation of $f$ and $\cU\in\sInvsetp(\cX,\cF)$, then
\begin{equation}
  f(|\cU|)
  \subseteq\Int|\cF(\cU)|
  \subseteq\Int|\cU|.
  \label{eq:invset-realization}
\end{equation}
By~\cite{Kalies:Mischaikow:Vandervorst:2016}, the evaluation map restricts to the bounded lattice monomorphism
\begin{equation}
  \sInvsetp(\cX,\cF)
  \to\sANbhdR(X,f),
  \quad
  \cU\longmapsto|\cU|.
  \label{eq:invset-neighborhood-map}
\end{equation}
The composition of~\eqref{eq:invset-neighborhood-map} with~\eqref{eq:classical-omega-epi} is a bounded lattice homomorphism into $\sAtt(X,f)$.

A sequence $\cX_0\geq\cX_1\geq\cdots$ is a \emph{cofiltration of grids}, called \emph{contracting} if $\diam(\cX_n)\to0$. For grids $\cX'\leq\cX$, write $|\cdot|_{\cX'}$ and $|\cdot|_{\cX}$ for the two evaluation maps. If $\cF':\cX'\mvmap\cX'$ and $\cF:\cX\mvmap\cX$, write $\cF'\leq\cF$ when
\begin{equation}
  |\cF'(\cU')|_{\cX'}\subseteq|\cF(\cU)|_{\cX}
  \quad\text{whenever}\quad
  |\cU'|_{\cX'}=|\cU|_{\cX},
  \label{eq:outer-map-refinement}
\end{equation}
for $\cU'\subseteq\cX'$ and $\cU\subseteq\cX$. Given a cofiltration $(\cX_n)_{n\geq0}$ of grids, a sequence $(\cF_n:\cX_n\mvmap\cX_n)_{n\geq0}$ is a \emph{cofiltration of multivalued maps} if $\cF_0\geq\cF_1\geq\cdots$ under this order~\cite{Kalies:Mischaikow:Vandervorst:2016}. For a grid $\cX_n$, the minimal and $r$-minimal multivalued maps are
\begin{equation*}
  \cF_n^{\min}(\xi) :=\setdef{\xi'\in\cX_n}{f(|\xi|)\cap|\xi'|\neq\emptyset},\quad
  \cF_n^{r}(\xi) :=\setdef{\xi'\in\cX_n}{B_r(f(|\xi|))\cap|\xi'|\neq\emptyset},
\end{equation*}
where $r>0$ and $B_r(E):=\setdef{x\in X}{d(x,E)<r}$. We also write $\overline B_r(E):=\setdef{x\in X}{d(x,E)\leq r}$ for $r\geq0$. A cofiltration $(\cF_n:\cX_n\mvmap\cX_n)_{n\geq0}$ of outer approximations of $f$ is \emph{convergent} if $\diam(\cX_n)\to0$ and there are numbers $r_n>0$ with $r_n\to0$ such that
\[
  \cF_n^{\min}(\xi)\subseteq\cF_n(\xi) \subseteq\cF_n^{r_n}(\xi), \quad \forall \xi\in\cX_n.
\]

\begin{theorem}[{\cite{Kalies:Mischaikow:Vandervorst:2016}}]
  \label{thm:realization}
  Let $X$ be a compact metric space and let $f:X\to X$ be continuous. Suppose $(\cF_n:\cX_n\mvmap\cX_n)_{n\geq0}$ is a convergent cofiltration of outer approximations of $f$. If $\cA\subseteq\sAtt(X,f)$ is a finite sublattice, then there exists $n_{\cA}$ such that for every $n\geq n_{\cA}$ there is a lattice monomorphism
  \[
    \ell_n:\cA\to\sInvsetp(\cX_n,\cF_n)
  \]
  satisfying
  \[
    \omega\left(|\ell_n(A)|,f\right)=A
    \quad\forall A\in\cA.
  \]
  Moreover, $\ell_n$ can be chosen so that $\setdef{|\ell_n(A)|}{A\in\cA}$ is a sublattice of $\sANbhd(X,f)$.
\end{theorem}

\subsection{Hybrid Dynamical Systems}
\label{subsec:hybrid-systems}

We use the framework of~\cite{Kvalheim:Gustafson:Koditschek} and the notation of~\cite{Rivas:Kalies}. A \emph{local semiflow} on a metric space $F$ is a continuous map $\varphi:\cU\to F$ whose domain $\cU$ is an open neighborhood of $\setof{0}\times F$ in $\Rp\times F$, with $\varphi(0,x)=x$ for every $x\in F$. For every $x\in F$ and $s,t\geq0$, its domain satisfies
\[
  (s+t,x)\in\cU\quad\Longleftrightarrow\quad(t,x)\in\cU\ \text{and}\ (s,\varphi(t,x))\in\cU,
\]
and whenever these conditions hold,
\[
  \varphi(s+t,x)=\varphi\left(s,\varphi(t,x)\right).
\]

\begin{definition}
  \label{def:hybrid-system}
  A \emph{hybrid dynamical system} is $\cH=(X,\varphi,G,r)$, where $X$ is a compact metric space, $G\subseteq X$ is a closed set called the \emph{guard set}, $\varphi:\cU\to F$ is a local semiflow on $F:=X\setminus G$, and $r:G\to X$ is a continuous map called the \emph{reset map}.
\end{definition}

\begin{definition}
  \label{defn:ImpactTime}
  The \emph{impact time} $\stoptime:X\to\Rp\cup\setof{\infty}$ is
  \[
    \stoptime(x):=
    \begin{cases}
      \sup\setdef{t\geq0}{(t,x)\in\cU}, & x\in F,\\
      0, & x\in G.
    \end{cases}
  \]
\end{definition}
Thus $\cU=\setdef{(t,x)\in\Rp\times F}{t<\stoptime(x)}$. We assume that each orbit with finite impact time converges to a point of $G$. Write $\evarphi$ for the extension of $\varphi$ to these endpoints, defined by
\[
  \evarphi(\stoptime(x),x):=\lim_{t\uparrow\stoptime(x)}\varphi(t,x)
  \quad\forall x\in F\text{ with }\stoptime(x)<\infty,
\]
and $\evarphi(0,g):=g$ for all $g\in G$. The \emph{impact map} $p:\setdef{x\in X}{\stoptime(x)<\infty}\to G$ is given by $p(x):=\evarphi(\stoptime(x),x)$, and in particular, $p|_G=\id_G$.

\begin{definition}
  \label{def:hybrid-trajectory}
  A \emph{hybrid trajectory} from $x\in X$ consists of $J\in\Zp\cup\setof{\infty}$, a nondecreasing sequence $(t_n)_{n=0}^{J}\subseteq\Rp$ with $t_0=0$, and continuous maps $\gamma_n:I_n\to X,\quad 0\leq n\leq J$, satisfying the following conditions.
  \begin{enumerate}[(i)]
    \item $I_n=[t_n,t_{n+1}]$ for $n<J$. If $J<\infty$, then $I_J=[t_J,\infty)$.
    \item $\gamma_0(0)=x$ and $t_{n+1}=t_n+\stoptime(\gamma_n(t_n))$ for $n<J$. If $J<\infty$, then $\stoptime(\gamma_J(t_J))=\infty$.
    \item For $n<J$, $\gamma_n(t_{n+1})\in G$ and $\gamma_{n+1}(t_{n+1})=r(\gamma_n(t_{n+1}))$.
    \item For $n\leq J$ and $t\in I_n$,
    \[
      \gamma_n(t)=\evarphi(t-t_n,\gamma_n(t_n)).
    \]
  \end{enumerate}
  We write
  \[
    \dom(\psi_x):=\bigcup_{n=0}^{J}(I_n\times\setof{n}),\quad
    \psi_x(t,n):=\gamma_n(t).
  \]
\end{definition}
A hybrid trajectory $\psi_x$ is \emph{global} if $\bigcup_n I_n$ is unbounded, and is \emph{Zeno} if $J=\infty$ and $\sup_n t_n<\infty$. Each $x\in X$ determines a unique maximal trajectory. In both cases, $t+n$ is unbounded on $\dom(\psi_x)$, and the guard state and its reset image occur at their respective values of $t+n$. For $I\subseteq\Rp$ and $A\subseteq X$, write
\[
  \psi(I,A):=\setdef{\psi_x(t,n)}{x\in A,\ (t,n)\in\dom(\psi_x),\ t+n\in I}.
\]

Related systems are studied as impulsive dynamical systems or impulsive semidynamical systems~\cite{Bonotto:Gimenes:Souto}, where the impact time is the next positive time at which a trajectory meets the impulsive set, while $\stoptime(g)=0$ for $g\in G$ here.
\begin{definition}[{\cite{Kvalheim:Gustafson:Koditschek}}]
  \label{def:trapping-guard}
  The system $\cH$ satisfies the \emph{trapping guard condition} (TGC) if there is an open neighborhood $U_G$ of $G$ on which $\stoptime$ is finite and continuous and the extension of $\varphi$
  \[
    \evarphi:\setdef{(t,x)\in\Rp\times U_G}{t\leq\stoptime(x)}\to X
  \]
  is continuous.
\end{definition}
Under TGC, $p(x)=\evarphi(\stoptime(x),x)$ is continuous on $U_G$ and $p|_G=\id_G$, so $p|_{U_G}:U_G\to G$ is a retraction. Compactness of $X$ also implies that TGC has the following equivalent formulations.

\begin{proposition}
  \label{prop:trapping-guard-characterizations}
  For a hybrid system $\cH$, the following are equivalent.
  \begin{enumerate}[(i)]
    \item $\cH$ satisfies TGC.
    \item There is an open neighborhood $U$ of $G$ on which $\stoptime$ is finite and continuous and $p|_U$ is continuous.
    \item There is an open neighborhood $V$ of $G$ on which $\stoptime$ is finite and continuous and
    \[
      H:V\times[0,1]\to V,\quad H(x,s):=\evarphi(s\stoptime(x),x),
    \]
    is a strong deformation retraction of $V$ onto $G$.
  \end{enumerate}
\end{proposition}
\begin{proof}
We prove $(i)\Rightarrow(ii)\Rightarrow(i)\Rightarrow(iii)\Rightarrow(ii)$.

First, $(i)\Rightarrow(ii)$ follows immediately by continuity of $\stoptime$ and $\evarphi$ on $U_G$.

Conversely, suppose $(ii)$ holds on an open neighborhood $U_G$. It suffices to prove continuity of $\evarphi$ at the endpoint. Consider $(t_n,x_n)\to(\stoptime(x),x)$ in $D_U=\setdef{(t,x)\in\Rp\times U_G}{t\leq\stoptime(x)}$ and set $y_n:=\evarphi(t_n,x_n)$. Then $\stoptime(y_n)=\stoptime(x_n)-t_n\to0$. We claim that every accumulation point of $(y_n)$ lies in $G$. Indeed, if $y_n\to z\notin G$ along a subsequence, then openness of the local-semiflow domain at $(0,z)$ yields $\varepsilon>0$ and a neighborhood $W$ of $z$ such that every point in $W$ flows for at least time $\varepsilon$. Thus $\stoptime(y_n)\geq\varepsilon$ eventually, a contradiction.

By compactness of $X$, let $z\in G$ be an accumulation point of $(y_n)$. Since $U_G$ is a neighborhood of $G$, eventually $y_n\in U_G$; hence $p(y_n)=p(x_n)$. By continuity of $p$ and $p|_G=\id_G$,
\[
  z=p(z)=\lim p(y_n)=\lim p(x_n)=p(x).
\]
Thus $p(x)$ is the unique accumulation point of $(y_n)$, and compactness gives $y_n\to p(x)=\evarphi(\stoptime(x),x)$. Hence $\evarphi$ is continuous on $D_U$, proving $(ii)\Rightarrow(i)$.

Now assume $(i)$ and define $H_0(x,s):=\evarphi(s\stoptime(x),x)$ on a TGC neighborhood $U_G$. This map is continuous. Set $V:=\{x\in U_G:H_0(x,[0,1])\subseteq U_G\}$. Then $V$ is an open neighborhood of $G$. If $y=H_0(x,s)$ with $x\in V$, we have
\[
\stoptime(y)=(1-s)\stoptime(x),\qquad
H_0(y,u)=H_0(x,s+u(1-s)),
\]
so $H_0(V\times[0,1])\subseteq V$. Therefore $H:=H_0|_{V\times[0,1]}$ satisfies $H(x,0)=x$, $H(x,1)=p(x)\in G$, and $H(g,s)=g$ for $g\in G$. Thus $H$ is a strong deformation retraction of $V$ onto $G$, proving $(i)\Rightarrow(iii)$.

Finally, $(iii)\Rightarrow(ii)$ follows immediately from $p(x)=H(x,1)$.
\end{proof}
The following example shows why finiteness and continuity of $\stoptime$ near $G$ alone do not ensure TGC.
\begin{example}
  \label{ex:continuous-impact-time}
  Let $X=[0,1]^2$, $G=[0,1]\times\setof{0}$, and $r(x,0)=(x,1)$. On $F=X\setminus G$, define
  \[
    \varphi(t,(x,y)):=\left(x+(1-x)\frac{t}{y},\,y-t\right),\quad 0\leq t<y.
  \]
  Note that $\stoptime(x,y)=y$ on $X$, so the impact time is finite and continuous on $X$. For $y>0$, $\varphi(t,(x,y))\to(1,0)$ as $t\to\stoptime(x,y)$, so $H((x,y),1)=(1,0)$, while $H((x,0),1)=(x,0)$. Thus $H(\cdot,1)$ is discontinuous at every $(x,0)$ with $x<1$, and TGC fails.
\end{example}

We recall the hybrid suspension construction of~\cite{Kvalheim:Gustafson:Koditschek}. For the mapping cylinder $X':=X\sqcup_G(G\times[0,1])$ of $G\injection X$, write $\iota:X\injection X'$ for the inclusion and define $\pi_0:X'\to X$ by $\pi_0(\iota(x))=x$ and $\pi_0(g,s)=g$. The \emph{hybrid suspension space} is
\[
  \susp_{\cH}X
  :=\frac{X'}{(g,1)\sim\iota(r(g)),\ g\in G}
  \cong
  \frac{X\sqcup(G\times[0,1])}
  {(g,0)\sim g,\ (g,1)\sim r(g)}.
\]
Let $\pi:X'\to\susp_{\cH}X$ be the quotient map and write $\incl:=\pi\circ\iota:X\injection\susp_{\cH}X$.

The \emph{hybrid suspension semiflow} $\Phi_{\cH}:\Rp\times\susp_{\cH}X\to\susp_{\cH}X$ is defined recursively as follows. If $\stoptime(x)=\infty$, then $\Phi_{\cH}(T,\incl(x)):=\incl(\varphi(T,x))$. If $\stoptime(x)<\infty$, then
\[
  \Phi_{\cH}(T,\incl(x))=
  \begin{cases}
    \incl(\evarphi(T,x)), & 0\leq T\leq\stoptime(x),\\
    \pi(\evarphi(\stoptime(x),x),T-\stoptime(x)), & \stoptime(x)<T\leq\stoptime(x)+1,\\
    \Phi_{\cH}(T-\stoptime(x)-1,\incl(r(\evarphi(\stoptime(x),x)))), & T>\stoptime(x)+1.
  \end{cases}
\]
For $(g,s)\in G\times[0,1]$, define
\[
  \Phi_{\cH}(T,\pi(g,s))=
  \begin{cases}
    \pi(g,s+T), & 0\leq T\leq1-s,\\
    \Phi_{\cH}(T-(1-s),\incl(r(g))), & T>1-s.
  \end{cases}
\]
Figure~\ref{fig:hybrid-suspension} illustrates the construction.

\begin{figure}[htbp]
  \centering
  \tikzset{
    midarrow/.style={postaction={decorate},decoration={markings,mark=at position #1 with {\arrow{>}}}},
    suspension flow/.style={thick,blue!70!black,midarrow=0.5},
    suspension reset/.style={thick,dashed,red!70!black,midarrow=0.5}
  }
  \begin{subfigure}[t]{0.32\textwidth}
    \centering
    \begin{tikzpicture}[x=0.82cm,y=0.82cm,font=\small]
      \path[use as bounding box] (-0.2,-0.65) rectangle (5.45,4.1);
      \draw[thick,fill=gray!4] (0,0) .. controls (1.5,-0.3) and (3,0.3) .. (4.5,0) -- (5.25,1.9) .. controls (3.75,2.2) and (2.25,1.6) .. (0.75,1.9) -- cycle;
      \filldraw[fill=gray!25,thick] (3.6,1) ellipse [x radius=0.65,y radius=0.25];
      \node at (3.65,1.45) {$G$};
      \draw[suspension flow] (1.1,0.45) .. controls (1.8,0.15) and (2.9,0.3) .. (3.6,0.75);
      \draw[suspension flow] (1.25,1.2) .. controls (1.5,1.45) and (1.8,1.65) .. (2.2,1.55);
      \draw[suspension reset] (3.6,0.75) .. controls (2.9,1.3) and (2,1.35) .. (1.25,1.2);
      \fill (1.1,0.45) circle (1.5pt);
      \fill (3.6,0.75) circle (1.5pt);
      \fill (1.25,1.2) circle (1.5pt);
      \node[below] at (1.1,0.4) {$x$};
      \node[below right] at (3.6,0.75) {$g$};
      \node[below right] at (1.25,1.15) {$r(g)$};
      \node[blue!70!black] at (2.25,-0.4) {$\varphi$};
      \node[red!70!black] at (2.45,1.55) {$r$};
    \end{tikzpicture}
    \caption{$X$}
  \end{subfigure}\hfill
  \begin{subfigure}[t]{0.32\textwidth}
    \centering
    \begin{tikzpicture}[x=0.82cm,y=0.82cm,font=\small]
      \path[use as bounding box] (-0.2,-0.65) rectangle (5.45,4.1);
      \draw[thick,fill=gray!4] (0,0) .. controls (1.5,-0.3) and (3,0.3) .. (4.5,0) -- (5.25,1.9) .. controls (3.75,2.2) and (2.25,1.6) .. (0.75,1.9) -- cycle;
      \fill[gray!10] (2.95,1) -- (2.95,3.25) -- (4.25,3.25) -- (4.25,1) -- cycle;
      \draw[thick] (3.6,1) ellipse [x radius=0.65,y radius=0.25];
      \draw[thick] (2.95,1) -- (2.95,3.25) (4.25,1) -- (4.25,3.25);
      \filldraw[fill=gray!25,thick] (3.6,3.25) ellipse [x radius=0.65,y radius=0.25];
      \node at (3.6,3.8) {$G\times\{1\}$};
      \node at (5.03,2.1) {$s$};
      \draw[->,blue!70!black] (4.7,1.55) -- (4.7,2.6);
      \draw[suspension flow] (1.1,0.45) .. controls (1.8,0.15) and (2.9,0.3) .. (3.6,0.75);
      \draw[suspension flow] (3.6,0.75) -- (3.6,3);
      \draw[suspension reset] (3.6,3) .. controls (2.25,3.1) and (1.2,2.2) .. (1.25,1.2);
      \fill (1.1,0.45) circle (1.5pt);
      \fill (3.6,0.75) circle (1.5pt);
      \fill (3.6,3) circle (1.5pt);
      \fill (1.25,1.2) circle (1.5pt);
      \node[below] at (3.6,-0.1) {$(g,0)$};
      \node[left] at (2.85,3.25) {$(g,1)$};
      \node[below right] at (1.25,1.15) {$\iota(r(g))$};
    \end{tikzpicture}
    \caption{$X'$}
  \end{subfigure}\hfill
  \begin{subfigure}[t]{0.32\textwidth}
    \centering
    \begin{tikzpicture}[x=0.82cm,y=0.82cm,font=\small]
      \path[use as bounding box] (-0.2,-0.65) rectangle (5.45,4.1);
      \draw[thick,fill=gray!4] (0,0) .. controls (1.5,-0.3) and (3,0.3) .. (4.5,0) -- (5.25,1.9) .. controls (3.75,2.2) and (2.25,1.6) .. (0.75,1.9) -- cycle;
      \filldraw[fill=gray!25,thick] (3.6,1) ellipse [x radius=0.65,y radius=0.25];
      \draw[gray!50] (3.2,0.83) .. controls (4.2,3.6) and (0.6,3.5) .. (1,1.35);
      \draw[gray!50] (4,0.83) .. controls (4.8,4.1) and (0.35,3.8) .. (1.55,1.2);
      \draw[suspension flow] (1.1,0.45) .. controls (1.8,0.15) and (2.9,0.3) .. (3.6,0.75);
      \draw[suspension flow] (3.6,0.75) .. controls (4.5,3.8) and (0.45,3.65) .. (1.25,1.2);
      \draw[suspension flow] (1.25,1.2) .. controls (1.5,1.45) and (1.8,1.65) .. (2.2,1.55);
      \fill (1.1,0.45) circle (1.5pt);
      \fill (3.6,0.75) circle (1.5pt);
      \fill (1.25,1.2) circle (1.5pt);
      \node[below] at (1.1,-0.1) {$\incl(x)$};
      \node[below] at (3.8,-0.1) {$\incl(g)$};
      \node[below right] at (1.25,1.15) {$\incl(r(g))$};
      \node[blue!70!black] at (2.4,3.35) {$\pi(g,s)$};
      \node[blue!70!black] at (2.45,-0.4) {$\Phi_{\cH}$};
    \end{tikzpicture}
    \caption{$\susp_{\cH}X$}
  \end{subfigure}
  \caption{The suspension construction. We obtain $X'$ by attaching $G\times[0,1]$ along $G\times\{0\}$, and $\susp_{\cH}X$ by identifying $(g,1)$ with $\iota(r(g))$. The semiflow follows $\varphi$ in the base and moves along each attached interval in one unit of time.}
  \label{fig:hybrid-suspension}
\end{figure}
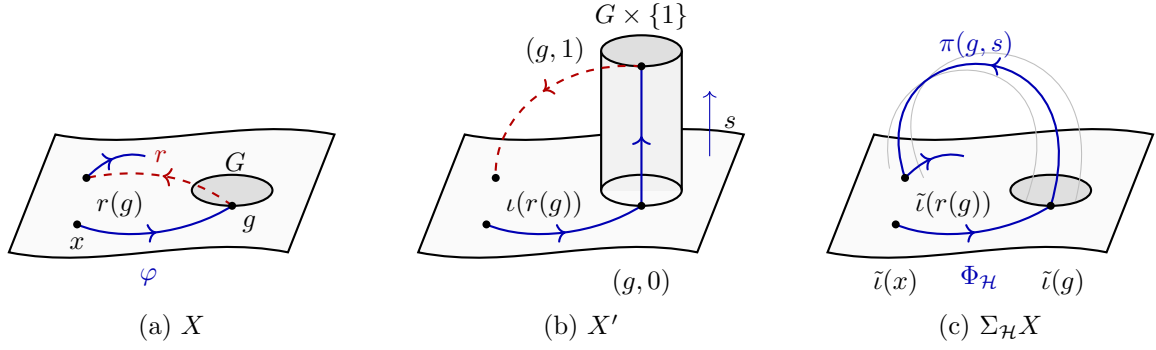

\begin{theorem}[{\cite[Proposition~6.9]{Kvalheim:Gustafson:Koditschek}}]
  \label{thm:hybrid-suspension-semiflow}
  The space $\susp_{\cH}X$ is compact and metrizable. If $\cH$ satisfies the trapping guard condition, then $\Phi_{\cH}$ is a continuous semiflow on $\susp_{\cH}X$.
\end{theorem}
\noindent 
For fixed $\cH$, write $\susp X:=\susp_{\cH}X$ and $\Phi:=\Phi_{\cH}$, and define $\susp:\cP(X)\to\cP(\susp X)$ by $\susp(A):=\pi\left(\pi_0^{-1}(A)\right)$.

\begin{proposition}[{\cite[Lemmas~3.5--3.7 and~3.9]{Rivas:Kalies}}]
  \label{prop:suspension-identities}
  The suspension satisfies the following:
  \begin{enumerate}[(i)]
    \item For every $A\subseteq X$, $\incl^{-1}\left(\susp(A)\right)=A\cup r(A\cap G)$.
    \item For every $\tB\subseteq\susp X$,
    \[
      \susp\left(\incl^{-1}(\tB)\right)
      =\left(\tB\cap\incl(X)\right)\cup\pi\left(\incl^{-1}\left(\tB\cap\incl(G)\right)\times(0,1]\right).
    \]
    \item For every hybrid trajectory $\psi_x$ and every $(t,n)\in\dom(\psi_x)$, $\incl\left(\psi_x(t,n)\right)=\Phi(t+n,\incl(x))$.
    \item For every $A\subseteq X$, $\susp\left(\psi([0,\infty),A)\right)=\Phi\left([0,\infty),\susp(A)\right)$.
  \end{enumerate}
\end{proposition}

\section{Attractor Lattices for Hybrid Systems}
\label{sec:hybrid-lattices}

Fix a hybrid dynamical system $\cH=(X,\varphi,G,r)$.

\begin{definition}
  \label{def:hybrid-omega}
  For $U\subseteq X$, define its $\omega$-limit under $\cH$ by
  \[
    \omega(U,\cH):=\bigcap_{T\geq0}\cl\,\left(\psi([T,\infty),U)\right).
  \]
  For $x\in X$, we abbreviate $\omega(x,\cH):=\omega(\setof{x},\cH)$.

  A set $S\subseteq X$ is \emph{forward invariant} under $\cH$ if $\psi([0,\infty),S)\subseteq S$.
\end{definition}

\begin{definition}
  \label{def:hybrid-trapping}
  A set $U\subseteq X$ is a \emph{trapping region} for $\cH$ if it is forward invariant under $\cH$ and $\omega(U,\cH)\subseteq\Int_X(U)$.
\end{definition}

\begin{definition}
  \label{def:hybrid-attractor}
  A set $U\subseteq X$ is an \emph{attracting neighborhood} for $\cH$ if $\omega(U,\cH)\subseteq\Int(U)$. A compact set $A\subseteq X$ is an \emph{attractor} for $\cH$ if $A=\omega(U,\cH)$ for some attracting neighborhood $U \subseteq X$.
\end{definition}
\noindent
We write $\sANbhd(X,\cH)$ for the attracting neighborhoods, $\sANbhdR(X,\cH):=\sANbhd(X,\cH)\cap\scR(X)$ for the regular closed attracting neighborhoods, and $\sAtt(\cH)$ for the attractors.

Without TGC, an $\omega$-limit can fail to be forward invariant, as the following example shows.
\begin{example}
  Let $X=[-1,0]\cup[1,2]$, let $G=\setof{0,2}$, and define $r(0)=1$ and $r(2)=-1$. The local semiflow on $F=X\setminus G$ is given by
  \[
    \varphi(t,x):=
    \begin{cases}
      e^{-t}x, & x\in[-1,0),\ t\geq0,\\
      x+t, & x\in[1,2),\ 0\leq t<2-x.
    \end{cases}
  \]
  Note that every neighborhood of $G$ contains points with $\stoptime(x)=\infty$ even though $\stoptime(0)=0$, so TGC fails. Note that $\omega(x,\cH)=\setof{0}$ for every $x \in X$, but $r(0)=1$.
\end{example}

In what follows, assume that $\cH$ satisfies the trapping guard condition.

\subsection{\texorpdfstring{$\omega$}{Omega}-limits and invariant sets}

We recall the following results relating $\omega$-limits and invariant subsets of $X$ and $\susp X$.
\begin{proposition}[{\cite[Proposition~6.11]{Kvalheim:Gustafson:Koditschek}}]
  \label{prop:omega-preimage-identity}
  For every $A\subseteq X$,
  \[
    \omega(A,\cH)=\incl^{-1}\left(\omega(\incl(A),\Phi)\right).
  \]
\end{proposition}
\begin{lemma}[{\cite[Lemma~3.10]{Rivas:Kalies}}]
  \label{lem:invariance-implies-fiber}
  Let $\tS\subseteq\susp X$ be invariant under $\Phi$. If $\pi(g,s)\in \tS$ for some $g\in G$ and $0\leq s<1$, then $\pi\left(\setof{g}\times[0,1]\right)\subseteq \tS$.
\end{lemma}

Applying~\cite[Propositions~3.11--3.12]{Rivas:Kalies} to $\omega$-limits yields the following identities.
\begin{lemma}
  \label{lem:omega-preimage}
  For $\tU\subseteq\susp X$, define $\tK:=\omega(\tU,\Phi)$ and $K:=\incl^{-1}(\tK)$. Then $K$ is forward invariant under $\cH$ and
  \[
    \tK=\susp(K),
    \quad
    \omega(\incl(K),\Phi)=\tK,
    \quad
    \omega(K,\cH)=K.
  \]
\end{lemma}
\begin{proof}
By Theorem~\ref{thm:hybrid-suspension-semiflow}, $\Phi$ is a continuous semiflow on the compact metrizable space $\susp X$, so $\tK$ is compact and invariant under $\Phi$. For $x\in K$, Proposition~\ref{prop:suspension-identities}(iii) yields $\incl(\psi_x(t,n))=\Phi(t+n,\incl(x))\in\tK$ for every $(t,n)\in\dom(\psi_x)$, so $K$ is forward invariant under $\cH$.

Since $\stoptime(g)=0$ for every $g\in G$, the definition of $\Phi$ implies
\[
  \pi(g,s)=\Phi(s,\incl(g))\quad\forall(g,s)\in G\times[0,1].
\]
By definition, $\susp(K)=\incl(K)\cup\pi\left((K\cap G)\times(0,1]\right)$, therefore
\[
  \susp(K)\subseteq\Phi([0,1],\incl(K))\subseteq\tK,
\]
and the second follows from $\incl(K)\subseteq\tK$ and the forward invariance of $\tK$. Conversely, $\susp X=\incl(X)\cup\pi\left(G\times(0,1)\right)$ and $\tK\cap\incl(X)=\incl(K)$. If $\pi(g,s)\in\tK$ with $g\in G$ and $0<s<1$, then Lemma~\ref{lem:invariance-implies-fiber} yields $\incl(g)=\pi(g,0)\in\tK$, so $g\in K\cap G$ and $\pi(g,s)\in\susp(K)$. Therefore $\tK=\susp(K)$.

Since $\incl(K)\subseteq\tK$ and $\tK$ is closed and forward invariant, $\omega(\incl(K),\Phi)\subseteq\cl\,\Phi([0,\infty),\incl(K))\subseteq\tK$. Conversely, the invariance of $\tK$ and the inclusion $\susp(K)\subseteq\Phi([0,1],\incl(K))$ imply
\[
  \tK=\Phi(T,\tK)=\Phi(T,\susp(K))\subseteq\Phi([T,T+1],\incl(K))\subseteq\Phi([T,\infty),\incl(K))\quad\forall T\geq0,
\]
so $\tK\subseteq\omega(\incl(K),\Phi)$. Finally, Proposition~\ref{prop:omega-preimage-identity} yields
\[
  \omega(K,\cH)=\incl^{-1}\left(\omega(\incl(K),\Phi)\right)=\incl^{-1}(\tK)=K.
\]
\end{proof}

\subsection{The lattice operations}

\begin{proposition}
  \label{prop:hybrid-omega-calculus}
  For all $U,V\subseteq X$,
  \[
    \omega(U\cup V,\cH)=\omega(U,\cH)\cup\omega(V,\cH),
    \quad
    \omega(U\cap V,\cH)\subseteq\omega(U,\cH)\cap\omega(V,\cH).
  \]
  If $U\subseteq V$, then $\omega(U,\cH)\subseteq\omega(V,\cH)$. For every $U\subseteq X$, the set $\omega(U,\cH)$ is compact and forward invariant under $\cH$, and $\omega\left(\omega(U,\cH),\cH\right)=\omega(U,\cH)$.
\end{proposition}
\begin{proof}
  By Proposition~\ref{prop:omega-preimage-identity}, the union identity, monotonicity, the intersection inclusion, and compactness follow from the corresponding properties of $\omega$-limits for $\Phi$ on $\susp X$. Applying Lemma~\ref{lem:omega-preimage} to $\tU=\incl(U)$ proves that $\omega(U,\cH)$ is forward invariant and $\omega(\omega(U,\cH),\cH)=\omega(U,\cH)$.
\end{proof}

\begin{proposition}
  \label{prop:ANbhd-lattice}
  The family $\sANbhdR(X,\cH)$ is a bounded distributive sublattice of $\scR(X)$.
\end{proposition}
\begin{proof}
  The union identity for $\omega$ proves closure under joins. For $U,V\in\sANbhdR(X,\cH)$, monotonicity and $U\wedge V\subseteq U\cap V$ imply
  \[
    \omega(U\wedge V,\cH)\subseteq\omega(U,\cH)\cap\omega(V,\cH)\subseteq\Int(U)\cap\Int(V)\subseteq\Int(U\wedge V).
  \]
  Since $\emptyset,X\in\sANbhdR(X,\cH)$, this is a bounded sublattice of $\scR(X)$ and inherits its distributivity.
\end{proof}

\begin{proposition}
  \label{prop:Att-lattice}
  The collection $\sAtt(\cH)$ is a bounded distributive lattice ordered by inclusion. For $A,B\in\sAtt(\cH)$, its operations are
  \[
    A\vee B:=A\cup B,
    \quad
    A\wedge B:=\omega(A\cap B,\cH),
  \]
  and its least and greatest elements are $\emptyset$ and $\omega(X,\cH)$, respectively. The map
  \[
    \omega(\cdot,\cH):\sANbhdR(X,\cH)\surjection\sAtt(\cH)
  \]
  is a bounded lattice epimorphism.
\end{proposition}
\begin{proof}
Let $A=\omega(U,\cH)$ and $B=\omega(V,\cH)$ such that $U,V\in\sANbhdR(X,\cH)$. By Proposition~\ref{prop:hybrid-omega-calculus}, $\omega(U\vee V,\cH)=A\cup B$. Also, $A\cap B\subseteq U\wedge V\subseteq U\cap V$, so monotonicity and idempotence of $\omega$ give
\[
  \omega(A\cap B,\cH)\subseteq\omega(U\wedge V,\cH)\subseteq A\cap B \quad\Longrightarrow\quad \omega(U\wedge V,\cH)=\omega(A\cap B,\cH).
\]
Thus $\omega(\cdot,\cH)$ preserves joins and meets.

For surjectivity, if $A=\omega(U,\cH)$ with $U$ attracting, choose an open $W$ such that $A\subseteq W\subseteq\cl W\subseteq\Int(U)$. Then $N:=\cl W\in\scR(X)$ and $A=\omega(A,\cH)\subseteq\omega(N,\cH)\subseteq\omega(U,\cH)=A$, so $N\in\sANbhdR(X,\cH)$ and $\omega(N,\cH)=A$.

Hence $\omega(\cdot,\cH)$ is a bounded lattice epimorphism. Since $\sANbhdR(X,\cH)$ is bounded distributive by Proposition~\ref{prop:ANbhd-lattice}, so is $\sAtt(\cH)$, with $A\vee B=A\cup B$, $A\wedge B=\omega(A\cap B,\cH)$, and bounds $\emptyset$ and $\omega(X,\cH)$.
\end{proof}

\section{The Suspension Correspondence}
\label{sec:suspension-attractor-iso}

The map $\incl^{-1}$ identifies attracting--repelling pairs of the hybrid system and its suspension~\cite[Proposition~6.15]{Kvalheim:Gustafson:Koditschek}. We identify the resulting attractor lattices, describe the inverse explicitly, and compare their regular closed attracting neighborhoods.

\subsection{Forward invariance and trapping regions}

\begin{proposition}
  \label{prop:HybridForwardInvariantSet}
  If $S\subseteq X$ is forward invariant under $\cH$, then $\susp(S)$ is forward invariant under $\Phi$ and $\incl^{-1}\left(\susp(S)\right)=S$.
\end{proposition}
\begin{proof}
  Forward invariance implies $r(S\cap G)\subseteq S$, so Proposition~\ref{prop:suspension-identities}(i) implies $\incl^{-1}(\susp(S))=S$. Proposition~\ref{prop:suspension-identities}(iv) and $\psi([0,\infty),S)\subseteq S$ imply that $\susp(S)$ is forward invariant.
\end{proof}

\begin{proposition}
  \label{prop:SemiflowForwardInvariantSet}
  If $\tS\subseteq\susp X$ is forward invariant under $\Phi$, then $\incl^{-1}(\tS)$ is forward invariant under $\cH$ and $\susp\left(\incl^{-1}(\tS)\right)\subseteq \tS$.
\end{proposition}
\begin{proof}
  Let $S:=\incl^{-1}(\tS)$. If $x\in S$, Proposition~\ref{prop:suspension-identities}(iii) yields $\incl(\psi_x(t,n))=\Phi(t+n,\incl(x))\in \tS$ for every $(t,n)\in\dom(\psi_x)$, so $S$ is forward invariant. Moreover, for $g\in S\cap G$, we have $\pi(\setof{g}\times[0,1])=\Phi([0,1],\incl(g))\subseteq\tS$, so Proposition~\ref{prop:suspension-identities}(ii) implies $\susp(S)\subseteq\tS$.
\end{proof}

The following is adapted from~\cite{Kvalheim:Gustafson:Koditschek}. Fix $0<b<a<1$. For $U\subseteq X$, define
\[
  \Wab(U):=\incl(U)\cup\pi\left((U\cap G)\times[0,a]\right)\cup\pi\left(r^{-1}(U)\times[b,1]\right).
\]

Figure~\ref{fig:wab} illustrates the definition.

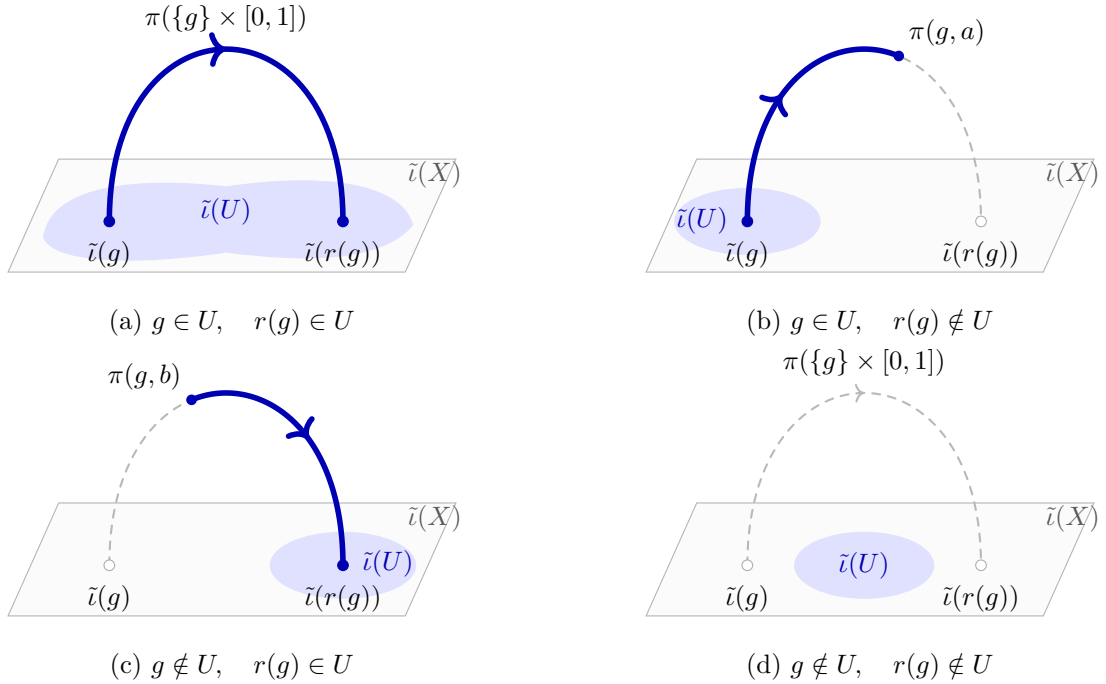
\begin{figure}[htbp]
  \centering
  \tikzset{midarrow/.style={postaction={decorate},decoration={markings,mark=at position #1 with {\arrow{>}}}}}
  \begin{subfigure}[t]{0.47\textwidth}
    \centering
    \begin{tikzpicture}[x=1.03cm,y=1.03cm,font=\small,line cap=round]
      \path[use as bounding box] (0,-0.30) rectangle (5.95,3.2);
      \draw[gray!60,fill=gray!3] (0.1,-0.1) -- (5.2,-0.1) -- (5.85,1.35) -- (0.75,1.35) -- cycle;
      \fill[blue!12] (0.55,0.35) .. controls (0.65,1.05) and (1.85,1.12) .. (2.9,1.0) .. controls (4.4,1.22) and (5.2,0.95) .. (5.3,0.5) .. controls (4.9,0.02) and (3.9,0.02) .. (2.9,0.15) .. controls (1.7,-0.03) and (0.65,0.02) .. cycle;
      \node[blue!70!black] at (2.9,0.68) {$\incl(U)$};
      \draw[gray!55,dashed,line width=0.8pt] plot[domain=0:1,samples=61,variable=\t] ({1.4+3*(3*\t*\t-2*\t*\t*\t)},{0.55+8.85*\t*(1-\t)});
      \draw[blue!70!black,line width=2.1pt,midarrow=0.5] plot[domain=0:1,samples=61,variable=\t] ({1.4+3*(3*\t*\t-2*\t*\t*\t)},{0.55+8.85*\t*(1-\t)});
      \draw[blue!70!black,fill=blue!70!black] (1.4,0.55) circle (2pt);
      \draw[blue!70!black,fill=blue!70!black] (4.4,0.55) circle (2pt);
      \node[below] at (1.4,0.43) {$\incl(g)$};
      \node[below] at (4.4,0.43) {$\incl(r(g))$};
      \node[gray!70!black,anchor=west] at (5.1,1.15) {$\incl(X)$};
      \node[above] at (2.9,2.85) {$\pi(\setof{g}\times[0,1])$};
    \end{tikzpicture}
    \caption{$g\in U,\quad r(g)\in U$}
  \end{subfigure}
  \hfill
  \begin{subfigure}[t]{0.47\textwidth}
    \centering
    \begin{tikzpicture}[x=1.03cm,y=1.03cm,font=\small,line cap=round]
      \path[use as bounding box] (0,-0.30) rectangle (5.95,3.2);
      \draw[gray!60,fill=gray!3] (0.1,-0.1) -- (5.2,-0.1) -- (5.85,1.35) -- (0.75,1.35) -- cycle;
      \fill[blue!12] (1.4,0.55) ellipse (0.94 and 0.43);
      \node[blue!70!black] at (0.82,0.57) {$\incl(U)$};
      \draw[gray!55,dashed,line width=0.8pt] plot[domain=0:1,samples=61,variable=\t] ({1.4+3*(3*\t*\t-2*\t*\t*\t)},{0.55+8.85*\t*(1-\t)});
      \draw[blue!70!black,line width=2.1pt,midarrow=0.5] plot[domain=0:0.6,samples=61,variable=\t] ({1.4+3*(3*\t*\t-2*\t*\t*\t)},{0.55+8.85*\t*(1-\t)});
      \draw[blue!70!black,fill=blue!70!black] (1.4,0.55) circle (2pt);
      \draw[gray!65,fill=white] (4.4,0.55) circle (2pt);
      \node[below] at (1.4,0.43) {$\incl(g)$};
      \node[below] at (4.4,0.43) {$\incl(r(g))$};
      \node[gray!70!black,anchor=west] at (5.1,1.15) {$\incl(X)$};
      \fill[blue!70!black] (3.344,2.674) circle (2pt);
      \node[above right] at (3.344,2.674) {$\pi(g,a)$};
    \end{tikzpicture}
    \caption{$g\in U,\quad r(g)\notin U$}
  \end{subfigure}

  \medskip
  \begin{subfigure}[t]{0.47\textwidth}
    \centering
    \begin{tikzpicture}[x=1.03cm,y=1.03cm,font=\small,line cap=round]
      \path[use as bounding box] (0,-0.30) rectangle (5.95,3.2);
      \draw[gray!60,fill=gray!3] (0.1,-0.1) -- (5.2,-0.1) -- (5.85,1.35) -- (0.75,1.35) -- cycle;
      \fill[blue!12] (4.4,0.55) ellipse (0.94 and 0.43);
      \node[blue!70!black] at (4.98,0.57) {$\incl(U)$};
      \draw[gray!55,dashed,line width=0.8pt] plot[domain=0:1,samples=61,variable=\t] ({1.4+3*(3*\t*\t-2*\t*\t*\t)},{0.55+8.85*\t*(1-\t)});
      \draw[blue!70!black,line width=2.1pt,midarrow=0.5] plot[domain=0.4:1,samples=61,variable=\t] ({1.4+3*(3*\t*\t-2*\t*\t*\t)},{0.55+8.85*\t*(1-\t)});
      \draw[gray!65,fill=white] (1.4,0.55) circle (2pt);
      \draw[blue!70!black,fill=blue!70!black] (4.4,0.55) circle (2pt);
      \node[below] at (1.4,0.43) {$\incl(g)$};
      \node[below] at (4.4,0.43) {$\incl(r(g))$};
      \node[gray!70!black,anchor=west] at (5.1,1.15) {$\incl(X)$};
      \fill[blue!70!black] (2.456,2.674) circle (2pt);
      \node[above left] at (2.456,2.674) {$\pi(g,b)$};
    \end{tikzpicture}
    \caption{$g\notin U,\quad r(g)\in U$}
  \end{subfigure}
  \hfill
  \begin{subfigure}[t]{0.47\textwidth}
    \centering
    \begin{tikzpicture}[x=1.03cm,y=1.03cm,font=\small,line cap=round]
      \path[use as bounding box] (0,-0.30) rectangle (5.95,3.2);
      \draw[gray!60,fill=gray!3] (0.1,-0.1) -- (5.2,-0.1) -- (5.85,1.35) -- (0.75,1.35) -- cycle;
      \fill[blue!12] (2.9,0.55) ellipse (0.90 and 0.43);
      \node[blue!70!black] at (2.9,0.55) {$\incl(U)$};
      \draw[gray!55,dashed,line width=0.8pt,midarrow=0.5] plot[domain=0:1,samples=61,variable=\t] ({1.4+3*(3*\t*\t-2*\t*\t*\t)},{0.55+8.85*\t*(1-\t)});
      \draw[gray!65,fill=white] (1.4,0.55) circle (2pt);
      \draw[gray!65,fill=white] (4.4,0.55) circle (2pt);
      \node[below] at (1.4,0.43) {$\incl(g)$};
      \node[below] at (4.4,0.43) {$\incl(r(g))$};
      \node[gray!70!black,anchor=west] at (5.1,1.15) {$\incl(X)$};
      \node[above] at (2.9,2.85) {$\pi(\setof{g}\times[0,1])$};
    \end{tikzpicture}
    \caption{$g\notin U,\quad r(g)\notin U$}
  \end{subfigure}
  \caption{The set $\Wab(U)$ for $G=\setof{g}$, shown by the shaded region $\incl(U)$ and the blue portions of the attached interval.}
  \label{fig:wab}
\end{figure}

\begin{lemma}
  \label{lem:Wab}
  For every $U\subseteq X$, the following statements hold.
  \begin{enumerate}[(i)]
    \item $\incl^{-1}(\Wab(U))=U$ and $\incl(\Int(U))\subseteq\Int_{\susp X}(\Wab(U))$.
    \item $\omega(\Wab(U),\Phi)=\omega(\incl(U),\Phi)$.
    \item If $A\subseteq X$ is forward invariant under $\cH$ and $A\subseteq\Int(U)$, then $\susp(A)\subseteq\Int_{\susp X}(\Wab(U))$.
  \end{enumerate}
\end{lemma}
\begin{proof}
For (i), since $\pi(g,0)=\incl(g)$ and $\pi(g,1)=\incl(r(g))$, the definition of $\Wab(U)$ gives $\incl^{-1}(\Wab(U))=U$. Moreover, $\Int(U)\cap G$ and $r^{-1}(\Int(U))$ are open in $G$, so, since $a>0$ and $b<1$, the quotient topology gives $\incl(\Int(U))\subseteq\Int_{\susp X}(\Wab(U))$.

For (ii), the finite-union property of $\omega$ gives
\[
  \omega(\Wab(U),\Phi)=\omega(\incl(U),\Phi)\cup\omega\!\left(\pi((U\cap G)\times[0,a]),\Phi\right)\cup\omega\!\left(\pi(r^{-1}(U)\times[b,1]),\Phi\right).
\]
Since $\pi(g,s)=\Phi(s,\incl(g))$ for $g\in U\cap G$ and $s\in[0,a]$, finite-time invariance and monotonicity give $\omega(\pi((U\cap G)\times[0,a]),\Phi)\subseteq\omega(\incl(U),\Phi)$. If $g\in r^{-1}(U)$ and $s\in[b,1]$, then $\Phi(1-s,\pi(g,s))=\incl(r(g))\in\incl(U)$ and $0\leq1-s\leq1-b$. Thus every point of $\pi(r^{-1}(U)\times[b,1])$ enters $\incl(U)$ in uniformly bounded time, so $\omega(\pi(r^{-1}(U)\times[b,1]),\Phi)\subseteq\omega(\incl(U),\Phi)$. Therefore $\omega(\Wab(U),\Phi)=\omega(\incl(U),\Phi)$.

For (iii), forward invariance gives $r(A\cap G)\subseteq A\subseteq\Int(U)$. Thus, for $g\in A\cap G$, $O:=(\Int(U)\cap G)\cap r^{-1}(\Int(U))$ is a neighborhood of $g$ in $G$. Since $b<a$, the definition of $\Wab(U)$ gives $\pi(O\times[0,1])\subseteq\Wab(U)$, and hence the quotient topology gives $\pi(\{g\}\times[0,1])\subseteq\Int_{\susp X}(\Wab(U))$. Together with (i), this yields $\susp(A)\subseteq\Int_{\susp X}(\Wab(U))$.
\end{proof}

\begin{theorem}
  \label{thm:trapping-equivalence}
  If $U$ is a trapping region for $\cH$, then $\Wab(U)$ is a trapping region for $\Phi$. Conversely, if $\tU$ is a trapping region for $\Phi$, then $\incl^{-1}(\tU)$ is a trapping region for $\cH$.
\end{theorem}
\begin{proof}
  Suppose that $U$ is a trapping region for $\cH$. Forward invariance implies $r(U\cap G)\subseteq U$, so $\Wab(U)=\susp(U)\cup\pi(r^{-1}(U)\times[b,1])$. Then Proposition~\ref{prop:HybridForwardInvariantSet} and the definition of $\Phi$ imply that $\Wab(U)$ is forward invariant.

  Since $U$ is a trapping region, $A:=\omega(U,\cH)$ satisfies $A\subseteq\Int_X(U)$ and is forward invariant by Proposition~\ref{prop:hybrid-omega-calculus}. Lemma~\ref{lem:Wab}(ii), Proposition~\ref{prop:omega-preimage-identity}, and Lemma~\ref{lem:omega-preimage} therefore yield
  \[
    \omega(\Wab(U),\Phi)=\susp(A)\subseteq\Int_{\susp X}(\Wab(U)),
  \]
  where the inclusion follows from Lemma~\ref{lem:Wab}(iii). Thus $\Wab(U)$ is a trapping region for $\Phi$.

  Conversely, suppose that $\tU$ is a trapping region for $\Phi$ and define $U:=\incl^{-1}(\tU)$. By Proposition~\ref{prop:SemiflowForwardInvariantSet}, $U$ is forward invariant, while Proposition~\ref{prop:omega-preimage-identity} implies
  \[
      \omega(U,\cH)=\incl^{-1}\left(\omega(\incl(U),\Phi)\right)
      \subseteq\incl^{-1}\left(\omega(\tU,\Phi)\right)
      \subseteq\incl^{-1}(\Int_{\susp X}(\tU))
      \subseteq\Int_X(U).
  \]
  Thus $U$ is a trapping region for $\cH$.
\end{proof}

Note that the set $\Wab(U)$ cannot simply be replaced by $\susp(U)$.
\begin{example}
  \label{ex:suspension-trapping-region}
  Let $X=\{x_0\}\sqcup[0,1]$, let $G=\{1\}$ and $r(1)=x_0$, fix $x_0$ under the local semiflow, and let $\varphi(t,x)=x+t$ for $x\in[0,1)$ and $0\leq t<1-x$. The set $U:=\{x_0\}$ is a hybrid trapping region because $x_0$ is isolated and fixed. Since $U\cap G=\emptyset$,
  \[
    \susp(U)=\{\incl(x_0)\},\quad \pi(1,1)=\incl(x_0).
  \]
  Every neighborhood of $\incl(x_0)$ in $\susp X$ contains points $\pi(1,s)$ with $s<1$, so $\susp(U)$ has empty interior and is not a trapping region for $\Phi$. Note that $\Wab(U)=\pi(\{1\}\times[b,1])$.
  \begin{center}
    \begin{tikzpicture}[x=1cm,y=1cm,>=stealth]
      \node[anchor=east] at (-0.45,0) {$X$};
      \draw[thick] (0,0) -- (4,0);
      \draw[->,thick] (1.4,0) -- (2.5,0);
      \node[above] at (2,0.08) {$\varphi$};
      \fill (0,0) circle (1.6pt);
      \fill (4,0) circle (1.6pt);
      \node[below] at (0,-0.05) {$0$};
      \node[below] at (4,-0.05) {$1$};
      \fill[blue!65!black] (9,0) circle (2.5pt);
      \node[above] at (9,0.08) {$x_0$};
      \node[below] at (9,-0.05) {$U=\{x_0\}$};
      \draw[->,dashed] (4.1,0.14) to[bend left=22] node[above] {$r$} (8.85,0.14);

      \node[anchor=east] at (-0.45,-2) {$\susp X$};
      \draw[thick] (0,-2) -- (4,-2);
      \draw[thick,gray] (4,-2) -- (9,-2);
      \draw[blue!65!black,line width=3pt] (7.1,-2) -- (9,-2);
      \draw[->,thick] (1.4,-2) -- (2.5,-2);
      \draw[->,thick,gray] (5,-2) -- (6,-2);
      \draw[->,thick] (7.6,-2) -- (8.4,-2);
      \fill (0,-2) circle (1.6pt);
      \fill (4,-2) circle (1.6pt);
      \fill[blue!65!black] (9,-2) circle (2.5pt);
      \draw (7.1,-2.08) -- (7.1,-1.92);
      \node[above] at (2,-1.6) {$\incl([0,1])$};
      \node[above] at (6,-1.6) {$\pi(\{1\}\times[0,1])$};
      \node[below] at (0,-2.08) {$\incl(0)$};
      \node[below] at (4,-2.08) {$\incl(1)$};
      \node[below] at (7.1,-2.08) {$\pi(1,b)$};
      \node[below] at (9,-2.08) {$\incl(x_0)$};
      \node[above] at (9,-1.15) {$\susp(U)=\{\incl(x_0)\}$};
      \draw[thin] (9,-1.2) -- (9,-1.82);
      \draw[blue!65!black] (7.1,-2.65) -- (7.1,-2.75) -- (9,-2.75) -- (9,-2.65);
      \node[below] at (8.05,-2.78) {$\Wab(U)$};
    \end{tikzpicture}
  \end{center}
\end{example}

\subsection{Attractors}

\begin{theorem}
  \label{thm:main-correspondence}
  The map $\incl^{-1}:\sAtt(\susp X,\Phi)\to\sAtt(\cH)$ is a bounded lattice isomorphism with inverse $A\mapsto\susp(A)$.
\end{theorem}
\begin{proof}
  Let $\tA=\omega(\tU,\Phi)$ for some attracting neighborhood $\tU\subseteq\susp X$ and define $A:=\incl^{-1}(\tA)$ and $U:=\incl^{-1}(\tU)$. Since $\incl(U)\subseteq\tU$, Proposition~\ref{prop:omega-preimage-identity} yields $\omega(U,\cH)\subseteq A$. Conversely, $A\subseteq\incl^{-1}(\Int_{\susp X}\tU)\subseteq\Int_X(U)$ and Lemma~\ref{lem:omega-preimage} yields $\omega(A,\cH)=A$, so monotonicity implies that $A\subseteq\omega(U,\cH)$. Therefore $U$ is an attracting neighborhood with $\omega(U,\cH)=A$.

  Conversely, let $A=\omega(U,\cH)$ for some attracting neighborhood $U\subseteq X$ and define $\tU:=\Wab(U)$ and $\tA:=\susp(A)$. Lemma~\ref{lem:Wab}(ii), Proposition~\ref{prop:omega-preimage-identity}, and Lemma~\ref{lem:omega-preimage} imply that
  \[
    \omega(\tU,\Phi)=\tA.
  \]
  Since $A$ is forward invariant and $A\subseteq\Int_X(U)$, Lemma~\ref{lem:Wab}(iii) yields $\tA\subseteq\Int_{\susp X}(\tU)$. Thus $\tU$ is an attracting neighborhood with attractor $\tA$. The maps $\incl^{-1}$ and $\susp$ are inverse by Proposition~\ref{prop:suspension-identities}(i) and Lemma~\ref{lem:omega-preimage}. Since both maps are monotone, they form an order isomorphism and therefore preserve joins, meets, and bounds.
\end{proof}

\subsection{Regularized preimages}

For $\tE\subseteq\susp X$, define its \emph{regularized preimage} by 
\[
  \rho(\tE):=\cl_X\left(\Int_X(\incl^{-1}(\tE))\right),
\]
and denote the regularization map in the suspension space by $\rhosusp:\cP(\susp X)\to\scR(\susp X)$, that is,
\[
  \rhosusp(\tE):=\cl_{\susp X}\Int_{\susp X}(\tE).
\]

\begin{lemma}
  \label{lem:regularized-preimage-joins}
  The map $\rho:\scR(\susp X)\to\scR(X)$ preserves joins and bounds.
\end{lemma}
\begin{proof}
Let $\tE,\tF\in\scR(\susp X)$ and set $E:=\incl^{-1}(\tE)$ and $F:=\incl^{-1}(\tF)$. Since $\incl$ is continuous, $E$ and $F$ are closed in $X$. So, $\rho(\tE\vee\tF)=\cl_X\Int_X(E\cup F)=\cl_X\Int_X(E)\cup\cl_X\Int_X(F) =\rho(\tE)\vee\rho(\tF).$ Also, $\rho(\emptyset)=\emptyset$ and $\rho(\susp X)=X$. Therefore $\rho$ preserves joins and bounds.
\end{proof}

\begin{lemma}
  \label{lem:regularized-preimage-meets}
  For closed sets $\tC,\tC'\subseteq\susp X$,
  \[
    \rho(\tC)\wedge\rho(\tC')=\rho(\tC\cap\tC').
  \]
  In particular, $\rho(\tU\wedge\tV)\subseteq\rho(\tU)\wedge\rho(\tV)$ for all $\tU,\tV\in\scR(\susp X)$.
\end{lemma}
\begin{proof}
The first statement follows from the continuity of $\incl$ and the meet operation for regular closed sets. For $\tU,\tV\in\scR(\susp X)$, since $\tU\wedge\tV\subseteq\tU\cap\tV$, the definition of $\rho$ and the first statement give $\rho(\tU\wedge\tV)\subseteq\rho(\tU\cap\tV) =\rho(\tU)\wedge\rho(\tV).$
\end{proof}

\begin{lemma}
  \label{lem:regularized-preimage-boundary}
  The sets $\tE\in\scR(\susp X)$ satisfying
  \begin{equation}
    \Int_X\incl^{-1}\left(\tE\setminus\Int_{\susp X}(\tE)\right)=\emptyset
    \label{eq:boundary-preimage}
  \end{equation}
  form a bounded sublattice on which $\rho$ is a bounded lattice homomorphism. Condition~\eqref{eq:boundary-preimage} is equivalent to
  \[
    \rho(\tE)=\cl_X\left(\incl^{-1}(\Int_{\susp X}(\tE))\right).
  \]
  Moreover, every $\tE\in\scR(\susp X)$ satisfies
  \[
    \Int_X\incl^{-1}\left(\tE\setminus\Int_{\susp X}(\tE)\right)\subseteq\Int_X(G\cup r(G)).
  \]
\end{lemma}
\begin{proof}
For a closed set $\tC\subseteq\susp X$, set $C:=\incl^{-1}(\tC), C_0:=\incl^{-1}(\Int_{\susp X}(\tC)).$ Since $C_0$ is open in $X$ and $C_0\subseteq C$, we have $C_0\subseteq\Int_X(C)$, and $C\setminus C_0=\incl^{-1}\left(\tC\setminus\Int_{\susp X}(\tC)\right).$ Thus $\Int_X(C\setminus C_0)=\emptyset \Longleftrightarrow \Int_X(C)\subseteq\cl_X(C_0).$
Thus we have $\cl_X\Int_X(C)=\cl_X(C_0)$, and by the definition of $\rho$, condition~\eqref{eq:boundary-preimage} is therefore equivalent to $\rho(\tC)=\cl_X\left(\incl^{-1}(\Int_{\susp X}(\tC))\right).$

Let $\tE,\tF$ satisfy~\eqref{eq:boundary-preimage}. Since
\[
(\tE\cup\tF)\setminus\Int_{\susp X}(\tE\cup\tF) \subseteq (\tE\setminus\Int_{\susp X}\tE)\cup (\tF\setminus\Int_{\susp X}\tF),
\]
and similarly
\[
(\tE\cap\tF)\setminus\Int_{\susp X}(\tE\cap\tF) \subseteq (\tE\setminus\Int_{\susp X}\tE)\cup (\tF\setminus\Int_{\susp X}\tF),
\]
the preimages of the sets on the left have empty interior, since the two preimages on the right are closed with empty interior. Hence both $\tE\cup\tF$ and $\tE\cap\tF$ satisfy~\eqref{eq:boundary-preimage}. In particular, Lemma~\ref{lem:regularized-preimage-joins} implies that $\tE\vee\tF=\tE\cup\tF$ satisfies the condition. Now, using  Lemma~\ref{lem:regularized-preimage-meets},
\[
\begin{aligned}
\rho(\tE)\wedge\rho(\tF)
&=\rho(\tE\cap\tF)
=\cl_X\incl^{-1}\!\left(\Int_{\susp X}(\tE\cap\tF)\right)\\
&\subseteq\cl_X\incl^{-1}\!\left(\Int_{\susp X}(\tE\wedge\tF)\right)
\subseteq\rho(\tE\wedge\tF)
\subseteq\rho(\tE)\wedge\rho(\tF),
\end{aligned}
\]
where $\Int_{\susp X}(\tE\cap\tF)\subseteq\Int_{\susp X}(\tE\wedge\tF)$. Thus $\rho(\tE\wedge\tF)=\rho(\tE)\wedge\rho(\tF)$, and $\tE\wedge\tF$ also satisfies~\eqref{eq:boundary-preimage}. Since $\emptyset$ and $\susp X$ clearly satisfy the condition, they form a bounded sublattice. Lemma~\ref{lem:regularized-preimage-joins} and the equality above show that $\rho$ restricts to a bounded lattice homomorphism.

Finally, let $B:=\incl^{-1}(\tE\setminus\Int_{\susp X}(\tE))$. Away from $G\cup r(G)$, the quotient construction makes $\incl$ locally an open embedding. Thus, if $x\in\Int_X(B)\setminus(G\cup r(G))$, some open neighborhood $V$ of $x$ satisfies $\incl(V)\subseteq\tE\setminus\Int_{\susp X}(\tE)$ with $\incl(V)$ open in $\susp X$, a contradiction. Hence $\Int_X(B)\subseteq G\cup r(G)$, and since $\Int_X(B)$ is open, $\Int_X\incl^{-1}\left(\tE\setminus\Int_{\susp X}(\tE)\right) \subseteq\Int_X(G\cup r(G)).$
\end{proof}

Thus every element of $\scR(\susp X)$ satisfies~\eqref{eq:boundary-preimage} if $\Int_X(G\cup r(G))=\emptyset$.

When $\Int_X(G\cup r(G))\neq\emptyset$, $\rho$ may fail to preserve meets.
\begin{example}
  \label{ex:regularized-preimage-meet}
  Let $X=G=\setof{g}$ and $r(g)=g$, so that $\susp X \cong S^1$. Define
  \[
    \tU:=\pi\left(\setof{g}\times[0,1/3]\right),
    \quad
    \tV:=\pi\left(\setof{g}\times[2/3,1]\right).
  \]
  Note that $\tU$ and $\tV$ are regular closed and satisfy $\tU\cap\tV=\setof{\incl(g)}$. Note that
  \[
    \tU\wedge\tV
    =\rhosusp(\tU\cap\tV)
    =\rhosusp\left(\setof{\incl(g)}\right)
    =\emptyset.
  \]
  Together with $\rho(\tU)=\rho(\tV)=X$, this yields
  \[
    \rho(\tU\wedge\tV)=\emptyset\neq X=\rho(\tU)\wedge\rho(\tV).
  \]
\end{example}
\begin{lemma}
  \label{lem:regularized-preimage-lattice}
  The map $\rho:\scR(\susp X)\to\scR(X)$ is a lattice epimorphism if and only if $\Int_X(G\cup r(G))=\emptyset$.
\end{lemma}
\begin{proof}
We first note that $\rho$ is surjective. Consider a compatible metric $d$ on $\susp X$. Let $U\in\scR(X)$ and set $O:=\Int_X(U)$. For $U\neq\emptyset,X$, define
\[
  W:=\setdef{y\in\susp X}
  {d(y,U)<d(y,X\setminus O)}.
\]
Then $W$ is open, $W\cap X=O$, and $\cl_{\susp X}(W)\cap X=U$: indeed, $O\subseteq W$ and $\cl_X(O)=U$, while if $x\in X\setminus U$, then $d(x,U)>0=d(x,X\setminus O)$, so a neighborhood of $x$ is disjoint from $W$. Hence $\tU:=\cl_{\susp X}(W)\in\scR(\susp X)$ satisfies $\incl^{-1}(\tU)=U$, and therefore $\rho(\tU)=U$. The cases $U=\emptyset,X$ are immediate. Suppose now that $\Int_X(G\cup r(G))=\emptyset$. By Lemma~\ref{lem:regularized-preimage-boundary}, every element of $\scR(\susp X)$ satisfies~\eqref{eq:boundary-preimage}, so $\rho$ preserves meets. Together with Lemma~\ref{lem:regularized-preimage-joins} and surjectivity, $\rho$ is a lattice epimorphism.

Conversely, suppose that $\rho$ is a lattice epimorphism. Since $\scR(\susp X)$ and $\scR(X)$ are Boolean algebras, so joins, meets, and bounds are perserved, and then complements also do. Thus, for $\tE\in\scR(\susp X)$, $\rho(\tE)=\rho(\tE^*)^*.$ Set $C:=\incl^{-1}(\tE^*)$. Since $C$ is closed and $X\setminus C=\incl^{-1}(\Int_{\susp X}(\tE))$, we obtain
\[
  \rho(\tE) =\bigl(\cl_X\Int_X(C)\bigr)^* =\cl_X(X\setminus C) =\cl_X\incl^{-1}(\Int_{\susp X}(\tE)).
\]
Lemma~\ref{lem:regularized-preimage-boundary} therefore implies that every $\tE\in\scR(\susp X)$ satisfies~\eqref{eq:boundary-preimage}. Suppose, to the contrary, that $\Int_X(G\cup r(G))\neq\emptyset$. Since $G$ and $r(G)$ are closed, either $\Int_X(G)\neq\emptyset$ or $\Int_X(r(G))\neq\emptyset$. Assume first that $\Int_X(G)\neq\emptyset$ and choose a nonempty open set $O\subseteq\Int_X(G)$. Let
\[
  J:=\bigcup_{n\geq1}(2^{-2n},2^{-2n+1}), \text{ such that } \tE:=\cl_{\susp X}\pi(O\times J).
\]
Since $\pi$ restricts to a homeomorphism from $G\times(0,1)$ onto an open subset of $\susp X$, the set $\pi(O\times J)$ is open and hence $\tE$ is regular closed. Both $J$ and $(0,1)\setminus\cl J$ accumulate at $0$, so for every $x\in O$, $\incl(x)\in\tE\setminus\Int_{\susp X}(\tE).$ Thus $O\subseteq \incl^{-1}\left(\tE\setminus\Int_{\susp X}(\tE)\right),$ contradicting~\eqref{eq:boundary-preimage}. 
A similar argument works at $s=1$ if $\Int_X(r(G))\neq\emptyset$, by taking a nonempty open set $O\subseteq\Int_X(r(G))$ and $\tE:=\cl_{\susp X}\pi\left(r^{-1}(O)\times(1-J)\right).$ Therefore $\Int_X(G\cup r(G))=\emptyset$.
\end{proof}

\begin{lemma}
  \label{lem:regularized-preimage-omega}
  If $\tU\in\sANbhd(\susp X,\Phi)$, then $\rho(\tU)\in\sANbhdR(X,\cH)$. Moreover, $\omega(\rho(\tU),\cH)=\incl^{-1}\left(\omega(\tU,\Phi)\right)$.
\end{lemma}

\begin{proof}
Let $\tA:=\omega(\tU,\Phi), A:=\incl^{-1}(\tA), U:=\incl^{-1}(\tU).$  Since $\tU$ is attracting, $\omega(\tU,\Phi) \subseteq \Int_{\susp X}\tU.$ By Proposition~\ref{prop:omega-preimage-identity},
\[
\begin{aligned}
  \omega(\rho(\tU),\cH) &=\incl^{-1}\left(\omega(\incl(\rho(\tU)),\Phi)\right) \subseteq \incl^{-1}\left(\omega(\cl_{\susp X}\tU,\Phi)\right) =\incl^{-1}\left(\omega(\tU,\Phi)\right)\\
  &\subseteq\incl^{-1}(\Int_{\susp X}\tU)
  \subseteq\Int_X(U)
  \subseteq\Int_X(\rho(\tU)).
\end{aligned}
\]
Thus $\rho(\tU)\in\sANbhdR(X,\cH)$. Now, by Theorem~\ref{thm:main-correspondence}, $A\in\sAtt(\cH)$, so $\omega(A,\cH)=A$. Hence monotonicity gives $A=\incl^{-1}\left(\omega(\tU,\Phi)\right)\subseteq\omega(\rho(\tU),\cH),$ and therefore $ \omega(\rho(\tU),\cH) =\incl^{-1}\left(\omega(\tU,\Phi)\right).$
\end{proof}

\subsection{Attracting neighborhoods}

Recall that a \emph{section} of a map $f:E\to B$ is a map $s:B\to E$ satisfying $f\circ s=\id_B$.

\begin{proposition}
  \label{prop:ANbhd-epi}
  The map $\rho:\sANbhdR(\susp X,\Phi)\to\sANbhdR(X,\cH)$ is surjective and preserves joins and bounds. For every $0<b<a<1$, the map
  \[
    \sANbhdR(X,\cH)\to\sANbhdR(\susp X,\Phi),
    \quad
    U\longmapsto\rhosusp(\Wab(U)),
  \]
  is a section of $\rho$.
\end{proposition}
\begin{proof}
By Lemma~\ref{lem:regularized-preimage-omega}, $\rho:\sANbhdR(\susp X,\Phi)\to\sANbhdR(X,\cH)$ is well-defined. Lemma~\ref{lem:regularized-preimage-joins} implies that it preserves joins and bounds. Fix $0<b<a<1$ and let $U\in\sANbhdR(X,\cH)$. Set $A:=\omega(U,\cH), \tU:=\Wab(U), \tV:=\rhosusp(\tU).$ By Lemma~\ref{lem:Wab}(ii), Proposition~\ref{prop:omega-preimage-identity}, and Lemma~\ref{lem:omega-preimage}, $\omega(\tU,\Phi)=\susp(A).$ Since $A\subseteq\Int_X(U)$ is forward invariant, Lemma~\ref{lem:Wab}(iii) gives $\susp(A)\subseteq\Int_{\susp X}(\tU).$ Moreover, $U$ is closed, hence $\tU$ is closed, so $\tV=\cl_{\susp X}\Int_{\susp X}(\tU)\subseteq\tU.$ Thus
\[
  \omega(\tV,\Phi)
  \subseteq\omega(\tU,\Phi)
  =\susp(A)
  \subseteq\Int_{\susp X}(\tU)
  \subseteq\Int_{\susp X}(\tV).
\]
Hence $\tV\in\sANbhdR(\susp X,\Phi)$.

By Lemma~\ref{lem:Wab}(i), $\incl^{-1}(\tU)=U, \incl(\Int_X(U))\subseteq\Int_{\susp X}(\tU).$ Since $U=\cl_X\Int_X(U)$, $\incl(U) \subseteq\cl_{\susp X}\incl(\Int_X(U)) \subseteq\cl_{\susp X}\Int_{\susp X}(\tU)=\tV.$ Together with $\tV\subseteq\tU$, this gives $U\subseteq\incl^{-1}(\tV) \subseteq\incl^{-1}(\tU)=U.$ Therefore $\incl^{-1}(\tV)=U$, and hence
\[
  \rho(\tV)
  =\cl_X\Int_X\left(\incl^{-1}(\tV)\right)
  =\cl_X\Int_X(U)
  =U.
\]
Thus $\rho\circ(\rhosusp\circ\Wab)=\id_{\sANbhdR(X,\cH)},$ $U\mapsto\rhosusp(\Wab(U))$ is a section of $\rho$, and in particular $\rho$ is surjective.
\end{proof}

\begin{remark}
  For arbitrary attracting neighborhoods, $\tU\mapsto\incl^{-1}(\tU)$ preserves unions and intersections and has section $U\mapsto\Wab(U)$. We use $\rho$ to remain in the regular closed lattice. Note that $\rho$ may fail to be injective, and its section $U\mapsto\rhosusp(\Wab(U))$ may fail to preserve meets.
\end{remark}

Even for regular closed attracting neighborhoods, $\rho$ may fail to preserve meets, as the following example shows.
\begin{example}
  \label{ex:ANbhd-preimage-meet}
  Let $X=\setof{0,1,2}$, $G=\setof{0,1}$, $r(0)=1$, and $r(1)=2$, and let $\varphi(t,2)=2$ for every $t\geq0$. Note that $\cH=(X,\varphi,G,r)$ satisfies the trapping guard condition. Note that $\susp X \cong [0,2]$ so that $\incl(k)=k$ for $k\in X$ and $\Phi(t,s)=\min\setof{s+t,2}$. Observe that the sets
  \[
    \tU:=[0,1]\cup[3/2,2],
    \quad
    \tV:=[1,2]
  \]
  are regular closed attracting neighborhoods, since $\omega(\tU,\Phi)=\omega(\tV,\Phi)=\setof{2}\subseteq\Int_{\susp X}(\tU\cap\tV)$, and their intersection is $\setof{1}\cup[3/2,2]$, so
  \[
    \tU\wedge\tV=\rhosusp\left(\setof{1}\cup[3/2,2]\right)=[3/2,2].
  \]
  Finally, note that $\rho(\tU)=X$ and $\rho(\tV)=\setof{1,2}$, while
  \[
    \rho(\tU\wedge\tV)=\setof{2}\neq\setof{1,2}=\rho(\tU)\wedge\rho(\tV).
  \]
\end{example}

\begin{theorem}
  \label{thm:ANbhd-meets}
  The following statements are equivalent.
  \begin{enumerate}[(i)]
    \item The map $\rho:\sANbhdR(\susp X,\Phi)\surjection\sANbhdR(X,\cH)$ is a bounded lattice epimorphism.
    \item Every $\tU\in\sANbhdR(\susp X,\Phi)$ satisfies~\eqref{eq:boundary-preimage}.
    \item For every $A\in\sAtt(\cH)$,
    \[
      \Int_X(G\cup r(G))\cap\setdef{x\in X}{\omega(x,\cH)\subseteq A}\subseteq A.
    \]
  \end{enumerate}
\end{theorem}
\begin{proof}
By Proposition~\ref{prop:ANbhd-epi}, the map $\rho$ is already surjective and preserves joins and bounds. Thus (i) is equivalent to preservation of meets.

First, $(ii)\Rightarrow(i)$ is immediate by Lemma~\ref{lem:regularized-preimage-boundary}.

Next, we prove $(iii)\Rightarrow(ii)$. Let $\tU\in\sANbhdR(\susp X,\Phi)$ and set $\tA:=\omega(\tU,\Phi), A:=\incl^{-1}(\tA).$ By Theorem~\ref{thm:main-correspondence}, $A\in\sAtt(\cH)$. Let $B:= \Int_X\incl^{-1} (\tU\setminus\Int_{\susp X}(\tU)).$ By Lemma~\ref{lem:regularized-preimage-boundary}, $B\subseteq\Int_X(G\cup r(G)).$
If $x\in B$, then $\incl(x)\in\tU$, and therefore
Proposition~\ref{prop:omega-preimage-identity} $\omega(x,\cH) \subseteq \incl^{-1}(\omega(\tU,\Phi))=A.$ Condition~(iii) therefore implies $x\in A$. On the other hand, since $\tU$ is attracting, $A =\incl^{-1}(\tA) \subseteq \incl^{-1}(\Int_{\susp X}\tU),$ which contradicts $x\in\incl^{-1}(\tU\setminus\Int_{\susp X}\tU)$. Hence $B=\emptyset$, so $\tU$ satisfies~\eqref{eq:boundary-preimage}.

Finally, suppose the contrary, so (iii) fails. Then there exist $A\in\sAtt(\cH)$ and $x\in\Int_X(G\cup r(G))\setminus A$ such that $\omega(x,\cH)\subseteq A$. Set $\tA:=\susp(A)$. By Theorem~\ref{thm:main-correspondence}, $\tA\in\sAtt(\susp X,\Phi)$, so choose $\tN_0\in\sANbhdR(\susp X,\Phi)$ with $\omega(\tN_0,\Phi)=\tA$. Since $\omega(\incl(x),\Phi)=\susp(\omega(x,\cH))\subseteq\tA\subseteq\Int_{\susp X}\tN_0$, there exist $T\geq0$ and a nonempty open neighborhood $P$ of $x$ such that $\cl_XP\subseteq\Int_X(G\cup r(G))\setminus A, \text{ and } \Phi(T,\incl(\cl_XP))\subseteq\Int_{\susp X}\tN_0.$ It follows that $\omega(\incl(\cl_XP),\Phi)\subseteq\tA$.

Since $P\subseteq G\cup r(G)$ and $G,r(G)$ are closed, the Baire category theorem gives a nonempty open set $O$ with $\cl_XO\subseteq P$ such that either $\cl_XO\subseteq G$ or $\cl_XO\subseteq r(G)$. In either case, $\omega(\incl(\cl_XO),\Phi)\subseteq\tA$. Assume first that $\cl_XO\subseteq G$. Since $\cl_XO\cap A=\emptyset$, choose $\epsilon>0$ such that $K:=\pi(\cl_XO\times[0,\epsilon])$ is disjoint from $\tA$. Choose $\tN\in\sANbhdR(\susp X,\Phi)$ such that $\omega(\tN,\Phi)=\tA$ and $\tN\cap K=\emptyset$. Define
\[
  J_0:=\bigcup_{n\geq1}(\epsilon2^{-2n},\epsilon2^{-2n+1}),
  \qquad
  J_1:=\bigcup_{n\geq1}(\epsilon2^{-2n-1},\epsilon2^{-2n}),
\]
and let $\tE_i:=\cl_{\susp X}\pi(O\times J_i)$ for $i=0,1$. Then $\tE_0,\tE_1$ are regular closed, $\incl(O)\subseteq\tE_i\setminus\Int_{\susp X}\tE_i, \text{ and } \Int_{\susp X}\tE_0\cap\Int_{\susp X}\tE_1=\emptyset.$ Moreover, $\tE_i\subseteq\Phi([0,\epsilon],\incl(\cl_XO))$, and hence $\omega(\tE_i,\Phi)\subseteq\tA$.
Set $\tU_i:=\tN\cup\tE_i$, $i=0,1$. Since $\tN\cap K=\emptyset$ and $\tE_i\subseteq K$, we have $\tU_0\wedge\tU_1=\tN$. Also,
\[
  \omega(\tU_i,\Phi)
  =\omega(\tN,\Phi)\cup\omega(\tE_i,\Phi)
  =\tA
  \subseteq\Int_{\susp X}\tN
  \subseteq\Int_{\susp X}\tU_i,
\]
so $\tU_i\in\sANbhdR(\susp X,\Phi)$.

Since $\incl(O)\subseteq\tU_i$ and $O$ is open, we have $O\subseteq\rho(\tU_i)$, and therefore $O\subseteq\rho(\tU_0)\wedge\rho(\tU_1).$ On the other hand, $\tN\cap\incl(O)=\emptyset$, so $O\cap\rho(\tN)=\emptyset$. Thus $\rho(\tU_0\wedge\tU_1) =\rho(\tN) \neq \rho(\tU_0)\wedge\rho(\tU_1),$ contradicting (i). If instead $\cl_XO\subseteq r(G)$, the same argument is applied at the $s=1$, using $\tE_i:=\cl_{\susp X}\pi(r^{-1}(O)\times(1-J_i))$. Hence (iii) holds.
\end{proof}

\begin{corollary}
  \label{cor:ANbhd-epi-nowhere-dense}
  If $\Int_X(G\cup r(G))=\emptyset$, the map $\rho:\sANbhdR(\susp X,\Phi)\surjection\sANbhdR(X,\cH)$ is a bounded lattice epimorphism.
\end{corollary}
\begin{proof}
  The map $\rho$ is surjective and preserves joins and bounds by Proposition~\ref{prop:ANbhd-epi}, and it preserves meets by Lemma~\ref{lem:regularized-preimage-lattice}.
\end{proof}

\begin{corollary}
  \label{cor:ANbhd-epi-necessary}
  If $\rho:\sANbhdR(\susp X,\Phi)\to\sANbhdR(X,\cH)$ preserves meets, then $\Int_X(G)\subseteq r(G)$ and $\Int_X(G\cup r(G))\subseteq\omega(X,\cH)$.
\end{corollary}
\begin{proof}
By Proposition~\ref{prop:ANbhd-epi} and the hypothesis, $\rho$ is a bounded lattice epimorphism, so Theorem~\ref{thm:ANbhd-meets}(iii) applies. Taking $A=\omega(X,\cH)$ and using monotonicity gives $\Int_X(G\cup r(G))\subseteq\omega(X,\cH).$ Moreover, by the definition of hybrid trajectories, $\Int_X(G)\cap\omega(X,\cH)\subseteq r(G),$ and hence $\Int_X(G)\subseteq\omega(X,\cH)\cap\Int_X(G)\subseteq r(G).$
\end{proof}

The converse of Corollary~\ref{cor:ANbhd-epi-nowhere-dense} fails.
\begin{example}
  \label{ex:ANbhd-epi-without-nowhere-dense}
  Let $X=[0,2]$, $G=[1,2]$, $r=\id_G$, and $\varphi(t,x)=x+t$ for $x\in[0,1)$ and $0\leq t<1-x$. Then $\stoptime(x)=\max\setof{1-x,0}$, and TGC is satisfied. Since $\omega(x,\cH)=\setof{x}$ for every $x\in G$, condition (iii) of Theorem~\ref{thm:ANbhd-meets} holds. Thus $\rho:\sANbhdR(\susp X,\Phi)\surjection\sANbhdR(X,\cH)$ is a lattice epimorphism, although $\Int_X(G\cup r(G))=(1,2]$.
\end{example}

The section in Proposition~\ref{prop:ANbhd-epi} also yields a sublattice on which $\rho$ preserves meets without requiring $\Int_X(G\cup r(G))=\emptyset$.
\begin{proposition}
  \label{prop:constructed-ANbhd-epi}
  Fix $0<b<a<1$, and let $\tcD_{a,b}$ be the sublattice of $\sANbhdR(\susp X,\Phi)$ generated by
  \[
    \setdef{\rhosusp(\Wab(U))}{U\in\sANbhdR(X,\cH)}.
  \]
  Then the restriction
  \[
    \rho|_{\tcD_{a,b}}:\tcD_{a,b}\surjection\sANbhdR(X,\cH)
  \]
  is a bounded lattice epimorphism.
\end{proposition}
\begin{proof}
  First, note that $\rhosusp(\Wab(\emptyset))=\emptyset$ and $\rhosusp(\Wab(X))=\susp X$, so $\tcD_{a,b}$ contains both bounds. Lemma~\ref{lem:regularized-preimage-joins} and the section in Proposition~\ref{prop:ANbhd-epi} imply that $\rho|_{\tcD_{a,b}}$ is surjective and preserves joins and bounds, so it only remains to prove meets. Let $U_1,\ldots,U_m\in\sANbhdR(X,\cH)$ and $\tU_j:=\rhosusp(\Wab(U_j))$. Since $\Wab(U_j)$ is closed, Lemma~\ref{lem:Wab}(i) and regular closedness of $U_j$ yield
  \[
    \incl(U_j)\subseteq\tU_j\subseteq \Wab(U_j)
    \quad \text{and} \quad
    \incl(\Int_X U_j)\subseteq\Int_{\susp X}\tU_j.
  \]
  In particular, $\incl^{-1}(\tU_j)=U_j$, and therefore
  \[
    \Int_X\left(\bigcap_{j=1}^m U_j\right)
    \subseteq\incl^{-1}\left(\rhosusp\left(\bigcap_{j=1}^m\tU_j\right)\right)
    \subseteq\bigcap_{j=1}^m U_j.
  \]
  Taking interiors and then closures in $X$ yields
  \[
    \rho\left(\bigwedge_{j=1}^m\tU_j\right)=\bigwedge_{j=1}^m U_j.
  \]
  Note that every element of $\tcD_{a,b}$ is a finite join of finite meets of these generators, so if $\tV=\bigvee_p\tV_p$ and $\tW=\bigvee_q\tW_q$, where each $\tV_p$ and $\tW_q$ is a finite meet of generators, then distributivity, Lemma~\ref{lem:regularized-preimage-joins}, and the identity above imply that
  \[
    \rho(\tV\wedge\tW)
    =\bigvee_{p,q}\rho(\tV_p\wedge\tW_q)
    =\bigvee_{p,q}\left(\rho(\tV_p)\wedge\rho(\tW_q)\right)
    =\rho(\tV)\wedge\rho(\tW).
  \]
  Thus $\rho|_{\tcD_{a,b}}$ preserves meets and is a bounded lattice epimorphism.
\end{proof}

\begin{corollary}
  \label{thm:omega-square}
  The following diagram commutes.
  \[
    \begin{tikzcd}[column sep=4.5em, row sep=3.0em]
      \sANbhdR(\susp X,\Phi) \arrow[r,"{\omega(\cdot,\Phi)}"] \arrow[d,"\rho"']
      & \sAtt(\susp X,\Phi) \arrow[d,"\incl^{-1}"]
      \\
      \sANbhdR(X,\cH) \arrow[r,"{\omega(\cdot,\cH)}"']
      & \sAtt(\cH)
    \end{tikzcd}
  \]
\end{corollary}

\section{Finite Realization and the Conley Index}
\label{sec:computability}

Fix a metric $d_{\susp X}$ inducing the quotient topology on $\susp X$ and a time $\tau>0$, and let $f_\tau:=\Phi(\tau,\cdot):\susp X\to\susp X$ be the $\tau$-map of $\Phi$. All diameters, distances, and neighborhoods in $\susp X$ refer to this metric. To apply the realization theorem of~\cite{Kalies:Mischaikow:Vandervorst:2016} to outer approximations of $f_\tau$, we construct a contracting cofiltration of suspension grids from grids on $X$ and dyadic subdivisions of $[0,1]$.

\subsection{Suspension grids}
\label{sec:grids}

Recall that every compact metric space admits a contracting cofiltration of grids~\cite{Kalies:Mischaikow:Vandervorst:2016}. We choose a contracting cofiltration $(\cI_n)_{n\geq0}$ of grids on $[0,1]$ as follows:
\[
  a_n:=2^{-n-2},\quad
  \Ink:=[ka_n,(k+1)a_n],\quad
  \cI_n:=\setdef{\Ink}{0\leq k<2^{n+2}}.
\]

For a grid $\cX$ on $X$ and $n\geq0$, the images $\incl(|\mu|)$ and $\pi\left((|\mu|\cap G)\times\Ink\right)$, with $\mu\in\cX$ and $\Ink\in\cI_n$, cover $\susp X$ but do not necessarily form a grid, as Example~\ref{ex:regularized-preimage-meet} shows. We therefore enlarge the base images near the identified endpoints and regularize the resulting generators before forming a grid.

\begin{definition}
  \label{def:suspension-grid}
  Let $\cX$ be a grid on $X$ and let $n\geq0$. For $\mu\in\cX$ and $1\leq k\leq 2^{n+2}-2$, define
  \begin{align*}
    \tE_n(\mu)&:=\incl(|\mu|)\cup\pi\left((|\mu|\cap G)\times[0,a_n]\right)\cup\pi\left(r^{-1}(|\mu|)\times[1-a_n,1]\right),\\
    \tK_n(\mu)&:=\rhosusp(\tE_n(\mu)),\\
    \tQ_n(\mu,k)&:=\rhosusp\left(\pi\left((|\mu|\cap G)\times\Ink\right)\right).
  \end{align*}
  The \emph{suspension grid} associated with $\cX$ and $\cI_n$, denoted $\Xi_n(\cX)$, is the grid on $\susp X$ generated by the sets $\tK_n(\mu)$ and $\tQ_n(\mu,k)$ with $\mu\in\cX$ and $1\leq k\leq 2^{n+2}-2$.
\end{definition}
Figure~\ref{fig:suspension-grid} illustrates this construction for $n=0$.

\begin{figure}
  \centering
  \begingroup
  \definecolor{gridA}{HTML}{3178A5}
  \definecolor{gridB}{HTML}{C76B35}
  \definecolor{gridC}{HTML}{21866D}
  \definecolor{gridD}{HTML}{B9638D}
  \definecolor{gridE}{HTML}{D0A02C}
  \definecolor{gridF}{HTML}{7865AA}
  \tikzset{grid midarrow/.style={postaction={decorate},decoration={markings,mark=at position 0.52 with {\arrow{>}}}}}
  \newcommand{\GridBase}[3]{
    \fill[#1] (0.15,0.1) -- (1.95,0.1) -- (2.55,1.3) -- (0.75,1.3) -- cycle;
    \fill[#2] (1.95,0.1) -- (3.75,0.1) -- (4.35,1.3) -- (2.55,1.3) -- cycle;
    \fill[#3] (3.75,0.1) -- (5.55,0.1) -- (6.15,1.3) -- (4.35,1.3) -- cycle;
    \draw[gray!65,line width=0.65pt] (0.15,0.1) -- (5.55,0.1) -- (6.15,1.3) -- (0.75,1.3) -- cycle;
    \draw[gray!65,line width=0.65pt] (1.95,0.1) -- (2.55,1.3) (3.75,0.1) -- (4.35,1.3);
  }
  \newcommand{\GridArc}[3]{
    \draw[#1] plot[domain=#2:#3,samples=81,variable=\t] ({2.25+2.7*(3*\t*\t-2*\t*\t*\t)},{0.7+10*\t*(1-\t)});
  }
  \newcommand{\GridEndpoints}[1][fill=black,draw=black]{
    \filldraw[#1] (2.25,0.7) circle (1.8pt);
    \filldraw[#1] (4.95,0.7) circle (1.8pt);
    \node[anchor=east,font=\scriptsize] at (2.05,0.58) {$\incl(g)$};
    \node[below,font=\scriptsize] at (4.95,0.55) {$\incl(r(g))$};
  }
  \newcommand{\GridCuts}{
    \foreach \t in {0.25,0.5,0.75}{\fill ({2.25+2.7*(3*\t*\t-2*\t*\t*\t)},{0.7+10*\t*(1-\t)}) circle (1.35pt);}
  }

  \begin{tikzpicture}[x=1cm,y=1cm,font=\small]
    \node[anchor=east] at (-0.25,0.14) {$\cI_0$};
    \foreach \j/\col in {0/gridD,1/gridE,2/gridF,3/gridC}{
      \fill[\col!17] (1.7*\j,0) rectangle (1.7*\j+1.7,0.28);
      \draw[\col,line width=1.4pt] (1.7*\j,0.14) -- (1.7*\j+1.7,0.14);
      \node[above,font=\scriptsize] at (1.7*\j+0.85,0.30) {$I_{0,\j}$};
    }
    \foreach \x/\label in {0/0,1.7/{\frac14},3.4/{\frac12},5.1/{\frac34},6.8/1}{
      \draw (\x,-0.04) -- (\x,0.32);
      \node[below,font=\scriptsize] at (\x,-0.08) {$\label$};
    }
  \end{tikzpicture}

  \medskip
  \begin{subfigure}[t]{0.48\textwidth}
    \centering
    \begin{tikzpicture}[x=1cm,y=1cm,font=\small,line cap=round]
      \path[use as bounding box] (-0.15,-0.55) rectangle (6.45,3.9);
      \GridBase{gridA!17}{gridB!17}{gridC!17}
      \draw[gridB,thick,dashed,grid midarrow] (2.25,0.7) .. controls (2.4,3.2) and (4.7,3.2) .. (4.95,0.7);
      \node[gridB,above] at (3.6,2.65) {$r$};
      \fill (2.25,0.7) circle (1.8pt);
      \fill (4.95,0.7) circle (1.8pt);
      \node[below] at (2.25,0.53) {$g$};
      \node[below] at (4.95,0.53) {$r(g)$};
      \node[gridA!80!black] at (1.05,-0.22) {$|\mu_1|$};
      \node[gridB!80!black] at (2.85,-0.22) {$|\mu_2|$};
      \node[gridC!70!black] at (4.65,-0.22) {$|\mu_3|$};
      \node[anchor=east] at (6.15,1.65) {$X$};
    \end{tikzpicture}
    \caption{Grid $\cX$}
  \end{subfigure}\hfill
  \begin{subfigure}[t]{0.48\textwidth}
    \centering
    \begin{tikzpicture}[x=1cm,y=1cm,font=\small,line cap=round]
      \path[use as bounding box] (-0.15,-0.55) rectangle (6.45,3.9);
      \GridBase{gridA!17}{gridB!17}{gridC!17}
      \GridArc{gray!55,dashed,line width=0.85pt}{0}{1}
      \GridArc{gridA,line width=3pt}{0}{0.25}
      \GridArc{gridB,line width=3pt,line cap=butt,dash pattern=on 3pt off 3pt}{0}{0.25}
      \GridArc{gridC,line width=2.5pt}{0.75}{1}
      \GridEndpoints
      \fill (2.671875,2.575) circle (1.4pt);
      \fill (4.528125,2.575) circle (1.4pt);
      \node[anchor=east,font=\scriptsize] at (2.51,2.70) {$s=a_0$};
      \node[anchor=west,font=\scriptsize] at (4.69,2.70) {$s=1-a_0$};
      \node[above] at (3.6,3.34) {$\pi(\{g\}\times[0,1])$};
      \node[gridA!80!black] at (1.05,-0.22) {$\tK_0(\mu_1)$};
      \node[gridB!80!black] at (2.85,-0.22) {$\tK_0(\mu_2)$};
      \node[gridC!70!black] at (4.65,-0.22) {$\tK_0(\mu_3)$};
    \end{tikzpicture}
    \caption{Generators $\tE_0(\mu_i)=\tK_0(\mu_i)$}
  \end{subfigure}

  \medskip
  \begin{subfigure}[t]{0.48\textwidth}
    \centering
    \begin{tikzpicture}[x=1cm,y=1cm,font=\small,line cap=round]
      \path[use as bounding box] (-0.15,-0.55) rectangle (6.45,3.9);
      \GridBase{gray!4}{gray!4}{gray!4}
      \GridArc{gray!55,dashed,line width=0.85pt}{0}{1}
      \GridArc{gridE,line width=2.6pt}{0.25}{0.5}
      \GridArc{gridF,line width=2.6pt}{0.5}{0.75}
      \GridCuts
      \GridEndpoints[fill=white,draw=gray!65]
      \node[gridE!75!black,anchor=east] at (3.19,3.58) {$\pi(\{g\}\times I_{0,1})$};
      \node[gridF,anchor=west] at (4.02,3.58) {$\pi(\{g\}\times I_{0,2})$};
      \node[anchor=east,font=\scriptsize] at (2.52,2.61) {$s=\frac14$};
      \node[above,font=\scriptsize] at (3.6,3.26) {$s=\frac12$};
      \node[anchor=west,font=\scriptsize] at (4.70,2.61) {$s=\frac34$};
      \node[anchor=east] at (6.15,1.65) {$\incl(X)$};
    \end{tikzpicture}
    \caption{Generators $\tQ_0(\mu_i,k)$}
  \end{subfigure}\hfill
  \begin{subfigure}[t]{0.48\textwidth}
    \centering
    \begin{tikzpicture}[x=1cm,y=1cm,font=\small,line cap=round]
      \path[use as bounding box] (-0.15,-0.55) rectangle (6.45,3.9);
      \GridBase{gridA!17}{gridB!17}{gridC!17}
      \GridArc{gridD,line width=2.6pt}{0}{0.25}
      \GridArc{gridE,line width=2.6pt}{0.25}{0.5}
      \GridArc{gridF,line width=2.6pt}{0.5}{0.75}
      \GridArc{gridC,line width=2.6pt}{0.75}{1}
      \GridCuts
      \GridEndpoints
      \node[gridD!80!black,anchor=east] at (2.10,1.94) {$\pi(\{g\}\times I_{0,0})$};
      \node[gridE!75!black,anchor=east] at (3.19,3.58) {$\pi(\{g\}\times I_{0,1})$};
      \node[gridF,anchor=west] at (4.02,3.58) {$\pi(\{g\}\times I_{0,2})$};
      \node[gridA!80!black] at (1.05,-0.22) {$\incl(|\mu_1|)$};
      \node[gridB!80!black] at (2.85,-0.22) {$\incl(|\mu_2|)$};
      \node[gridC!70!black] at (4.65,-0.22) {$\tK_0(\mu_3)$};
    \end{tikzpicture}
    \caption{Suspension grid $\Xi_0(\cX)$}
  \end{subfigure}
  \endgroup
  \caption{The suspension grid $\Xi_0(\cX)$ associated with $\cX=\{\mu_1,\mu_2,\mu_3\}$ and $\cI_0$.}
  \label{fig:suspension-grid}
\end{figure}

Now let $(\cX_n)$ be a contracting cofiltration of grids on $X$, and define the associated suspension grids recursively by
\[
  \Xi_0:=\Xi_0(\cX_0),
  \quad
  \Xi_n:=\Xi_{n-1}\wedge\Xi_n(\cX_n)\quad\forall n\geq1.
\]
Equivalently, $\Xi_n$ is generated by the sets $\tK_j(\mu)$ and $\tQ_j(\mu,k)$ with $0\leq j\leq n$, $\mu\in\cX_j$, and $1\leq k\leq 2^{j+2}-2$. Note that, by construction, $\Xi_{n+1}$ refines $\Xi_n$. However, refining the base and interval grids alone does not guarantee a refinement, since $r^{-1}(|\mu|)$ for $\mu\in\cX_n$ is not necessarily a union of the sets $|\mu'|\cap G$ with $\mu'\in\cX_{n+1}$.

We write $|\xi|_{\susp}$ for the element indexed by $\xi\in\Xi_n$ and $|\cU|_{\susp}:=\bigcup_{\xi\in\cU}|\xi|_{\susp}$ for $\cU\subseteq\Xi_n$, and let $\tcB_n:=\scR_{\Xi_n}(\susp X)$ be the Boolean algebra of unions of elements of $\Xi_n$.

For $1\leq k\leq2^{n+2}-2$, we have $\Ink\subseteq(0,1)$ and hence $\tQ_n(\mu,k)\cap\incl(X)=\emptyset$. Since $\pi$ restricts to a homeomorphism from $G\times(0,1)$ onto an open subset of $\susp X$,
\[
  \tQ_n(\mu,k)=\pi\left(\cl_G\Int_G(|\mu|\cap G)\times\Ink\right),
\]
which is empty when $|\mu|\cap G$ has empty interior in $G$. Note that the sets $|\mu|\cap G$ do not necessarily have disjoint interiors in $G$.

\begin{remark}
  \label{rem:guard-grids}
  The sets $|\mu|\cap G$ in the product generators may be replaced by the elements of any contracting cofiltration $(\cG_n)$ of grids on $G$, and the results below remain valid. If $\cG_n$ consists of the nonzero atoms of the Boolean algebra in $\scR(G)$ generated by $S_{j,\mu}:=\cl_G\Int_G(|\mu|\cap G)$ for $\mu\in\cX_j$ and $j\leq n$, this construction generates the same algebras $\tcB_n$ and hence the same suspension grids. We use $\cX_n$ as the indexing family for computational convenience.

  Indeed, every $S_{i,\mu}$ with $i\leq n$ is a union of the sets $S_{n,\mu'}$, which cover $G$. The corresponding Boolean algebras therefore increase, and their atoms form a contracting cofiltration since each atom lies in some $S_{n,\mu'}$ and has diameter at most $\diam(\cX_n)$. For each closed interval $J\subseteq(0,1)$ of positive length, the map $S\mapsto\pi(S\times J)$ preserves Boolean operations relative to $\pi(G\times J)$. Hence the images of the generators and the images of the corresponding atoms generate the same Boolean algebra, so the algebras $\tcB_n$ agree.
\end{remark}

\begin{proposition}
  \label{prop:suspension-grid}
  The sequence $(\Xi_n)_{n\geq0}$ is a contracting cofiltration of grids on $\susp X$. For each $\mu\in\cX_n$, there is a unique $\tmu\in\Xi_n$ satisfying
  \[
    \incl(\Int_X|\mu|)\subseteq\Int_{\susp X}|\tmu|_{\susp},
    \quad
    \incl^{-1}(|\tmu|_{\susp})=|\mu|.
  \]
  In particular, $d_n:\cX_n\injection\Xi_n$, $d_n(\mu)=\tmu$, is injective.
\end{proposition}
\begin{proof}
For $j\geq0$ and $\nu\in\cX_j$, the definition of $\tE_j(\nu)$ and Lemma~\ref{lem:Wab}(i) give
\[
  \incl^{-1}(\tK_j(\nu))=|\nu|,
  \qquad
  \incl(\Int_X|\nu|)\subseteq\Int_{\susp X}\tK_j(\nu),
  \qquad
  \tQ_j(\nu,k)\cap\incl(X)=\emptyset.
\]
By construction, $\Xi_{n+1}$ refines $\Xi_n$. Moreover, uniform continuity of $\incl$ and $\pi$, together with $\diam(\cX_n)\to0$ and $a_n\to0$, implies $\max_{\mu}\diam\tK_n(\mu)\to0,$ and $\max_{\mu,k}\diam\tQ_n(\mu,k)\to0.$ Since every element of $\Xi_n$ is contained in a level-$n$ generator, $\diam(\Xi_n)\to0$. Thus $(\Xi_n)$ is a contracting cofiltration.

Now, fix $\mu\in\cX_n$, and for $j\leq n$ let $\mu_j\in\cX_j$ be the unique element containing $\mu$. Then $\incl(\Int_X|\mu|)\subseteq\Int_{\susp X}\tK_j(\mu_j)$ and is disjoint from every $\tK_j(\nu)$ with $\nu\neq\mu_j$ and every $\tQ_j(\nu,k)$. Thus, for each generator $E$ of $\Xi_n$, $\incl(\Int_X|\mu|)$ is contained in either $\Int_{\susp X}E$ or $\Int_{\susp X}E^*$, and hence in the interior of a unique atom $\tmu\in\Xi_n$. Therefore
\[
  \incl(\Int_X|\mu|)\subseteq\Int_{\susp X}|\tmu|_{\susp}.
\]
Since $|\tmu|_{\susp}\subseteq\tK_n(\mu)$, we have $\incl^{-1}(|\tmu|_{\susp})\subseteq|\mu|$. Conversely, since $|\tmu|_{\susp}$ is closed and contains $\incl(\Int_X|\mu|)$, continuity of $\incl$ and $|\mu|=\cl_X\Int_X|\mu|$ give $\incl(|\mu|)\subseteq\cl_{\susp X}\incl(\Int_X|\mu|)\subseteq|\tmu|_{\susp}$. Hence $\incl^{-1}(|\tmu|_{\susp})=|\mu|$.

Finally, $d_n(\mu)=d_n(\nu)$ implies $|\mu|=\incl^{-1}(|d_n(\mu)|_{\susp})=\incl^{-1}(|d_n(\nu)|_{\susp})=|\nu|$, hence $\mu=\nu$. Thus $d_n$ is injective.
\end{proof}

\begin{proposition}
  \label{prop:finite-grid-preimage}
  The restriction
  \[
    \rho_n:=\rho|_{\tcB_n}:\tcB_n\surjection\scR_{\cX_n}(X)
  \]
  is a bounded lattice epimorphism that preserves complements and satisfies
  \begin{equation}
    \rho_n(|\cU|_{\susp})=|d_n^{-1}(\cU)|
    \quad\forall\,\cU\subseteq\Xi_n.
    \label{eq:finite-grid-preimage}
  \end{equation}
  For $m\geq n$, let $p_{mn}:\cX_m\to\cX_n$ and $q_{mn}:\Xi_m\to\Xi_n$ assign to each grid element the element containing it in $\cX_n$ and $\Xi_n$, respectively. Then
  \begin{equation}
    q_{mn}\circ d_m=d_n\circ p_{mn},
    \quad
    d_m^{-1}\circ q_{mn}^{-1}=p_{mn}^{-1}\circ d_n^{-1},
    \quad
    \rho_m|_{\tcB_n}=\rho_n.
    \label{eq:finite-grid-refinement}
  \end{equation}
\end{proposition}
\begin{proof}
We first claim that, for $\tau\in\Xi_n$,
\[
  \rho(|\tau|_{\susp})=
  \begin{cases}
    |\mu|, & \tau=d_n(\mu),\\
    \emptyset, & \tau\notin d_n(\cX_n).
  \end{cases}
\]
Indeed, if $\tau=d_n(\mu)$, Proposition~\ref{prop:suspension-grid} gives $\incl^{-1}(|\tau|_{\susp})=|\mu|$, and hence $\rho(|\tau|_{\susp})=|\mu|$. Otherwise, if $\Int_X\incl^{-1}(|\tau|_{\susp})\neq\emptyset$, then, since $\bigcup_{\mu\in\cX_n}\Int_X|\mu|$ is dense in $X$, there exist $\mu\in\cX_n$ and $x\in\Int_X|\mu|$ with $\incl(x)\in|\tau|_{\susp}$. But $\incl(x)\in\Int_{\susp X}|d_n(\mu)|_{\susp}$, so the grid property forces $\tau=d_n(\mu)$, a contradiction. Thus, using preservation of finite joins by $\rho$,
\[
  \rho_n(|\cU|_{\susp})
  =\bigvee_{\tau\in\cU}\rho(|\tau|_{\susp})
  =\bigcup_{\mu\in d_n^{-1}(\cU)}|\mu|
  =|d_n^{-1}(\cU)|.
\]
Since the evaluation maps identify $\cP(\Xi_n)$ and $\cP(\cX_n)$ with $\tcB_n$ and $\scR_{\cX_n}(X)$, respectively, and $d_n^{-1}$ preserves unions, intersections, complements, and bounds, $\rho_n$ is a bounded lattice homomorphism preserving complements. It is surjective because $d_n$ is injective.

Now let $m\geq n$ and $\mu\in\cX_m$. Since $|d_m(\mu)|_{\susp}\subseteq|q_{mn}(d_m(\mu))|_{\susp}$, monotonicity of $\rho$ and the preceding formula give
\[
  |\mu|\subseteq\rho\bigl(|q_{mn}(d_m(\mu))|_{\susp}\bigr),
\]
so $q_{mn}(d_m(\mu))=d_n(\nu)$ for the unique $\nu\in\cX_n$ containing $\mu$, namely $\nu=p_{mn}(\mu)$. Hence $q_{mn}\circ d_m=d_n\circ p_{mn}$. Taking inverse images yields $d_m^{-1}\circ q_{mn}^{-1}=p_{mn}^{-1}\circ d_n^{-1}$. Finally, $\Xi_m$ refines $\Xi_n$, so $\tcB_n\subseteq\tcB_m$. Since both $\rho_m=\rho|_{\tcB_m}$ and $\rho_n=\rho|_{\tcB_n}$, therefore $\rho_m|_{\tcB_n}=\rho_n$.
\end{proof}
Thus the maps $d_n$ are compatible with refinement or, equivalently, for every $m\geq n$, the diagram
\[
\begin{tikzcd}
\cX_m \arrow[r, "d_m", hook] \arrow[d, "p_{mn}"']
  & \Xi_m \arrow[d, "q_{mn}"] \\
\cX_n \arrow[r, "d_n"', hook]
  & \Xi_n
\end{tikzcd}
\]
commutes.

\subsection{Outer approximations}

\begin{definition}
  \label{def:suspension-multivalued-map}
  Choose $\varepsilon_n>0$ with $\varepsilon_{n+1}\leq\varepsilon_n$ and $\varepsilon_n\to0$. For each $n$, define $\cF_n:\Xi_n\mvmap\Xi_n$ by
  \[
    \cF_n(\xi):=\setdef{\xi'\in\Xi_n}{|\xi'|_{\susp}\cap B_{\varepsilon_n}(f_\tau(|\xi|_{\susp}))\neq\emptyset},
    \quad\forall\,\xi\in\Xi_n,
  \]
  and extend it to subsets by $\cF_n(\cU):=\bigcup_{\xi\in\cU}\cF_n(\xi)$ for $\cU\subseteq\Xi_n$.
\end{definition}

\begin{proposition}
  \label{prop:suspension-multivalued-map}
  The maps $\cF_n$ satisfy
  \[
    f_\tau(|\xi|_{\susp})\subseteq\Int_{\susp X}|\cF_n(\xi)|_{\susp},
    \quad\forall\,n\geq0,\quad\forall\,\xi\in\Xi_n,
  \]
  and form a convergent cofiltration of $\varepsilon_n$-minimal outer approximations of $f_\tau$.
\end{proposition}
\begin{proof}
Fix $n\geq0$, $\xi\in\Xi_n$, and $x\in|\xi|_{\susp}$. If $\xi'\notin\cF_n(\xi)$, then $|\xi'|_{\susp}\cap B_{\varepsilon_n}(f_\tau(|\xi|_{\susp}))=\emptyset$, and hence $f_\tau(x)\notin|\xi'|_{\susp}$. Consequently,
\[
  f_\tau(x)\in\susp X\setminus\bigcup_{\xi'\in\Xi_n\setminus\cF_n(\xi)}|\xi'|_{\susp}
  \subseteq\Int_{\susp X}|\cF_n(\xi)|_{\susp}.
\]

Now, for $m\geq n$, let $\xi,\xi'\in\Xi_m$ with $\xi'\in\cF_m(\xi)$. Since $|\xi|_{\susp}\subseteq|q_{mn}(\xi)|_{\susp}$, $|\xi'|_{\susp}\subseteq|q_{mn}(\xi')|_{\susp}$, and $\varepsilon_m\leq\varepsilon_n$, we have
\[
  |q_{mn}(\xi')|_{\susp}\cap
  B_{\varepsilon_n}\!\left(f_\tau(|q_{mn}(\xi)|_{\susp})\right)\neq\emptyset.
\]
Thus $q_{mn}(\xi')\in\cF_n(q_{mn}(\xi))$, so $(\cF_n)$ is a cofiltration.

Finally, by definition $\cF_n=\cF_n^{\varepsilon_n}$, and therefore $\cF_n^{\min}(\xi)\subseteq\cF_n(\xi)=\cF_n^{\varepsilon_n}(\xi)$. Since $\diam(\Xi_n)\to0$ by Proposition~\ref{prop:suspension-grid} and $\varepsilon_n\to0$, the cofiltration is convergent.
\end{proof}

The following propositions construct a convergent cofiltration of outer approximations on $(\Xi_n)$ from compatible enclosures of $f_\tau$ on a refining sequence of finite compact covers.

\begin{proposition}
  \label{prop:cover-grid-transfer}
  Let $(\tC_q)_{q\in\cC}$ be a finite cover of $\susp X$ and suppose that $\cR:\cC\mvmap\cC$ satisfies
  \[
    f_\tau(\tC_q)\subseteq\bigcup_{q'\in\cR(q)}\tC_{q'}.
  \]
  On a grid $\Xi$ of $\susp X$, define $\cF^{\cC}:\Xi\mvmap\Xi$ by
  \begin{equation}
    \cF^{\cC}(\xi):=\setdef{\xi'\in\Xi}{\exists q\in\cC,\ q'\in\cR(q),\ |\xi|_{\susp}\cap\tC_q\neq\emptyset,\ |\xi'|_{\susp}\cap\tC_{q'}\neq\emptyset}.
    \label{eq:cover-grid-transfer}
  \end{equation}
  Then $\cF^{\cC}$ is an outer approximation of $f_\tau$.
\end{proposition}
\begin{proof}
  Given $x\in|\xi|_{\susp}$, choose $q\in\cC$ and $q'\in\cR(q)$ such that $x\in\tC_q$ and $f_\tau(x)\in\tC_{q'}$. Note that if $f_\tau(x)\in|\xi'|_{\susp}$, then $|\xi'|_{\susp}\cap\tC_{q'}\neq\emptyset$, so $\xi'\in\cF^{\cC}(\xi)$. Consequently,
  \[
    f_\tau(x)\in\susp X\setminus\bigcup_{\xi'\in\Xi\setminus\cF^{\cC}(\xi)}|\xi'|_{\susp}\subseteq\Int_{\susp X}|\cF^{\cC}(\xi)|_{\susp}.
  \]
\end{proof}

\begin{proposition}
  \label{prop:cover-transfer-convergence}
  Let $(\tC_{n,q})_{q\in\cC_n}$ be finite covers of $\susp X$ by nonempty compact sets such that
  \[
    c_n:=\max_q\diam(\tC_{n,q})\to0.
  \]
  Suppose that there are $\cR_n:\cC_n\mvmap\cC_n$ and $\delta_n\geq0$ with $\delta_n\to0$ satisfying
  \[
      f_\tau(\tC_{n,q})\subseteq\bigcup_{q'\in\cR_n(q)}\tC_{n,q'}\subseteq\overline B_{\delta_n}(f_\tau(\tC_{n,q})),\quad\forall\,q\in\cC_n.
  \]
  Suppose also that, for $m\geq n$, there are maps $b_{mn}:\cC_m\to\cC_n$ satisfying
  \[
    \tC_{m,q}\subseteq\tC_{n,b_{mn}(q)},
    \quad
    b_{mn}(\cR_m(q))\subseteq\cR_n(b_{mn}(q)),
    \quad \forall q\in\cC_m.
  \]
  Then $\cF_n^{\cC}:\Xi_n\mvmap\Xi_n$ as in~\eqref{eq:cover-grid-transfer} are a convergent cofiltration of outer approximations of $f_\tau$.
\end{proposition}
\begin{proof}
  It suffices to show that $(\cF_n^{\cC})$ is a cofiltration and that $\cF_n^{\cC}(\xi)\subseteq\cF_n^{r_n}(\xi)$ for some $r_n>0$ with $r_n\to0$, since the outer approximation follows from Proposition~\ref{prop:cover-grid-transfer} and~\eqref{eq:cover-grid-transfer} implies $\cF_n^{\min}(\xi)\subseteq\cF_n^{\cC}(\xi)$.

  For $m\geq n$, recall that $q_{mn}:\Xi_m\to\Xi_n$ assigns each element of $\Xi_m$ to the element of $\Xi_n$ containing it. Fix $\xi,\xi'\in\Xi_m$ with $\xi'\in\cF_m^{\cC}(\xi)$, and choose $q\in\cC_m$, $q'\in\cR_m(q)$ as in~\eqref{eq:cover-grid-transfer}. Since $\tC_{m,q}\subseteq\tC_{n,b_{mn}(q)}$ and $\tC_{m,q'}\subseteq\tC_{n,b_{mn}(q')}$,
  \[
    \begin{aligned}
      \emptyset\neq|\xi|_{\susp}\cap\tC_{m,q}&\subseteq|q_{mn}(\xi)|_{\susp}\cap\tC_{n,b_{mn}(q)},\\
      \emptyset\neq|\xi'|_{\susp}\cap\tC_{m,q'}&\subseteq|q_{mn}(\xi')|_{\susp}\cap\tC_{n,b_{mn}(q')}.
    \end{aligned}
  \]
  Moreover, $b_{mn}(q')\in b_{mn}(\cR_m(q))\subseteq\cR_n(b_{mn}(q))$, so~\eqref{eq:cover-grid-transfer} implies $q_{mn}(\xi')\in\cF_n^{\cC}(q_{mn}(\xi))$, proving that $(\cF_n^{\cC})$ is a cofiltration. Since $c_n\to0$, uniform continuity implies that
  \[
    \max_{q\in\cC_n}\diam(f_\tau(\tC_{n,q}))\to0.
  \]
  Together with $\diam(\Xi_n)\to0$ and $\delta_n\to0$, this yields
  \[
    r_n:=\diam(\Xi_n)+\delta_n+\max_{q\in\cC_n}\diam(f_\tau(\tC_{n,q}))+2^{-n}\to0.
  \]
  Finally, for $\xi'\in\cF_n^{\cC}(\xi)$, choose $q\in\cC_n$, $q'\in\cR_n(q)$ and
  \[
    x\in|\xi|_{\susp}\cap\tC_{n,q},
    \quad x'\in|\xi'|_{\susp}\cap\tC_{n,q'}.
  \]
  The hypothesis
  \[
    \bigcup_{q'\in\cR_n(q)}\tC_{n,q'}\subseteq\overline B_{\delta_n}(f_\tau(\tC_{n,q}))
  \]
  implies $d_{\susp X}(x',f_\tau(\tC_{n,q}))\leq\delta_n$. Thus, for $y\in|\xi'|_{\susp}$,
  \[
    \begin{aligned}
      d_{\susp X}(y,f_\tau(|\xi|_{\susp}))&\leq d_{\susp X}(y,f_\tau(x))\\
      &\leq d_{\susp X}(y,x')+d_{\susp X}(x',f_\tau(\tC_{n,q}))+\diam(f_\tau(\tC_{n,q}))\\
      &\leq\diam(\Xi_n)+\delta_n+\diam(f_\tau(\tC_{n,q}))<r_n.
    \end{aligned}
  \]
  It follows that
  \[
    |\cF_n^{\cC}(\xi)|_{\susp}\subseteq B_{r_n}(f_\tau(|\xi|_{\susp})),
    \quad\forall\,\xi\in\Xi_n,
  \]
  hence $\cF_n^{\cC}(\xi)\subseteq\cF_n^{r_n}(\xi)$.
\end{proof}

The punchline is the following. The families
\[
  \setdef{\incl(|\mu|)}{\mu\in\cX_n}\cup\setdef{\pi\left((|\mu|\cap G)\times\Ink\right)}{\mu\in\cX_n,\ \Ink\in\cI_n}
\]
are compact covers of $\susp X$ whose diameters tend to zero and whose elements are contained in elements of each preceding cover. If we denote these covers by $(\tC_{n,q})_{q\in\cC_n}$, then Proposition~\ref{prop:cover-transfer-convergence} applies with
\[
  \cR_n(q):=\setdef{q'\in\cC_n}{\tC_{n,q'}\cap f_\tau(\tC_{n,q})\neq\emptyset}, \text{ and } \delta_n=c_n.
\]
 This bypasses the need to compute $f_\tau(|\xi|_{\susp})$ since \eqref{eq:cover-grid-transfer} depends only on $\cR_n$ and the nonempty intersections $|\xi|_{\susp}\cap\tC_{n,q}$.

\subsection{\texorpdfstring{$\tau$}{Tau}-maps}

The following lemma relates the $\omega$-limits of the suspension semiflow and its $\tau$-map $f_\tau$.

\begin{lemma}
  \label{lem:fixed-time-map}
  For every $\tU\subseteq\susp X$,
  \[
    \omega(\tU,\Phi)=\Phi\left([0,\tau],\omega(\tU,f_\tau)\right).
  \]
\end{lemma}
\begin{proof}
Set $\tA:=\omega(\tU,f_\tau)$. If $z\in\tA$ and $s\in[0,\tau]$, choose $x_j\in\tU$ and $n_j\to\infty$ such that $f_\tau^{n_j}(x_j)\to z$. Since $f_\tau^{n_j}=\Phi(n_j\tau,\cdot)$, continuity gives
\[
  \Phi(s,z)=\lim_{j\to\infty}\Phi(n_j\tau+s,x_j)\in\omega(\tU,\Phi).
\]
Hence $\Phi([0,\tau],\tA)\subseteq\omega(\tU,\Phi)$. Conversely, if $y\in\omega(\tU,\Phi)$, choose $x_j\in\tU$ and $t_j\to\infty$ such that $\Phi(t_j,x_j)\to y$. Write $t_j=n_j\tau+s_j$ with $s_j\in[0,\tau)$, and we can pass to a subsequence such that $s_j\to s$ and $f_\tau^{n_j}(x_j)\to z$. Since $n_j\to\infty$, we have $z\in\tA$, and continuity gives $y=\Phi(s,z)$.
\end{proof}

Consequently, $\Phi$ and $f_\tau$ have the same attracting neighborhoods.

\begin{proposition}
  \label{prop:fixed-time-neighborhoods}
  The attracting neighborhoods coincide,
  \[
    \sANbhd(\susp X,\Phi)=\sANbhd(\susp X,f_\tau),
  \]
  and $\omega(\tU,\Phi)=\omega(\tU,f_\tau)$ for every $\tU$ in this family.
\end{proposition}
\begin{proof}
Suppose first that $\tU\in\sANbhd(\susp X,f_\tau)$ and set $\tA:=\omega(\tU,f_\tau)$. Then $\tA\subseteq\Int_{\susp X}(\tU)$, and $\tA$ is compact with $f_\tau(\tA)=\tA$. By continuity of $\Phi$, there exists $\delta>0$ such that $\Phi([0,\delta],\tA)\subseteq\Int_{\susp X}(\tU)$. Choose $m$ such that $h:=\tau/m<\delta$, and write $\tA_s:=\Phi(s,\tA)$. For $0\leq s\leq h$, the set $\tA_s$ is $f_\tau$-invariant and contained in $\tU$, hence $\tA_s\subseteq\omega(\tU,f_\tau)=\tA$. Therefore
\[
  \tA=\tA_0\supseteq\tA_h\supseteq\cdots\supseteq\tA_{mh}
  =\tA_\tau=f_\tau(\tA)=\tA,
\]
so $\tA_h=\tA$. If $0\leq s\leq h$, then $\tA_s,\tA_{h-s}\subseteq\tA$, and hence
\[
  \tA=\tA_h=\Phi(s,\tA_{h-s})\subseteq\tA_s\subseteq\tA.
\]
Thus $\Phi(s,\tA)=\tA$ for $s\in[0,h]$. For any $t\geq0$, write $t=kh+s$ with $k\in\mathbb{N}$ and $s\in[0,h)$; then the semiflow property gives $\Phi(t,\tA)=\tA$. Lemma~\ref{lem:fixed-time-map} now yields $\omega(\tU,\Phi)=\tA\subseteq\Int_{\susp X}(\tU)$, so $\tU\in\sANbhd(\susp X,\Phi)$ and $\omega(\tU,\Phi)=\omega(\tU,f_\tau)$.

Conversely, if $\tU\in\sANbhd(\susp X,\Phi)$ and $\tA:=\omega(\tU,f_\tau)$, then Lemma~\ref{lem:fixed-time-map} gives $\tA\subseteq\Phi([0,\tau],\tA)=\omega(\tU,\Phi)\subseteq\Int_{\susp X}(\tU)$. Hence $\tU\in\sANbhd(\susp X,f_\tau)$, and the preceding argument again gives $\omega(\tU,\Phi)=\omega(\tU,f_\tau)$.
\end{proof}

\begin{corollary}
  \label{cor:fixed-time-attractors}
  The identity map is a bounded lattice isomorphism
  \[
    \sAtt(\susp X,\Phi)\longrightarrow\sAtt(\susp X,f_\tau).
  \]
\end{corollary}
\begin{proof}
Proposition~\ref{prop:fixed-time-neighborhoods} implies that $\Phi$ and $f_\tau$ have the same attracting neighborhoods and assign the same $\omega$-limit to each of them. Hence $\sAtt(\susp X,\Phi)=\sAtt(\susp X,f_\tau)$. The identity is therefore an order isomorphism, and consequently a bounded lattice isomorphism.
\end{proof}

\subsection{Finite realization}

In this subsection, $(\cF_n)$ may be any convergent cofiltration of outer approximations of $f_\tau$ on the grids $(\Xi_n)$, including the maps in Definition~\ref{def:suspension-multivalued-map} and those obtained from Proposition~\ref{prop:cover-transfer-convergence}.

\begin{corollary}
  \label{cor:realization-suspension}
  If $\tcA\subseteq\sAtt(\susp X,\Phi)$ is a finite sublattice, then for every sufficiently large $n$ there is a lattice monomorphism $\tell_n:\tcA\injection\sInvsetp(\Xi_n,\cF_n)$ such that, for every $\tA\in\tcA$,
  \[
  \tU_{\tA}:=|\tell_n(\tA)|_{\susp} \in \sANbhdR(\susp X,f_\tau) \quad \text{and} \quad \omega(\tU_{\tA},f_\tau)=\tA.
  \]
\end{corollary}
\begin{proof}
  For every sufficiently large $n$, the existence of a lattice monomorphism $\tell_n$ satisfying $\omega(|\tell_n(\tA)|_{\susp},f_\tau)=\tA$ for all $\tA\in\tcA$ follows from Theorem~\ref{thm:realization}, since $(\cF_n)$ is a convergent cofiltration of outer approximations of $f_\tau$ and Corollary~\ref{cor:fixed-time-attractors} identifies $\tcA$ with a finite sublattice of $\sAtt(\susp X,f_\tau)$. Since $\tell_n(\tA)$ is forward invariant under $\cF_n$, \eqref{eq:invset-neighborhood-map} implies that $\tU_{\tA}\in\sANbhdR(\susp X,f_\tau)$.
\end{proof}

\begin{theorem}
  \label{thm:FiniteSublattice}
  Suppose that $\cH$ satisfies the trapping guard condition. If $\cA\subseteq\sAtt(\cH)$ is a finite sublattice, then for every sufficiently large $n$ there is a lattice monomorphism $\ell_n:\cA\injection\sInvsetp(\Xi_n,\cF_n)$ such that for every $A\in\cA$,
  \[
    \omega\left(
      \rho\left(|\ell_n(A)|_{\susp}\right),
      \cH
    \right)
    =A.
  \]
\end{theorem}
\begin{proof}
  The existence of $\ell_n$ for every sufficiently large $n$ follows by composing the lattice isomorphism $A\mapsto\susp(A)$ from Theorem~\ref{thm:main-correspondence} with the monomorphism of Corollary~\ref{cor:realization-suspension} on the image of $\cA$. For $A\in\cA$, the evaluation $\tU_A:=|\ell_n(A)|_{\susp}$ belongs to $\sANbhdR(\susp X,f_\tau)$ and satisfies $\omega(\tU_A,f_\tau)=\susp(A)$. Proposition~\ref{prop:fixed-time-neighborhoods} implies $\tU_A\in\sANbhdR(\susp X,\Phi)$ and $\omega(\tU_A,\Phi)=\susp(A)$, so Corollary~\ref{thm:omega-square} implies
  \[
    \omega(\rho(\tU_A),\cH)=\incl^{-1}\left(\omega(\tU_A,\Phi)\right)=\incl^{-1}(\susp(A))=A.
  \]
\end{proof}

Recall that $d_n:\cX_n\injection\Xi_n$ maps each base grid element $\mu$ to its suspension grid element $\tmu$, with $\incl^{-1}(|\tmu|_{\susp})=|\mu|$.

\begin{corollary}
  \label{cor:base-generated-neighborhoods}
  The map
  \[
    \begin{aligned}
      \sInvsetp(\Xi_n,\cF_n)&\longrightarrow\sANbhdR(X,\cH)\cap\scR_{\cX_n}(X),\\
      \cU&\longmapsto|d_n^{-1}(\cU)|
    \end{aligned}
  \]
  is a bounded lattice homomorphism, and its composition with $\omega(\cdot,\cH)$ is a bounded lattice homomorphism into $\sAtt(\cH)$. Consequently, for $\ell_n$ as in Theorem~\ref{thm:FiniteSublattice}, the composition
  \[
    \begin{aligned}
      \cA&\longrightarrow\sANbhdR(X,\cH)\cap\scR_{\cX_n}(X),\\
      A&\longmapsto|d_n^{-1}(\ell_n(A))|
    \end{aligned}
  \]
  is a lattice monomorphism satisfying
  \[
    \omega(|d_n^{-1}(\ell_n(A))|,\cH)=A
    \quad\forall A\in\cA.
  \]
\end{corollary}
\begin{proof}
  The first map is a bounded lattice homomorphism, since $|d_n^{-1}(\cU)|=\rho_n(|\cU|_{\susp})$ by Proposition~\ref{prop:finite-grid-preimage} and both evaluation and $\rho_n$ preserve joins, meets, and bounds. Since \eqref{eq:invset-neighborhood-map} and Proposition~\ref{prop:fixed-time-neighborhoods} imply $|\cU|_{\susp}\in\sANbhdR(\susp X,\Phi)$ whenever $\cU$ is forward invariant, Lemma~\ref{lem:regularized-preimage-omega} implies that $|d_n^{-1}(\cU)|\in\sANbhdR(X,\cH)$. The map $\cU\mapsto\omega(|d_n^{-1}(\cU)|,\cH)$ is therefore a bounded lattice homomorphism into $\sAtt(\cH)$, since $\omega(\cdot,\cH)$ is one by Proposition~\ref{prop:Att-lattice}.

  Since $\ell_n$ preserves joins and meets, so does $A\mapsto|d_n^{-1}(\ell_n(A))|$. This map is injective because Theorem~\ref{thm:FiniteSublattice} implies that
  \[
    \omega(|d_n^{-1}(\ell_n(A))|,\cH)=\omega(\rho(|\ell_n(A)|_{\susp}),\cH)=A,
    \quad\forall A\in\cA.
  \]
\end{proof}

\subsection{The multivalued map on the base grid}

Let $(\cF_n:\Xi_n\mvmap\Xi_n)$ be a convergent cofiltration of outer approximations of $f_\tau$. Recall that $d_n:\cX_n\injection\Xi_n$ maps each base grid element $\mu$ to the suspension grid element $\tmu$ with $\incl^{-1}(|\tmu|_{\susp})=|\mu|$. On the base grid, define $\cF_n^X:\cX_n\mvmap\cX_n$ by $\mu'\in\cF_n^X(\mu)$ if there is a path in $\cF_n$ of the form
\[
  d_n(\mu)=\xi_0\to\cdots\to\xi_k=d_n(\mu'),
  \quad k\geq1,\quad
  \xi_i\notin d_n(\cX_n)\quad\forall\,i\in\setof{1,\ldots,k-1}.
\]

\begin{proposition}
  \label{prop:base-scc-lattice}
  The preimage map
  \[
    d_n^{-1}:\sInvsetp(\Xi_n,\cF_n)\surjection\sInvsetp(\cX_n,\cF_n^X)
  \]
  is a bounded lattice epimorphism, and evaluation is a lattice isomorphism from $\sInvsetp(\cX_n,\cF_n^X)$ onto
  \[
    \setdef{\rho_n(|\cU|_{\susp})}{\cU\in\sInvsetp(\Xi_n,\cF_n)}\subseteq\sANbhdR(X,\cH).
  \]
\end{proposition}
\begin{proof}
  Since preimages preserve unions, intersections, and bounds, and each edge of $\cF_n^X$ represents a path in $\cF_n$, the map $d_n^{-1}$ is a bounded lattice homomorphism. To show surjectivity, let $\cV\in\sInvsetp(\cX_n,\cF_n^X)$ and define
  \[
    \cU:=\bigcup_{j\geq0}\cF_n^j(d_n(\cV))
  \]
  so that $\cF_n(\cU)\subseteq\cU$ and $\cV\subseteq d_n^{-1}(\cU)$. Conversely, for $\mu'\in d_n^{-1}(\cU)$, choose a path in $\cF_n$ with
  \[
    d_n(\mu_0)=\xi_0\to\cdots\to\xi_m=d_n(\mu'),
    \quad \mu_0\in\cV.
  \]
  Choose $0=j_0<\cdots<j_k=m$ such that
  \[
    \setof{j_0,\ldots,j_k}
    =\setdef{j\in\setof{0,\ldots,m}}{\xi_j\in d_n(\cX_n)},
  \]
  and define $\mu_i\in\cX_n$ by $d_n(\mu_i)=\xi_{j_i}$ for every $i\in\setof{0,\ldots,k}$. For every $i\in\setof{0,\ldots,k-1}$,
  \[
    \xi_j\notin d_n(\cX_n)
    \quad\forall\,j\in\setof{j_i+1,\ldots,j_{i+1}-1},
  \]
  so $\mu_{i+1}\in\cF_n^X(\mu_i)$ by definition of $\cF_n^X$. Since $\cV$ is forward invariant,
  \[
    \mu'=\mu_k\in(\cF_n^X)^k(\cV)\subseteq\cV.
  \]
  Thus $d_n^{-1}(\cU)=\cV$, proving surjectivity. The evaluation isomorphism follows from~\eqref{eq:finite-grid-preimage} and the grid property, and its image is in $\sANbhdR(X,\cH)$ by Corollary~\ref{cor:base-generated-neighborhoods}.
\end{proof}

For a finite sublattice $\cA\subseteq\sAtt(\cH)$, Theorem~\ref{thm:FiniteSublattice} and Proposition~\ref{prop:base-scc-lattice} therefore imply that, for every sufficiently large $n$, the composition
\[
  d_n^{-1}\circ\ell_n:\cA\longrightarrow\sInvsetp(\cX_n,\cF_n^X)
\]
is a bounded lattice homomorphism. By Corollary~\ref{cor:base-generated-neighborhoods}, its evaluations form a sublattice of $\sANbhdR(X,\cH)$ and satisfy
\[
  \omega\left(|d_n^{-1}(\ell_n(A))|,\cH\right)=A,
  \quad\forall A\in\cA.
\]
This identity proves injectivity and completes the proof of Theorem~\ref{thm:main}(iii).

The path decomposition above shows that $M\mapsto d_n^{-1}(M)$ is an order isomorphism
\[
  \setdef{M\in\MG(\cF_n)}{M\cap d_n(\cX_n)\neq\emptyset}\longrightarrow\MG(\cF_n^X).
\]
We can therefore compute these Morse sets and their reachability order on the base grid, then use the associated suspension pairs to compute their Conley indices.

\subsection{The suspension Conley index}

We refer to~\cite{Mischaikow:Mrozek} for a review of Conley index theory, and to~\cite{Rybakowski} and~\cite{Franks:Richeson} for the Conley index of semiflows and of maps, respectively. For an isolated invariant set $S$ of $\cH$, the suspension Conley index is
\[
  \SCH_*(S):=\CH_*(\susp(S),\Phi),
\]
with field $\F$ coefficients. Conversely, every isolated invariant set $\tM$ of $\Phi$ satisfies $\tM=\susp(S)$ for the isolated invariant set $S:=\incl^{-1}(\tM)$ of $\cH$~\cite{Rivas:Kalies}.

Let $\cF:\Xi\mvmap\Xi$ be an outer approximation of $f_\tau$ on a grid $\Xi$ of $\susp X$. For $\cV\subseteq\cU$ in $\sInvsetp(\Xi,\cF)$, define
\[
  \begin{aligned}
    \tN&:=|\cU|_{\susp},&\quad \tA&:=\omega(\tN,\Phi),\\
    \tL&:=|\cV|_{\susp},& \tB&:=\omega(\tL,\Phi).
  \end{aligned}
\]
By~\eqref{eq:invset-realization}, $\tN$ and $\tL$ are trapping regions for $f_\tau$, hence attracting neighborhoods for $\Phi$ with the same attractors and dual repellers~\cite{Kalies:Mischaikow:Vandervorst:2016}. The regular closed difference
\[
  \tK:=|\cU\setminus\cV|_{\susp}=\cl_{\susp X}(\tN\setminus\tL)
\]
is therefore isolating for both systems~\cite{Kalies:Mischaikow:Vandervorst:2022}, with
\[
  \tM:=\Inv(\tK,f_\tau)=\Inv(\tK,\Phi)=\tA\cap\tB^*,
\]
where $\tB^*:=\setdef{x\in\susp X}{\omega(x,\Phi)\cap\tB=\emptyset}$ is the dual repeller of $\tB$.

For $M\in\MG(\cF)$, the families
\[
  \cU_M:=\bigcup_{j\geq0}\cF^j(M),\quad \cV_M:=\cU_M\setminus M,
\]
are forward invariant, since an edge from $\cV_M$ into $M$ would contradict maximality of $M$. Consequently, $|M|_{\susp}$ isolates $\tM:=\Inv(|M|_{\susp},\Phi)$, and the nonempty $\tM$ form a Morse representation of $\omega(\susp X,\Phi)$ with order $\tM'\leq\tM$ if and only if $M'\leq M$~\cite{Kalies:Mischaikow:Vandervorst:2022}.

For $M_X\in\MG(\cF_n^X)$, let $M\in\MG(\cF_n)$ satisfy $d_n^{-1}(M)=M_X$. Evaluating the families above for $\cF=\cF_n$ defines the suspension pair
\[
  (\tN,\tL):=(|\cU_M|_{\susp},|\cV_M|_{\susp}),
\]
from which the following proposition computes the index.

\begin{proposition}
  \label{prop:grid-conley-index}
  Let $\tN=|\cU|_{\susp}$ and $\tL=|\cV|_{\susp}$ for $\cV\subseteq\cU$ in $\sInvsetp(\Xi,\cF)$. The pair $(\htN,\htL)$ defined by
  \[
    \htN:=\Phi([0,\tau],\tN),\quad \htL:=\Phi([0,\tau],\tL)
  \]
  is an index pair for $\tS:=\Inv(|\cU\setminus\cV|_{\susp},\Phi)$ under $\Phi$, and for all sufficiently large $k$,
  \[
    \im (f_\tau)_*^k\cong\CH_*(\tS,\Phi)\cong\SCH_*(\incl^{-1}(\tS)),\quad (f_\tau)_*|_{\im (f_\tau)_*^k}=\id. 
  \]
\end{proposition}
\begin{proof}

Write $f:=f_\tau$, and set $\tA:=\omega(\tN,\Phi)$ and $\tB:=\omega(\tL,\Phi)$. Since $\cU,\cV\in\sInvsetp(\Xi,\cF)$, \eqref{eq:invset-realization} gives $f(\tN)\subseteq\Int_{\susp X}\tN$ and $f(\tL)\subseteq\Int_{\susp X}\tL$. Thus $\tN$ and $\tL$ are trapping regions for $f$, and Proposition~\ref{prop:fixed-time-neighborhoods} implies that they are attracting neighborhoods for $\Phi$ with the same attractors. Also, we have $\tS=\tA\cap\tB^*$. We first show that $(\htN,\htL)$ is an index pair for $\tS$. Since $f(\tN)\subseteq\tN$,
\[
  \Phi([0,\tau],\htN)=\htN\cup\Phi([0,\tau],f(\tN))\subseteq\htN,
\]
and similarly $\Phi([0,\tau],\htL)\subseteq\htL$. Therefore, by the semiflow property, $\htN$ and $\htL$ are positively invariant. Consequently, $\htL$ is positively invariant in $\htN$, and the exit
condition is trivial.

We claim that $\Inv(\cl_{\susp X}(\htN\setminus\htL),\Phi)=\tA\cap\tB^*=\tS$. So, if $x$ belongs to the left-hand side, there exists a full trajectory $\gamma$ through $x$ in $\cl(\htN\setminus\htL)$. Since $x=\Phi(j\tau,\gamma(-j\tau))$ with $\gamma(-j\tau)\in\htN$ for all $j\geq1$, we have $x\in\omega(\htN,\Phi)=\tA$. Moreover, $\omega(x,\Phi)\cap\tB=\emptyset$, since $\tB\subseteq\Int\htL$ and $\cl_{\susp X}(\htN\setminus\htL)\cap\Int\htL=\emptyset$; hence $x\in\tB^*$.
Conversely, if $x\in\tA\cap\tB^*$, then $x\in\Int\tN\subseteq\htN$, while $x\notin\htL$, since otherwise positive invariance of $\htL$ would imply $\omega(x,\Phi)\subseteq\omega(\htL,\Phi)=\tB$. Thus $\tS\subseteq\htN\setminus\htL$, and invariance of $\tS$ gives the reverse inclusion. Finally, $\tS\subseteq\Int(\htN\setminus\htL)$, so $(\htN,\htL)$ is an index pair for $\tS$.

Set $T:=f_*:H_*(\tN,\tL;\F)\to H_*(\tN,\tL;\F)$. Since $\htN$ and $\htL$ are positively invariant, $H(r,x):=\Phi(r\tau,x)$ is a homotopy of pairs on $(\htN,\htL)$ from $\id$ to $f$. Hence $f_*=\id$ on $H_*(\htN,\htL;\F)$. 

Here, we have $f^j(\tN)$ decreasing to $\omega(\tN,f)=\tA$, and similarly $f^j(\tL)$ decreasing to $\tB$. Since $\Phi([0,\tau],\tA)=\tA\subseteq\Int\tN$ and $\Phi([0,\tau],\tB)=\tB\subseteq\Int\tL$, compactness and continuity give, for some $q\geq1$,
\[
  f^q(\htN)=\Phi([0,\tau],f^q(\tN))\subseteq\tN,\qquad f^q(\htL)=\Phi([0,\tau],f^q(\tL))\subseteq\tL.
\]
Thus the inclusion $i:(\tN,\tL)\hookrightarrow(\htN,\htL)$ and $f^q:(\htN,\htL)\to(\tN,\tL)$ induce maps $\alpha:=i_*$ and $\beta:=(f^q)_*$ satisfying $\beta\alpha=T^q$ and $\alpha\beta=\id$. Hence $\beta$ is injective and $\im T^q=\im\beta\cong H_*(\htN,\htL;\F)$. Moreover, $T\beta=\beta$, since $f_*=\id$ on $H_*(\htN,\htL;\F)$. Therefore $T|_{\im T^q}=\id$ and, for every $k\geq q$, $\im T^k=\im T^q$ and $T|_{\im T^k}=\id$.

Finally, since $(\htN,\htL)$ is an index pair for $\tS$,
\[
  \im(f_\tau)_*^k\cong H_*(\htN,\htL;\F)\cong\CH_*(\tS,\Phi)\cong\SCH_*(\incl^{-1}(\tS))
\]
for every $k\geq q$, where the last isomorphism follows from $\tS=\susp(\incl^{-1}(\tS))$ and the definition of the suspension Conley index. Also $(f_\tau)_*|_{\im(f_\tau)_*^k}=\id$, completing the proof.
\end{proof}
These results suggest a reverse interpretation of our computational framework. Instead of saying, start with a hybrid system and construct its time-$\tau$ map's outer approximation, one can start with a multivalued map $\cF$ and ask which hybrid systems have a time-$\tau$ map that is a selector of $\cF$. More specifically, after fixing our state space and guard, a small perturbation of the flow and reset map may give a different hybrid system, but the dynamical structures of its suspension are represented by the same outer approximation $\cF$.
This suggests a notion of robustness using the combinatorial dynamics framework. That is, we are not describing just one unique model, but all hybrid systems represented by a common multivalued enclosure. A precise formulation of such hybrid suspension selectors is left for future work.

\section{Examples}
\label{sec:example}

This section illustrates the constructions above on four hybrid systems: the bouncing ball is a toy example that checks that the computation recovers the known homology, the rimless wheel is a mechanical model, the spiking neuron is an application to neuroscience, and the impacting oscillator is a mathematical example with several attractors. For each system, we compute a multivalued map on a suspension grid from samples of $f_\tau$ and its Conley--Morse graph, that is, its Morse graph with each Morse set labeled by its Conley index. In the first three examples, the computed Conley--Morse graph coincides with the known one, and for the impacting oscillator we describe the dynamics it suggests. The computations follow~\cite{Arai:Kalies:Kokubu:Mischaikow:Oka:Pilarczyk}, and we use the computational topology framework of CMGDB~\cite{CMGDB} and CHomP~\cite{pychomp}, both based on cubical homology~\cite{Kaczynski:Mischaikow:Mrozek}. We highlight that the multivalued maps are constructed heuristically via sampling, and hence they are not rigorous outer approximations.

For each example, we choose a rectangular region $R\subseteq X$ with $r(G\cap R)\subseteq R$, and restrict ourselves to
\[
  \tD:=\susp(R)=\incl(R)\cup\pi\left((G\cap R)\times[0,1]\right),
\]
in order to describe the maximal invariant set of $f_\tau$ in $\tD$. The base grid $\cX_n$ subdivides $R$ into congruent rectangles, so for every $\xi\in\Xi_n$, the set $|\xi|_{\susp}$ is a finite union of sets $\incl(|\mu|)$ with $\mu\in\cX_n$ and $\pi(J\times\Ink)$ with $J\subseteq G\cap R$ an interval. For each sample, we compute $f_\tau$ via numerical integration and locate the impacts by root finding. This way, the multivalued map $\cF:\Xi_n\mvmap\Xi_n$ maps $\xi$ to $\xi'$ if $f_\tau(p)\in |\xi'|$ for some $p \in |\xi|$. We choose $\tau>0$ sufficiently small to minimize integration errors, but large enough so that $f_\tau$ is not close to the identity at the resolution of the grid.

Recall that the Morse sets $M(0),M(1),\dots$ of $\cF$ are its recurrent strongly connected components, and the Morse graph $\MG(\cF)$ orders them by reachability. For each Morse set $M$, we compute its Conley index from $\cF$ and report the dimensions of this index in degrees $0$, $1$, and $2$ as a triple. For any hybrid system for which $\cF$ is an outer approximation of $f_\tau$, this index is the Conley index $\CH_*(\tM,\Phi)=\SCH_*(\incl^{-1}(\tM))$ of $\tM:=\Inv(|M|_{\susp},\Phi)$. A nontrivial index implies that $\tM$ and the invariant set $\incl^{-1}(\tM)$ of $\cH$ are nonempty, some indices imply that $\incl^{-1}(\tM)$ contains an equilibrium, and under further hypotheses, others imply that it contains a periodic orbit~\cite{Rivas:Kalies}. The trivial index does not mean that $\tM$ is empty. For each example, we show the sets associated to the Morse sets with nontrivial index and the computed Conley--Morse graph. If some Morse sets have trivial index, the figures show both the full Conley--Morse graph, with these Morse sets in gray, and the restricted Conley--Morse graph, which orders the Morse sets with nontrivial index by reachability in $\MG(\cF)$.

\subsection{Bouncing ball}

We use the bouncing ball with height $h$, velocity $v$, flow $\varphi$ generated by $(\hdot,\vdot)=(v,-g)$, guard $G=\setdef{(0,v)}{v\leq0}$, and reset $r(0,v)=(0,-cv)$, and we take $g=9.81$ and $c=0.8$. 
Since $r(0,0)=(0,0)$, the origin corresponds to the suspended periodic orbit $\widetilde{Z}:=\susp(\setof{(0,0)})$ of $\Phi$, which has period $1$. Recall that $\SCH_*(\setof{(0,0)})\cong H_*(S^1)$~\cite{Rivas:Kalies}.

Figure~\ref{fig:bouncing-ball} shows the computation for $\tau=0.5$. The Morse graph has a single Morse set, with Conley index $(1,1,0)$. In $R$, this Morse set lies along $h=0$ near the origin.

\begin{figure}[!htbp]
  \centering
  \captionsetup[sub]{justification=centering}
  \begin{subfigure}[b]{0.36\linewidth}
    \centering
    \parbox[b][5.2cm][c]{\linewidth}{\centering\includegraphics[width=\linewidth,height=5.2cm,keepaspectratio]{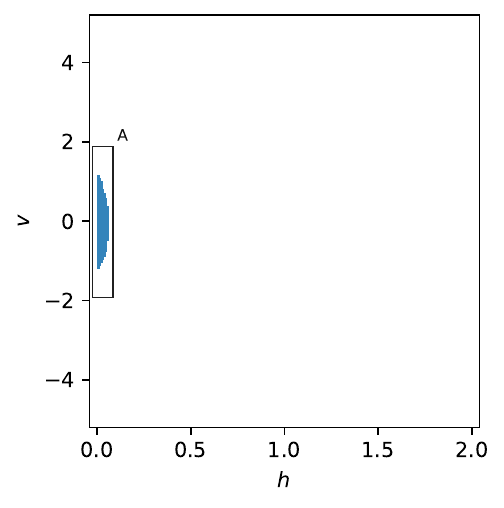}}
    \caption{Morse sets in $R$.}
  \end{subfigure}\hfill
  \begin{subfigure}[b]{0.3\linewidth}
    \centering
    \parbox[b][5.2cm][c]{\linewidth}{\centering\includegraphics[width=\linewidth,height=4.2cm,keepaspectratio]{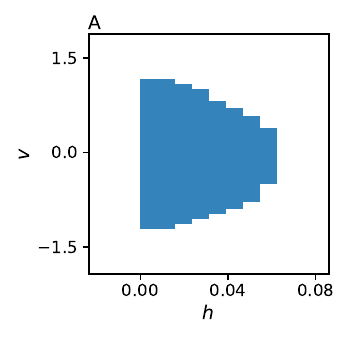}}
    \caption{$M(0)$ in A.}
  \end{subfigure}\hfill
  \begin{subfigure}[b]{0.26\linewidth}
    \centering
    \parbox[b][5.2cm][c]{\linewidth}{\centering\includegraphics[width=\linewidth,height=1.8cm,keepaspectratio]{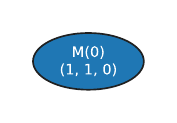}}
    \caption{Conley--Morse graph.}
  \end{subfigure}
  \caption{Bouncing ball at $\tau=0.5$, on a $2^8\times2^8$ grid of $R=[0,2]\times[-5,5]$, with $2^6$ phase intervals.}
  \label{fig:bouncing-ball}
\end{figure}
\FloatBarrier

\subsection{Rimless wheel}

We use the rimless wheel~\cite{McGeer}, a planar wheel with spokes rolling down an inclined plane, in the form and with the parameters of~\cite{Shia:Vasudevan:Bajcsy:Tedrake}. Its state consists of the angle $\theta$ between the vertical and the spoke in contact with the ground and the angular velocity $\thetadot$, the flow $\varphi$ is generated by $\ddot\theta=\sin\theta$, the guard is $G=\setdef{(\alpha+\gamma,\thetadot)}{\thetadot\geq0}$, and the reset is $r(\theta,\thetadot)=(2\gamma-\theta,\cos(2\alpha)\thetadot)$, where the slope is $\gamma=0.2$ and the angle between two spokes is $2\alpha=0.8$. 

Let $R:=[-0.2,0.6]\times[-0.5,1]$. Let $X$ be the closure of the set of states reached from $R$. Note that $\cH$ satisfies TGC on $X$ and recall that the periodic orbit has index $H_*(S^1)$, and that the saddle, a hyperbolic equilibrium with one unstable direction, has index $\SCH_k(\setof{(0,0)})\cong\F$ for $k=1$ and $\SCH_k(\setof{(0,0)})=0$ otherwise~\cite{Rivas:Kalies}.

Figures~\ref{fig:rimless-wheel} and~\ref{fig:rimless-wheel-full} compare the computations for $\tau=0.5$ on $2^7\times2^7$ and $2^{11}\times2^{11}$ base grids, with $2^9$ phase intervals in both. At each resolution, $M(0)$ has Conley index $(1,1,0)$, and the Morse set around the saddle, $M(12)$ on the coarse grid and $M(9)$ on the fine grid, has Conley index $(0,1,0)$. Both restricted Conley--Morse graphs have an edge from the saddle to the stable periodic orbit. The full graphs have $20$ and $17$ Morse sets, respectively, so the coarse computation already distinguishes these invariant sets and their order, although the Morse sets with trivial index change with resolution.

\begin{figure}[!htbp]
  \centering
  \captionsetup[sub]{justification=centering}
  \begin{subfigure}[b]{0.48\linewidth}
    \centering
    \parbox[b][4.5cm][c]{\linewidth}{\centering\includegraphics[width=\linewidth,height=4.5cm,keepaspectratio]{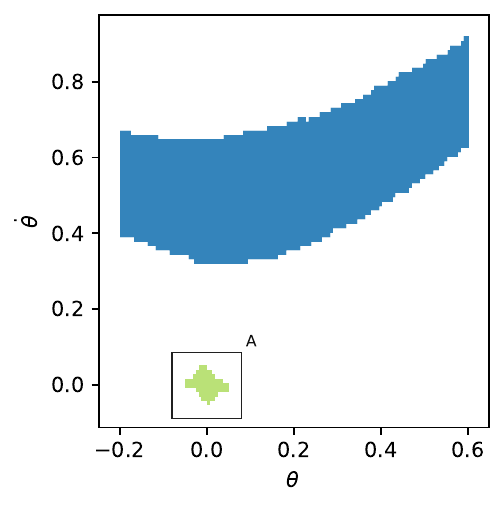}}
    \caption{$2^7\times2^7$: $M(0)$ and $M(12)$.}
  \end{subfigure}\hfill
  \begin{subfigure}[b]{0.48\linewidth}
    \centering
    \parbox[b][4.5cm][c]{\linewidth}{\centering\includegraphics[width=\linewidth,height=4.5cm,keepaspectratio]{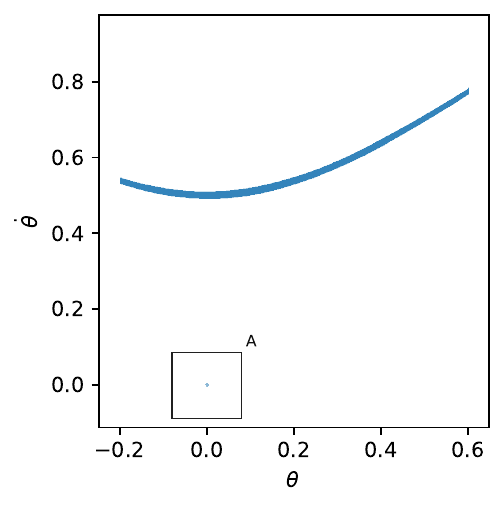}}
    \caption{$2^{11}\times2^{11}$: $M(0)$ and $M(9)$.}
  \end{subfigure}
  \par\medskip
  \begin{subfigure}[b]{0.48\linewidth}
    \centering
    \parbox[b][3.2cm][c]{\linewidth}{\centering\includegraphics[width=\linewidth,height=3.2cm,keepaspectratio]{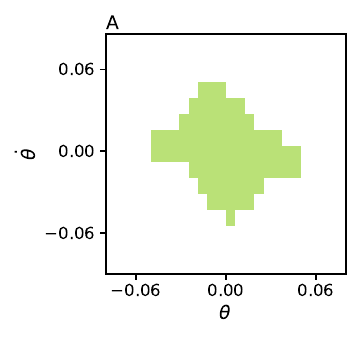}}
    \caption{$M(12)$ in A.}
  \end{subfigure}\hfill
  \begin{subfigure}[b]{0.48\linewidth}
    \centering
    \parbox[b][3.2cm][c]{\linewidth}{\centering\includegraphics[width=\linewidth,height=3.2cm,keepaspectratio]{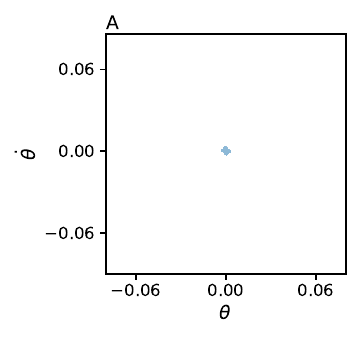}}
    \caption{$M(9)$ in A.}
  \end{subfigure}
  \par\medskip
  \begin{subfigure}[b]{0.48\linewidth}
    \centering
    \parbox[b][2.7cm][c]{\linewidth}{\centering\includegraphics[width=\linewidth,height=2.7cm,keepaspectratio]{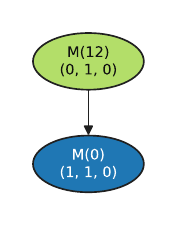}}
    \caption{Coarse restricted Conley--Morse graph.}
  \end{subfigure}\hfill
  \begin{subfigure}[b]{0.48\linewidth}
    \centering
    \parbox[b][2.7cm][c]{\linewidth}{\centering\includegraphics[width=\linewidth,height=2.7cm,keepaspectratio]{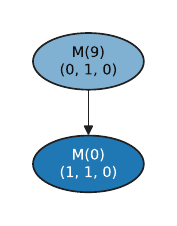}}
    \caption{Fine restricted Conley--Morse graph.}
  \end{subfigure}
  \caption{Rimless wheel at $\tau=0.5$ on $2^7\times2^7$ (left) and $2^{11}\times2^{11}$ (right) grids of $R=[-0.2,0.6]\times[-0.5,1]$, with $2^9$ phase intervals in both computations. Corresponding axes have the same limits.}
  \label{fig:rimless-wheel}
\end{figure}

\begin{figure}[!htbp]
  \centering
  \captionsetup[sub]{justification=centering}
  \begin{subfigure}[b]{0.48\linewidth}
    \centering
    \parbox[b][14cm][c]{\linewidth}{\centering\includegraphics[width=\linewidth,height=14cm,keepaspectratio]{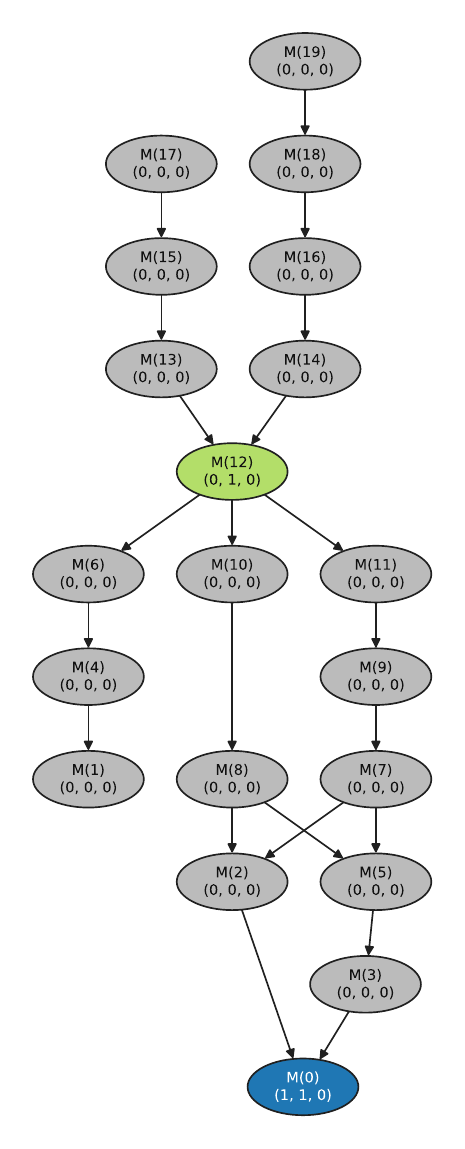}}
    \caption{$2^7\times2^7$: $20$ Morse sets.}
  \end{subfigure}\hfill
  \begin{subfigure}[b]{0.48\linewidth}
    \centering
    \parbox[b][14cm][c]{\linewidth}{\centering\includegraphics[width=\linewidth,height=14cm,keepaspectratio]{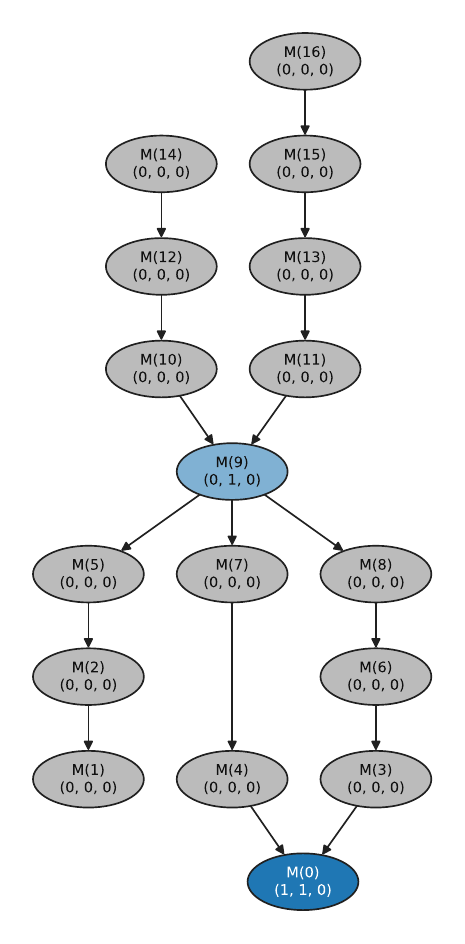}}
    \caption{$2^{11}\times2^{11}$: $17$ Morse sets.}
  \end{subfigure}
  \caption{Full Conley--Morse graphs for the two resolutions in Figure~\ref{fig:rimless-wheel}. Morse sets with trivial index are gray.}
  \label{fig:rimless-wheel-full}
\end{figure}
\FloatBarrier

\subsection{Spiking neuron}

We use the regular-spiking neuron of~\cite{Izhikevich:2007}, a quadratic integrate-and-fire model, with membrane potential $v$, recovery variable $u$, flow $\varphi$ generated by
\[
  \begin{aligned}
    100\,\vdot&=0.7(v+60)(v+40)-u+70,\\
    \udot&=0.03\left(-2(v+60)-u\right),
  \end{aligned}
\]
guard $G=\setdef{(35,u)}{-300\leq u\leq160}$, and reset $r(35,u)=(-50,u+100)$. 
Recall that the L-shaped set
\[
  X:=([-80,-40]\times[-300,600])\cup([-40,35]\times[-300,160])
\]
is compact and forward invariant, that $r(G)\subseteq\Int X$, that $\cH$ has a periodic orbit, and that the maximal invariant set $A$ of $\cH$ in $X$ has index $\SCH_*(A)\cong H_*(S^1)$~\cite{Rivas:Kalies}. 

The base grid consists of the elements of a uniform grid of $[-120,200]\times[-400,880]$ contained in $X$, whose union is $X$. Figure~\ref{fig:spiking-neuron} shows the computation for $\tau=5$. The Morse graph has a single Morse set, with Conley index $(1,1,0)$. In $X$, this Morse set is a band that follows the periodic orbit from its reset point on the line $v=-50$ to the guard.

\begin{figure}[!htbp]
  \centering
  \captionsetup[sub]{justification=centering}
  \begin{subfigure}[b]{0.5\linewidth}
    \centering
    \parbox[b][6cm][c]{\linewidth}{\centering\includegraphics[width=\linewidth,height=6cm,keepaspectratio]{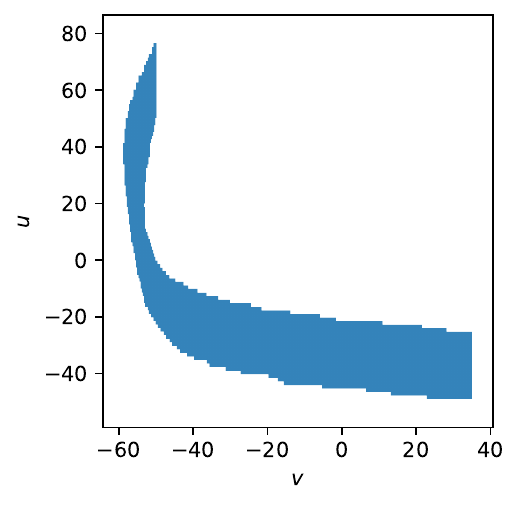}}
    \caption{Morse set in $X$.}
  \end{subfigure}\hfill
  \begin{subfigure}[b]{0.5\linewidth}
    \centering
    \parbox[b][6cm][c]{\linewidth}{\centering\includegraphics[width=\linewidth,height=2cm,keepaspectratio]{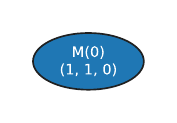}}
    \caption{Conley--Morse graph.}
  \end{subfigure}
  \caption{Spiking neuron at $\tau=5$, on a $2^{10}\times2^{10}$ grid of $[-120,200]\times[-400,880]$.}
  \label{fig:spiking-neuron}
\end{figure}
\FloatBarrier

\subsection{Impacting oscillator}
\label{subsec:impacting-oscillator}

A particle moves in the double-well potential $V(x):=x^4/4-x^2/2$ with the damping coefficient $\varepsilon(x^2-\beta)$. The particle hits a rigid stop at $x=w$ inside the right well and bounces back with coefficient of restitution $c$. With velocity $v$ and $x\leq w$, the flow $\varphi$ is generated by
\[
  \dot x=v,\quad\vdot=x-x^3+\varepsilon(\beta-x^2)v,
\]
the guard is $G=\setdef{(w,v)}{v\geq0}$, and the reset is $r(w,v)=(w,-cv)$, where $\varepsilon=1$, $\beta=0.76$, $w=0.8$, and $c=0.7$. Let $R:=[-1.95,0.8]\times[-2.35,1.95]$, and let $X$ be the closure of the set of states reached from $R$. 

\begin{figure}[!htbp]
  \centering
  \captionsetup[sub]{justification=centering}
  \begin{subfigure}[b]{0.4\linewidth}
    \centering
    \parbox[b][5.6cm][c]{\linewidth}{\centering\includegraphics[width=\linewidth,height=5.6cm,keepaspectratio]{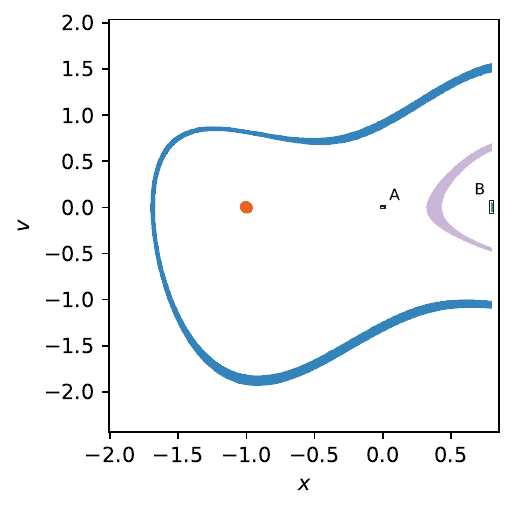}}
    \caption{Morse sets in $R$.}
  \end{subfigure}\hfill
  \begin{subfigure}[b]{0.27\linewidth}
    \centering
    \parbox[b][5.6cm][c]{\linewidth}{\centering\includegraphics[width=\linewidth,height=4.2cm,keepaspectratio]{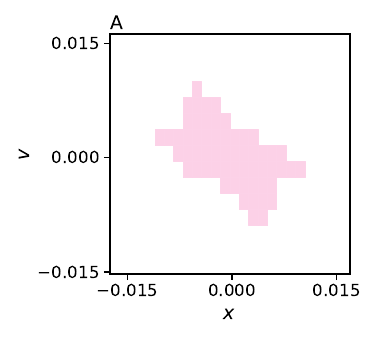}}
    \caption{$M(11)$ in region A.}
  \end{subfigure}\hfill
  \begin{subfigure}[b]{0.27\linewidth}
    \centering
    \parbox[b][5.6cm][c]{\linewidth}{\centering\includegraphics[width=\linewidth,height=4.2cm,keepaspectratio]{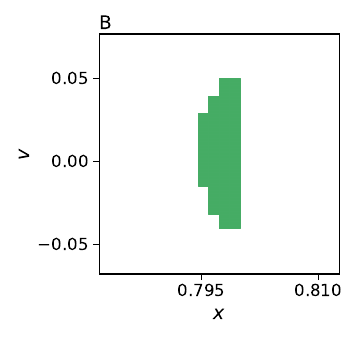}}
    \caption{$M(2)$ in region B.}
  \end{subfigure}
  \par\medskip
  \begin{subfigure}[b]{0.44\linewidth}
    \centering
    \parbox[b][9.5cm][c]{\linewidth}{\centering\includegraphics[width=\linewidth,height=4.3cm,keepaspectratio]{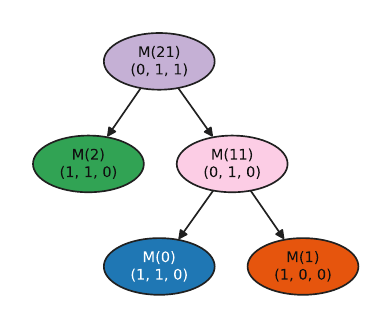}}
    \caption{Restricted Conley--Morse graph.}
  \end{subfigure}\hfill
  \begin{subfigure}[b]{0.5\linewidth}
    \centering
    \parbox[b][9.5cm][c]{\linewidth}{\centering\includegraphics[width=\linewidth,height=9.5cm,keepaspectratio]{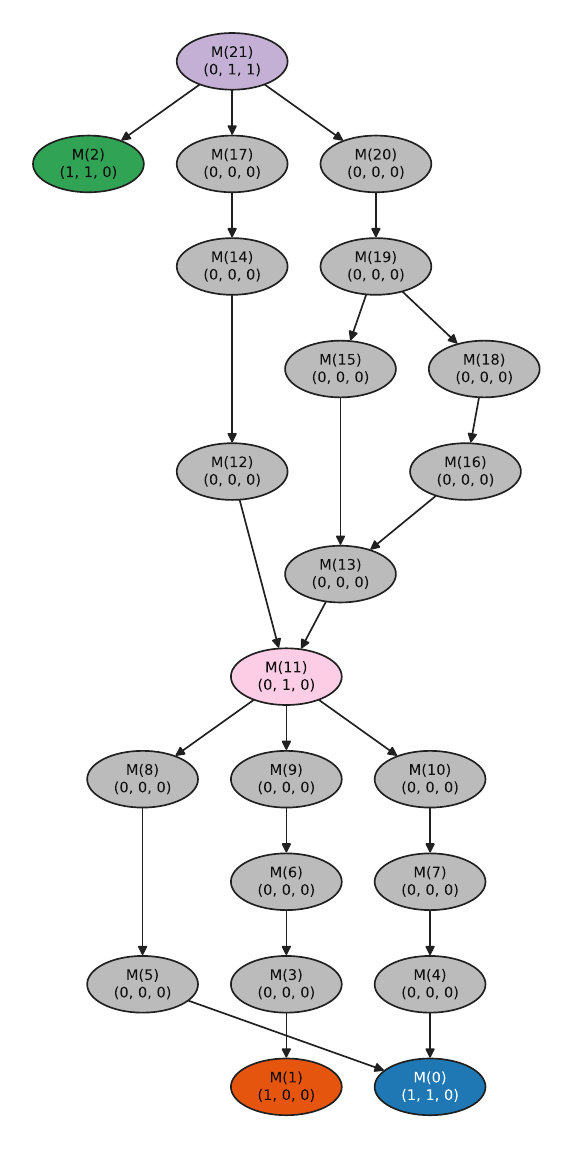}}
    \caption{Full Conley--Morse graph.}
  \end{subfigure}
  \caption{Impacting oscillator at $\tau=0.5$, on a $2^{11}\times2^{11}$ grid of $R=[-1.95,0.8]\times[-2.35,1.95]$.}
  \label{fig:impacting-oscillator}
\end{figure}

Figure~\ref{fig:impacting-oscillator} shows the computation for $\tau=0.5$. The Morse graph has $22$ Morse sets, and the five with nontrivial Conley index are $M(0)$, $M(1)$, $M(2)$, $M(11)$, and $M(21)$, with Conley indices $(1,1,0)$, $(1,0,0)$, $(1,1,0)$, $(0,1,0)$, and $(0,1,1)$, respectively. The restricted Conley--Morse graph orders them by $M(21)\to M(2)$, $M(21)\to M(11)$, $M(11)\to M(0)$, and $M(11)\to M(1)$, and the minimal Morse sets of $\MG(\cF)$ are $M(0)$, $M(1)$, and $M(2)$.

The Conley--Morse graph suggests the following dynamics. The minimal Morse sets $M(0)$, $M(1)$, and $M(2)$ correspond to attractors. The Conley index of $M(1)$ is that of an attracting equilibrium, and the Conley indices of $M(0)$ and $M(2)$ are that of an attracting periodic orbit. The Conley index of $M(11)$ is that of a hyperbolic equilibrium with one unstable direction, and the Conley index of $M(21)$ is that of a repelling periodic orbit. The restricted Conley--Morse graph then suggests connecting orbits from $M(21)$ to $M(2)$ and $M(11)$, and from the equilibrium in $M(11)$ to the attractors in $M(0)$ and $M(1)$.

By Birkhoff's representation theorem~\cite{Kalies:Mischaikow:Vandervorst:2022}, a finite lattice of attractors is isomorphic to the lattice of down-sets of its join-irreducible elements, which correspond to Morse sets. Accordingly, Figure~\ref{fig:impacting-oscillator-lattice} shows the lattice of down-sets of the restricted Conley--Morse graph, whose join-irreducible elements are the down-sets ${\downarrow}M:=\setdef{M'}{M'\leq M}$.

\begin{figure}[!htbp]
  \centering
  \includegraphics[width=0.6\linewidth]{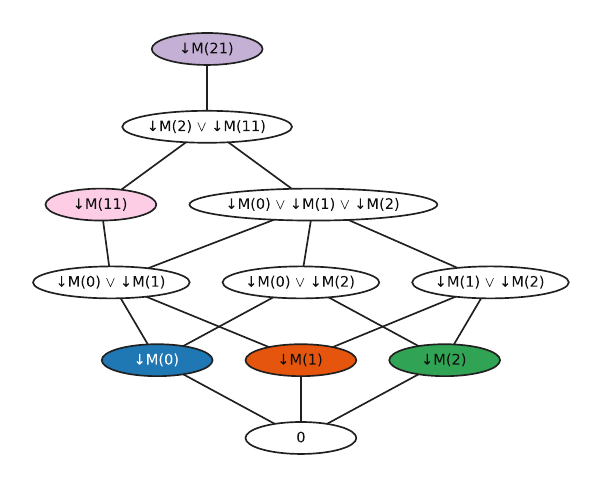}
  \caption{Lattice of attractors associated to the Morse sets with nontrivial Conley index.}
  \label{fig:impacting-oscillator-lattice}
\end{figure}

\FloatBarrier

\section*{Acknowledgments}
This work was supported by the Air Force Office of Scientific Research under award numbers FA9550-23-1-0400 and FA9550-23-1-0011.

\section*{Code availability}
The code used for the computations in Section~\ref{sec:example} is available at \url{https://github.com/bernardorivas/hybrid_boxmap}.

\newpage
\bibliographystyle{plain}
\bibliography{refs}

@article{Kalies:Mischaikow:Vandervorst:2016,
  author  = {Kalies, William D. and Mischaikow, Konstantin and Vandervorst, Robert C.A.M.},
  title   = {Lattice structures for attractors {II}},
  journal = {Foundations of Computational Mathematics},
  volume  = {16},
  number  = {5},
  pages   = {1151--1191},
  year    = {2016},
  doi     = {10.1007/s10208-015-9272-x}
}

@article{Rybakowski,
  author  = {Rybakowski, Krzysztof P.},
  title   = {On the Homotopy Index for Infinite-Dimensional Semiflows},
  journal = {Transactions of the American Mathematical Society},
  volume  = {269},
  number  = {2},
  pages   = {351--382},
  year    = {1982},
  doi     = {10.1090/S0002-9947-1982-0637695-7}
}

@article{Bonotto:Gimenes:Souto,
  author  = {Bonotto, Everaldo M. and Gimenes, Luciene P. and Souto, Ginnara M.},
  title   = {Asymptotically almost periodic motions in impulsive semidynamical systems},
  journal = {Topological Methods in Nonlinear Analysis},
  volume  = {49},
  number  = {1},
  pages   = {133--163},
  year    = {2017},
  doi     = {10.12775/TMNA.2016.065}
}

@article{Kvalheim:Gustafson:Koditschek,
  author  = {Matthew D. Kvalheim and Paul Gustafson and Daniel E. Koditschek},
  title   = {Conley's Fundamental Theorem for a Class of Hybrid Systems},
  journal = {SIAM Journal on Applied Dynamical Systems},
  volume  = {20},
  number  = {2},
  pages   = {784--825},
  year    = {2021},
  doi     = {10.1137/20M1336576}
}

@article{Kalies:Mischaikow:Vandervorst:2014,
  author  = {Kalies, William D. and Mischaikow, Konstantin and Vandervorst, Robert C.A.M.},
  title   = {Lattice structures for attractors {I}},
  journal = {Journal of Computational Dynamics},
  volume  = {1},
  number  = {2},
  pages   = {307--338},
  year    = {2014}
}

@article{Kalies:Mischaikow:Vandervorst:2022,
  author  = {Kalies, William D. and Mischaikow, Konstantin and Vandervorst, Robert C.A.M.},
  title   = {Lattice structures for attractors {III}},
  journal = {Journal of Dynamics and Differential Equations},
  volume  = {34},
  pages   = {1729--1768},
  year    = {2022},
  doi     = {10.1007/s10884-021-10056-8}
}

@misc{Rivas:Kalies,
  author = {Bernardo Rivas and William Kalies},
  title  = {Conley Index Theory for Hybrid Systems},
  year   = {2026},
  eprint = {2606.18501},
  archivePrefix = {arXiv},
  doi    = {10.48550/arXiv.2606.18501},
  url    = {https://arxiv.org/abs/2606.18501},
  note   = {arXiv:2606.18501v1}
}

@book{diBernardo:Budd:Champneys:Kowalczyk,
  author    = {di Bernardo, Mario and Budd, Christopher J. and Champneys, Alan R. and Kowalczyk, Piotr},
  title     = {Piecewise-Smooth Dynamical Systems: Theory and Applications},
  series    = {Applied Mathematical Sciences},
  volume    = {163},
  publisher = {Springer London},
  year      = {2008},
  doi       = {10.1007/978-1-84628-708-4}
}

@book{vanDerSchaft:Schumacher,
  author    = {van der Schaft, Arjan and Schumacher, Hans},
  title     = {An Introduction to Hybrid Dynamical Systems},
  series    = {Lecture Notes in Control and Information Sciences},
  volume    = {251},
  publisher = {Springer London},
  year      = {2000},
  doi       = {10.1007/BFb0109998}
}

@misc{CMGDB,
  author = {Gameiro, Marcio and Harker, Shaun and others},
  title  = {{CMGDB}: {Conley-Morse-Graph Database}},
  url    = {https://github.com/marciogameiro/cmgdb},
  year   = {2020}
}

@article{Franks:Richeson,
  author = {Franks, John and Richeson, David},
  title = {Shift Equivalence and the {Conley} Index},
  journal = {Transactions of the American Mathematical Society},
  year = {2000},
  volume = {352},
  number = {7},
  pages = {3305--3322},
  doi = {10.1090/S0002-9947-00-02488-0},
  url = {https://arxiv.org/abs/math/9910171}
}

@book{Goebel:Sanfelice:Teel,
  author    = {Goebel, Rafal and Sanfelice, Ricardo G. and Teel, Andrew R.},
  title     = {Hybrid Dynamical Systems: Modeling, Stability, and Robustness},
  publisher = {Princeton University Press},
  year      = {2012},
  doi       = {10.23943/princeton/9780691153896.001.0001},
  isbn      = {9780691153896}
}

@article{Johansson:Egerstedt:Lygeros:Sastry,
  author  = {Johansson, Karl Henrik and Egerstedt, Magnus and Lygeros, John and Sastry, Shankar},
  title   = {On the regularization of {Zeno} hybrid automata},
  journal = {Systems \& Control Letters},
  volume  = {38},
  number  = {3},
  pages   = {141--150},
  year    = {1999},
  doi     = {10.1016/S0167-6911(99)00059-6}
}

@article{Arai:Kalies:Kokubu:Mischaikow:Oka:Pilarczyk,
  author  = {Arai, Zin and Kalies, William and Kokubu, Hiroshi and Mischaikow, Konstantin and Oka, Hiroe and Pilarczyk, Pawe{\l}},
  title   = {A Database Schema for the Analysis of Global Dynamics of Multiparameter Systems},
  journal = {SIAM Journal on Applied Dynamical Systems},
  volume  = {8},
  number  = {3},
  pages   = {757--789},
  year    = {2009},
  doi     = {10.1137/080734935}
}

@inproceedings{Shia:Vasudevan:Bajcsy:Tedrake,
  author    = {Shia, Victor and Vasudevan, Ram and Bajcsy, Ruzena and Tedrake, Russ},
  title     = {Convex computation of the reachable set for controlled polynomial hybrid systems},
  booktitle = {53rd IEEE Conference on Decision and Control},
  pages     = {1499--1506},
  year      = {2014},
  doi       = {10.1109/CDC.2014.7039612}
}

@article{McGeer,
  author  = {McGeer, Tad},
  title   = {Passive Dynamic Walking},
  journal = {The International Journal of Robotics Research},
  volume  = {9},
  number  = {2},
  pages   = {62--82},
  year    = {1990},
  doi     = {10.1177/027836499000900206}
}

@book{Izhikevich:2007,
  author    = {Izhikevich, Eugene M.},
  title     = {Dynamical Systems in Neuroscience: The Geometry of Excitability and Bursting},
  series    = {Computational Neuroscience},
  publisher = {MIT Press},
  address   = {Cambridge, MA},
  year      = {2007},
  isbn      = {9780262090438},
  doi       = {10.7551/mitpress/2526.001.0001}
}

@incollection{Mischaikow:Mrozek,
  author    = {Mischaikow, Konstantin and Mrozek, Marian},
  title     = {Conley index},
  booktitle = {Handbook of Dynamical Systems},
  editor    = {Fiedler, Bernold},
  volume    = {2},
  pages     = {393--460},
  publisher = {Elsevier},
  year      = {2002},
  isbn      = {9780444501684},
  doi       = {10.1016/S1874-575X(02)80030-3}
}

@book{Kaczynski:Mischaikow:Mrozek,
  author    = {Kaczynski, Tomasz and Mischaikow, Konstantin and Mrozek, Marian},
  title     = {Computational Homology},
  series    = {Applied Mathematical Sciences},
  publisher = {Springer},
  address   = {New York},
  year      = {2004},
  doi       = {10.1007/b97315}
}

@misc{pychomp,
  author = {Gameiro, Marcio and Harker, Shaun},
  title  = {{pyCHomP2}: {CHomP} ({Computational Homology Project}) with {Python} bindings},
  note   = {Version 1.1.2},
  url    = {https://github.com/marciogameiro/pyCHomP2},
  year   = {2025}
}

\end{document}